\documentclass[preprint,11pt]{elsarticle}
\usepackage[margin=2cm]{geometry}
\usepackage{mathrsfs}
\usepackage{amsmath,amsfonts,amssymb,amsthm,amscd}
\usepackage[english]{babel}
\usepackage[utf8]{inputenc}
\usepackage[T1]{fontenc}
\usepackage{color}
\usepackage{graphicx}
\usepackage{comment}
\usepackage[dvipsnames]{xcolor}
\usepackage[colorlinks,linkcolor=blue,citecolor=blue]{hyperref}

\usepackage{soul}

\def\be{\begin{eqnarray}}
\def\ee{\end{eqnarray}}
\def\b*{\begin{eqnarray*}}
\def\e*{\end{eqnarray*}}
\def \1{{\bf 1}}

\def \Frac{\displaystyle\frac}

\def \E{\mathbb{E}}
\def \F{\mathbb{F}}

\def \N{\mathbb{N}}
\def \P{\mathbb{P}}
\def \R{\mathbb{R}}
\def \Q{\mathbb{Q}}

\def \Ec{\mathcal{E}}

\def\Hc{{\cal H}}
\def\Fc{{\cal F}}
\def\Tc{{\cal T}}

\def\Gc{{\cal G}}
\def\Sc{{\cal S}}

\def\eps{\varepsilon}

\newtheorem{Theorem}{Theorem}[section]
\newtheorem{Definition}{Definition}[section]
\newtheorem{Proposition}{Proposition}[section]

\newtheorem{Lemma}{Lemma}[section]
\newtheorem{Corollary}{Corollary}[section]
\newtheorem{Remark}{Remark}[section]
\newtheorem{Example}{Example}[section]

\begin{document}

\begin{frontmatter}
\title{Law-invariant BSDEs and dynamic risk measures: new characterizations}

\author[lemans]{Zakaria Bensaid}
\ead{zakaria.bensaid@univ-lemans.fr}

\author[ensae]{Roxana Dumitrescu\corref{cor1}}
\ead{roxana.dumitrescu@ensae.fr}

\author[lemans]{Anis Matoussi}
\ead{anis.matoussi@univ-lemans.fr}

\author[lemans]{Wissal Sabbagh}
\ead{wissal.sabbagh@univ-lemans.fr}

\cortext[cor1]{Corresponding author}

\address[lemans]{Le Mans Universit\'e, LMM, IRA, F-72000 Le Mans, France}

\address[ensae]{ENSAE Paris, CREST, Institut Polytechnique de Paris, Palaiseau, France}

\begin{abstract}
We provide a new characterization of law-invariant backward stochastic differential equations (i.e. BSDEs) with quadratic growth. This answers the open question raised in Xu–Xu–Zhou (2022) on necessary conditions for law-invariance of $g$-expectations, and extends the analysis to general (possibly non-deterministic) generators. We also introduce and compare several dynamic notions of law-invariance in continuous time, establishing precise relationships among them.

As an application, we study dynamic risk measures. For cash-additive, normalized risk measures, we recover and extend to continuous time the Kupper--Schachermayer (2009) characterization obtained in discrete time, showing that law-invariance and strong time-consistency force an entropic structure. We further obtain a new characterization of cash non-additive law-invariant risk measures generated by BSDEs via a time-dependent certainty equivalent representation. 
\end{abstract}

\end{frontmatter}
\section{Introduction}

In this paper, we investigate law-invariance for operators induced by backward stochastic differential equations (BSDEs) --- in particular Peng's conditional $g$-expectation \cite{CoquetHuMeminPeng2002} --- and dynamic convex risk measures (or dynamic concave utilities) \cite{CvitanicKaratzas1999DynamicRisk, DetlefsenScandolo2005, beatrice, DelbaenPengRosazza2010}.
Our main contributions are threefold. First, we introduce novel dynamic notions of law invariance in continuous time and establish precise relationships among them.
Second, we derive structural \emph{necessary} conditions for the law invariance of $g$-expectations under a Brownian filtration, thereby answering the open question raised in Xu--Xu--Zhou \cite{XuXuZhou2022}, and we identify a significantly broader class of law-invariant generators.
Third, we apply these results to dynamic risk measures: we recover and extend the Kupper--Schachermayer characterization \cite{KupperSchachermayer2009} in continuous time and exhibit a new explicit class of law-invariant, cash non-additive risk measures.

In a Brownian filtration, dynamic risk measures and $g$-expectations are closely related: this connection was first identified by Rosazza Gianin \cite{RosazzaGianin2006}, and later developed by Barrieu and El Karoui \cite{BarrieuElKaroui2009IndifferencePricingRiskMeasures} and Di Nunno and Rosazza Gianin \cite{NunnoGianin}, among others. A key property in the dynamic setting is time-consistency, which admits several formulations; see \cite{NunnoGianin} and the references therein. Using the strong formulation (recursivity), Bion-Nadal \cite{BionNadal2009} showed that a normalized, time-consistent dynamic convex risk measure satisfies a supermartingale property which, combined with the representation results of Delbaen et al.\ \cite{DelbaenPengRosazza2010} and the regularity results of Delbaen et al.\ \cite{DelbaenHuBao2011}, yields a one-to-one correspondence between normalized, cash-additive dynamic risk measures and quadratic $g$-expectations. The main tools are representation theorems for filtration-consistent nonlinear expectations (see \cite{CoquetHuMeminPeng2002, HU20081518, ZHENG2024104464} and the references therein). A similar correspondence for non-normalized or cash non-additive time-consistent risk measures is, to the best of our knowledge, still an open problem. Nevertheless, we aim to provide an explicit class of law-invariant, cash non-additive dynamic risk measures by characterizing law-invariant $g$-expectations under a Brownian filtration.

Informally, law-invariance means that the evaluation of a payoff depends only on its distribution under the reference probability measure~$\P$, and not on its particular realization. This property is central in finance and insurance whenever preferences or capital requirements are formulated in distributional terms \cite{WANG1997173, Dybvig98, XuXuZhou2022}. Classical examples include quantile-based criteria such as Value-at-Risk, tail-based criteria such as Expected Shortfall, and expected utility-based functionals such as shortfall risk measures \cite{RockafellarUryasevCVaR, FrittelliRosazzaGianin2002, FollmerSchied2002, BionNadal2008}. In the static setting, law-invariant risk measures admit deep structural representations (e.g.\ Kusuoka-type representations on atomless spaces); the literature is vast and includes \cite{Kusuoka2001, JouiniSchachermayerTouzi2006, FollmerSchiedBook, EkelandSchachermayer2011}.

Passing from static to dynamic evaluation is subtle. In discrete time, combining law-invariance with strong time consistency drastically reduces the admissible class: under cash additivity, normalization, and relevance, Kupper and Schachermayer \cite{KupperSchachermayer2009} showed that one is essentially led to an entropic-type structure. This raises two natural questions: to what extent is this restriction driven by law-invariance and time consistency alone, and to what extent does it rely on cash additivity and normalization? Furthermore, do the different dynamic notions of law-invariance that arise in continuous time lead to distinct structural conditions? The latter question does not appear in the discrete-time framework of \cite{KupperSchachermayer2009}, which does not distinguish between these notions.

In continuous time, BSDEs and $g$-expectations are time-consistent by construction (via the flow property), so it is natural to ask to what extent they can be reconciled with law-invariance. This question was first explored in the setting of discrete-time BSDEs on a binomial tree by Elliott, Siu, and Cohen \cite{cohen}. More recently, Xu, Xu, and Zhou \cite{XuXuZhou2022} formalized the notion of law-invariant $g$-expectations in continuous time and provided PDE-based \emph{sufficient} conditions for deterministic generators. However, such sufficiency results do not clarify whether law-invariance imposes genuine structural conditions on the generator, nor how the picture changes when the generator is allowed to be random; identifying necessary conditions has remained open and constitutes a central motivation for the present work.

On the technical side, the properties of $g$-expectations were established in the Lipschitz case by Peng \cite{Peng}, extended to the uniformly continuous case by Jia \cite{JIA20102241}, to the quadratic growth case by Ma and Yao \cite{Ma02062010}, and further generalized by Zheng \cite{ZHENG2024104464}. In addition, it is not clear whether the static definition of law-invariance (as used in \cite{KupperSchachermayer2009, XuXuZhou2022}) is sufficiently general in continuous time; we therefore extend the notions developed in \cite{weber2006distribution, beatrice, cohen, de2023convex}. Our study also contributes to the literature on representation properties of generators that can be inferred from BSDE solutions \cite{Briand2000}.
\paragraph{Main contributions}
We first introduce new dynamic definitions of law-invariance in continuous time (see Definitions~\ref{def:cli}, \ref{def:oscli}, and \ref{def:mcli}) and analyze the relationships among these four notions (see Remark~\ref{rmk:defs} and Theorem~\ref{thm:const}).

We then derive necessary conditions by distinguishing three regimes according to the randomness of the generator at $z=0$, namely the dependence of $g(t,\omega,y,0)$ on $(t,\omega)$, and we address each case using different techniques.

\begin{itemize}
\item In the \textit{first regime}, assuming $g(t,\omega,y,0)=0$, we obtain a complete characterization: law-invariance holds if and only if $g$ is deterministic and purely quadratic in $z$ (Theorem~\ref{thm:main} and Corollary~\ref{cor:mcli}), i.e.,
$$
g(t,\omega,y,z) = k(y)\lvert z\rvert^2.
$$
Under this condition, the four notions of law-invariance are equivalent.

\item In the \textit{second regime}, where $g(t,\cdot,y,0)=h(t,y)$ is nonzero and deterministic, we reduce the problem to the first regime via a deterministic change of variables and recover, once again, a structural restriction: law-invariance holds if and only if $g$ is deterministic $dt\otimes d\P$-a.e.\ and of the form (Theorem~\ref{thm:mainB})
$$
g(t,\omega,y,z) = h(t,y) + f(t,y)\lvert z\rvert^2,
$$
with $f$ solving the PDE
$$
\partial_t f(t,y) - h(t,y)\partial_y f(t,y) - \partial_y h(t,y) f(t,y) = \tfrac{1}{2}\partial_{yy} h(t,y),
$$
which coincides with the condition proposed in \cite{XuXuZhou2022} as sufficient. Our result closes this gap by showing that this PDE constraint is also necessary.

\item Finally, in the \textit{general regime} where $g(\cdot,y,0)$ is non-deterministic, we derive a sharp pathwise necessary condition by differentiating the BSDE, thereby restricting the class of law-invariant generators. Under an additional structural assumption, this condition yields an explicit quadratic form (Theorems~\ref{cor:litransinv} and~\ref{cor:CLLItransinv})
$$
g(t,\omega,z) = g(t,\omega,0) + \beta \lvert z\rvert^2.
$$
\end{itemize}
The proof of the third regime relies on a Gâteaux differentiability result for the $g$-expectation with respect to the terminal condition (Theorem~\ref{thm:diff} in the Appendix). Differentiability results for quadratic BSDEs have been established by Ankirchner, Imkeller, and Dos Reis \cite{ankirchner2007classical}, and more recently by Imkeller, Likibi Pellat, and Menoukeu-Pamen \cite{imkeller2024differentiability}. These works, however, impose conditions on $\partial_y g$ that are either bounded or of sublinear growth in $|z|$ (typically $|z|^\alpha$ with $\alpha < 1$), which exclude the full range $\alpha \in [1,2)$ allowed under our Assumption~\textbf{(A4$^*$)}. In our setting, the differentiability result is applied before the structural properties of the generator are known, and thus it is not possible to rely on assumptions that already encode a specific functional form of $g$.

Moreover, the arguments in \cite{ankirchner2007classical} rely on a dominated convergence step (see also \cite[Chapter~3]{dos2010some}), which requires uniform integrability of products involving $\partial_z g$ and the linearized solution. In the quadratic growth setting, $\partial_z g$ may grow linearly in $z$, and the available estimates do not immediately provide the required domination in our framework. Rather than extending these arguments under additional structural restrictions, we provide a self-contained and original proof under Assumptions~\textbf{(A1)--(A4$^*$)}, which relies on a combination of uniform BMO energy estimates and Vitali’s convergence theorem.

As a byproduct of our characterization results, we derive new results for dynamic risk measures, that we split in two classes. For cash-additive, normalized risk measures, Theorem~\ref{thm:risk} recovers and extends the Kupper--Schachermayer result in continuous time: law-invariance and time-consistency, together with natural regularity conditions, force the risk measure to be entropic. As a consequence, we show that a dynamic shortfall risk measure is time-consistent if and only if its loss function is linear or exponential (Theorem~\ref{cor:shortfall}).

For cash non-additive risk measures, Theorem~\ref{thm:cash_nonadd_charac} provides a complete characterization under Assumption~\textbf{(H)}: every law-invariant risk measure generated by a BSDE admits a time-dependent certainty equivalent representation
$$
\rho_{t,T}(X) = \Phi^{-1}\!\big(t,\, \E[\Phi(T,X)\,|\,\Fc_t]\big),
$$
where $\Phi(s,\cdot)$ is a strictly monotone deterministic function determined by the generator. This extends the framework of cash non-additive risk measures introduced in \cite{di2024cash} and complements the recent generalized shortfall approach of \cite{di2026capturing}. In particular, the structural restriction observed in the cash-additive setting does not persist: once cash-additivity is relaxed, a rich class of non-entropic law-invariant risk measures emerges, including the generalized $q$-entropic risk measures of \cite{ma2021generalized, di2024cash}.

We work in the quadratic BSDE framework initiated by Kobylanski \cite{Kobylanski2000}. Concerning bounded terminal conditions, establishing necessary conditions on $L^\infty(\Fc_T)$ is strictly stronger than on larger domains (e.g.\ $L^2(\Fc_T)$ or exponentially integrable random variables), since law-invariance on a larger space implies law-invariance on its bounded subspace; hence any structural restriction derived from bounded payoffs remains valid \emph{a fortiori} for any extension of the operator to an unbounded domain.
\paragraph{Organization}
Section~\ref{sec:prelim} presents the notation and recalls the well-posedness theory for quadratic BSDEs and the properties of $g$-expectations.
Section~\ref{sec:LI} introduces the four notions of law-invariance, establishes their equivalence, and presents the characterization theorems in the three regimes.
Section~\ref{sec:risk} applies these results to cash-additive and cash non-additive dynamic risk measures.
Appendix~\hyperref[app:A]{A} contains the technical results on quadratic BSDEs, including the differentiability theorem.
Appendices~\hyperref[app:B]{B}--\hyperref[app:D]{D} gather results on conditional atomlessness, invariance properties of Brownian motion, and the change-of-variables construction.

\section{Preliminaries on quadratic g-expectations} \label{sec:prelim}
\paragraph{Notation}
Fix $T>0$. Let $(\Omega,\Fc,\P)$ support a $d$--dimensional Brownian motion
$W=(W_t)_{0\le t\le T}$ with augmented filtration $\F=(\Fc_t)_{0\le t\le T}$.
Unless otherwise stated, equalities and inequalities between processes (or random fields) hold
$dt\otimes d\P$--a.e. We identify the Euclidean scalar product on $\R^d$ with
multiplication: for $z,z'\in\R^d$ we write $zz':=\langle z,z'\rangle$ and
$|z|^2:=zz$. For $\Gc$ a sub-$\sigma$-field of $\Fc$, $X\stackrel{d}{=}Y$ denotes equality in law, and when needed we write
$Law(X | \Gc)=Law(Y | \Gc)$, $\P$-a.s., for almost sure equality of conditional laws. 
Consider a probability measure $\Q$ on $(\Omega,\Fc)$. We use the following spaces and conventions:

\begin{itemize}\setlength\itemsep{10pt}
\item $L^\infty(\Gc)$ denotes the space of essentially
bounded $\Gc$--measurable random variables endowed with the norm
$$
\|X\|_{L^\infty(\Gc)}:=\mathrm{ess\,sup}_{\omega\in\Omega}|X(\omega)|,
$$
and $L_+^{\infty}(\Gc)$ denote the subset of $L^\infty(\Gc)$ of non-negative random variables.
\item For $p\in[1,\infty)$, $L^{p}(\Q)$ denotes the usual space of $\Fc$--measurable random variables with finite $p$--moment under $\Q$, endowed with the norm
$$
\|X\|_{L^{p}(\Q)}:=\Big(\E^{\Q}[|X|^{p}]\Big)^{1/p}.
$$
We write $L^{p}(\Gc,\Q)$ for the subspace of $L^{p}(\Q)$ consisting of $\Gc$--measurable random variables. 
\item For $t\in[0,T]$, $\Tc_{t,T}$ denotes the set of $\F$--stopping times taking values in $[t,T]$. If $\tau\in\Tc_{t,T}$, then $\Tc_{t,\tau}$ denotes the set of $\F$--stopping times $\sigma$ such that $\sigma\in\Tc_{t,T}$ and $\sigma\le \tau$ $\P$-a.s.

\item For $p\ge1$, $\Sc^{p}(\Q)$ denotes the space of $\F$--progressively measurable c\`adl\`ag processes $Y$ such that the norm
$$
\|Y\|_{\Sc^{p}(\Q)}:=\Big(\E^{\Q}\big[\sup_{0\le s\le T}|Y_s|^{p}\big]\Big)^{1/p}<\infty.
$$
We write $\Sc^{\infty}$ for the corresponding essential supremum space, i.e.\ the space of $\F$--progressively measurable c\`adl\`ag processes $Y$ such that
$$
\|Y\|_{\Sc^{\infty}}:=\Big\|\sup_{0\le s\le T}|Y_s|\Big\|_{L^{\infty}(\Fc_T)}<\infty.
$$

\item For $p\ge1$, $\Hc^{p}(\Q)$ denotes the space of $\F$--predictable processes $Z$ such that
$$
\|Z\|_{\Hc^{p}(\Q)}:=\Big(\E^{\Q}\Big[\Big(\int_0^T|Z_s|^2\,ds\Big)^{p/2}\Big]\Big)^{1/p}<\infty.
$$
When no measure is specified, we write $\Sc^{p}$ and $\Hc^{p}$ as shorthand for $\Sc^{p}(\P)$ and $\Hc^{p}(\P)$, respectively. For a subinterval $[t,s]\subseteq[0,T]$, we write $\Hc^{p}_{[t,s]}$ for the analogous space with the integration over $[t,s]$, i.e.\ $\|Z\|_{\Hc^{p}_{[t,s]}}:=\bigl(\E\bigl[(\int_t^s|Z_r|^2\,dr)^{p/2}\bigr]\bigr)^{1/p}$.

\item $\mathrm{BMO}$ denotes the space of continuous martingales $M$ on $[0,T]$ with $M_0=0$ such that the BMO norm
$$
\|M\|_{\mathrm{BMO}}^{2}
:=\sup_{\tau\in\Tc_{0,T}}
\Big\|\E\big[\langle M\rangle_T-\langle M\rangle_{\tau}\ \big|\ \Fc_{\tau}\big]\Big\|_{L^{\infty}(\Fc_T)}<\infty.
$$
 If $\theta\in\Hc^{2}$ is predictable, we write $\theta\in\mathrm{BMO}$ whenever the martingale $M^\theta:=\int_0^{\cdot}\theta_s\,dW_s$ belongs to $\mathrm{BMO}$, and we then set
$$
\|\theta\|_{\mathrm{BMO}}
:=\Big\|\int_0^{\cdot}\theta_s\,dW_s\Big\|_{\mathrm{BMO}}
=\sup_{\tau\in\Tc_{0,T}}
\Big\|\E\Big[\int_{\tau}^{T}|\theta_s|^2\,ds\ \Big|\ \Fc_{\tau}\Big]\Big\|_{L^{\infty}(\Fc_T)}^{1/2}.
$$
\item For any interval $I \subset \R$ and $p>0$, $L^{p}(I)$ denotes the space of measurable functions $u: I \rightarrow \R$ such that
$$\|u\|_{L^p(I)}:= \left(\int_a^b \lvert u(t) \lvert^p dt\right)^\frac{1}{p} < \infty.$$

\item We write $C_b^2(\R)$ for the space of twice continuously differentiable functions $\psi:\R\to\R$ such that
$\psi,\psi'$ and $\psi''$ are bounded.
For maps $\varphi:[0,T]\times\R\to\R$ and integers $k,\ell\ge0$, we write $\varphi\in C^{k,\ell}([0,T]\times\R)$ if
\begin{itemize}\setlength\itemsep{2pt}
\item for each $j=0,\dots,k$ and $m=0,\dots,\ell$, the mixed partial derivative $\partial_{t^j y^m}\varphi$
exists (in the classical sense) on $[0,T]\times\R$ whenever $j+m\ge1$;
\item all derivatives $\partial_{t^j y^m}\varphi$ with $0\le j\le k$ and $0\le m\le \ell$ are continuous on
$[0,T]\times\R$.
\end{itemize}
\end{itemize}

We start this section by reminding the well-posedness results for a quadratic BSDE established in \cite{Kobylanski2000}.
For $X \in L^{\infty}(\Fc_T)$ consider
\begin{equation}\label{eq:BSDE}
Y_t = X + \int_t^T g(s, Y_s, Z_s)ds - \int_t^T Z_s dW_s,
\qquad \forall t\in[0,T].
\end{equation}
Throughout the rest of the paper, we denote \eqref{eq:BSDE} as BSDE($T, X, g$) and assume that the generator $g$ takes the form:
$$
 g(t,\omega,y,z)=g_1(t,\omega,y,z)y+g_2(t,\omega,y,z), \quad ~~ \forall (t,\omega,y,z) \in
 [0,T] \times \Omega  \times \R \times \R^d,
$$
and satisfies the following \textit{Assumptions}:
 \begin{itemize}
 \item[{\bf (A1)}] Both $g_1$ and $g_2$ are progressively measurable
 and both $g_1(t,\omega,\cdot,\cdot)$ and $g_2(t,\omega,\cdot,\cdot)$ are continuous uniformly in $(t,\omega) \in
 [0,T] \times \Omega$.
 \item[{\bf (A2)}] There exist a constant $\kappa>0$ and an increasing  continuous function $\ell:
 \R^+ \to \R^+$, such that for $dt \otimes d\P$-a.e. $(t,\omega) \in
 [0,T] \times \Omega$, $$ |g_1(t,\omega,y,z)| \leq \kappa \quad \mbox{and} \quad
  |g_2(t,\omega,y,z)| \leq \kappa +\ell(|y|)|z|^2,\quad ~~ (y,z) \in \R \times
 \R^d. $$
 \item[{\bf (A3)}] With the same increasing continuous function $\ell$, for $dt \otimes d\P$-a.e.
 $(t,\omega) \in [0,T] \times \Omega$,
 $$
 \Big| \partial_z g (t,\omega,y,z)\Big| \leq
\ell(|y|)(1+|z|),\quad \quad (y,z) \in \R \times \R^d.
 $$
\item[{\bf (A4)}] For any $ \eps > 0$, there exists a positive function $h_{\eps}(t) \in L^1\left([0,T] \right)$
such that for $dt \otimes d\P$-a.e. $(t,\omega) \in [0,T] \times \Omega$,
$$ \partial_y g(t,\omega,y,z) \leq h_{\eps}(t)+\eps|z|^2,\quad \quad  (y,z)
\in \R \times \R^d. $$
\end{itemize}
 \begin{Remark}
Whenever there is no ambiguity, we suppress the $\omega$--dependence and write
$g(t,y,z)$ for $g(t,\omega,y,z)$. When we need to emphasize randomness we write
$g(t,\omega,y,z)$ explicitly, and we call $g$ deterministic if $g(t,\omega,y,z)$ is independent of $\omega$ $dt\otimes d\P$--a.e. Similarly, when $g$ does not depend on $y$ we write $g(t,\omega,z)$ or simply $g(t,z)$.
 \end{Remark}

Under \textit{Assumptions} \textbf{(A1)}–\textbf{(A4)} for any $X \in L^{\infty}(\Fc_T)$,
\eqref{eq:BSDE} admits a unique solution $(Y,Z)\in\Sc^{\infty}\times\Hc^2$ (see~\cite{Kobylanski2000}). Furthermore, we present the following assumptions to guarantee the differentiability result studied in the appendix (see Theorem \ref{thm:diff}):

\smallskip

\emph{Assumption} \textbf{(A4$^{*}$)}: (a stronger version of \textbf{(A4)}) For $dt\otimes d\P$-a.e. $(t,\omega)\in[0,T]\times\Omega$, the map
$(y,z)\mapsto g(t,\omega,y,z)$ is of class $C^1(\R\times\R^d)$ i.e. both
$\partial_y g(t,\omega,\cdot,\cdot)$ and $\partial_z g(t,\omega,\cdot,\cdot)$
exist and are continuous uniformly in $(t,\omega)$. Moreover, $\partial_y g$ satisfies: For any $\eps > 0$, there exists a positive function $h_{\eps} \in L^1\left([0,T] \right)$
such that for $dt \otimes d\P$-a.e. $(t,\omega) \in [0,T] \times \Omega$,
$$ \Big \lvert \partial_y g (t,\omega,y,z) \Big \lvert \leq h_{\eps}(t)+\eps|z|^2,\quad \quad  (y,z)
\in \R \times \R^d. $$

 We now recall the notion of nonlinear $g$-expectation introduced by Peng (\cite{Peng}).
For any $\sigma$, $\tau \in \Tc_{0,T}$ such that $\sigma \leq \tau$, $\P$-a.s., the quadratic $g$-expectation is the operator $$\Ec^g_{\sigma,\tau}:
L^{\infty}(\Fc_{\tau}) \to L^{\infty}(\Fc_{\sigma})$$ defined by $\Ec^g_{\sigma,\tau}[X]
:= Y_\sigma$, where $X \in L^{\infty}(\Fc_{\tau})$, and $(Y,Z)$ is the solution of BSDE($\tau, X, g$).
In particular, if $\tau=T$, we denote the quadratic $g$-expectation of $X$ given $\Fc_{\sigma}$ by $\Ec^g[X|\Fc_{\sigma}] := \Ec^g_{\sigma,T}[X]$.
We recall the following properties of the $g$-expectation operator, which play an important role in the proof of the main result.

\begin{Proposition} \label{prop:Eg}
Let Assumptions \textbf{(A1)}–\textbf{(A4)} hold, then $\Ec^g$ satisfies, for any $\rho$, $\sigma$,
$\tau \in \Tc_{0,T}$ with $\rho \le \sigma \le \tau$, $\P$-a.s. :
\begin{itemize}    
\item[\(1.\)] {\it Time-Consistency:}
 $$
\Ec^g_{\rho,\sigma}\big[\Ec^g_{\sigma,\tau}[X]\big]=\Ec^g_{\rho,\tau}[X],
\quad ~~ \P\text{-a.s.} \quad \forall X \in L^{\infty}(\Fc_{\tau}).
$$

\item[\(2.\)] {\it Constant-Preserving:}~~Assume $g(t,\omega, y,0) = 0$,  for $dt \otimes d\P$-a.e.
 $(t,\omega) \in [0,T] \times \Omega$, for all $y \in \R$, then
$$\Ec^g_{\sigma,\tau}[X]=X, \quad \P\text{-a.s.}, ~~ \forall X \in
L^{\infty}(\Fc_{\sigma}).$$

\item[\(3.\)] {\it Zero-one Law:}~~For any $X, X^{\prime} \in L^{\infty}(\Fc_{\tau})$
and $A \in \Fc_{\sigma}$, we have $$\Ec^g_{\sigma,\tau}[\1_A X + \1_{A^c}X^{\prime} ]=\1_A
\Ec^g_{\sigma,\tau}[X] + \1_{A^c}\Ec^g_{\sigma,\tau}[X^{\prime}],~~\P\text{-a.s.}$$ 
Moreover, if $g(t,0,0)=0$, $dt \otimes d\P$-a.e.,
then $$\Ec^g_{\sigma,\tau}[\1_A X]=\1_A \Ec^g_{\sigma,\tau}[X], \quad \P\text{-a.s}.$$

\item[\(4.\)] {\it Translation Invariance :}~~If $g$ is independent of $y$, then
$$
 \Ec^g_{\sigma,\tau}[X+X^{\prime}]=\Ec^g_{\sigma,\tau}[X]+X^{\prime},\quad \P\text{-a.s.} \quad
\forall X \in L^{\infty}(\Fc_{\tau}),\quad X^{\prime} \in L^{\infty}(\Fc_{\sigma}).
$$

\item[\(5.\)] {\it Strict Monotonicity:}~~For any $X, X^{\prime} \in L^{\infty}(\Fc_{\tau})$ with $X \geq X^{\prime}$,
$\P$-a.s., we have $\Ec^g_{\sigma,\tau}[X] \geq \Ec^g_{\sigma,\tau}[X^{\prime}]$,
$\P$-a.s.; Moreover, if $\Ec^g_{\sigma,\tau}[X] = \Ec^g_{\sigma,\tau}[X^{\prime}]$,
$\P$-a.s., then $X=X^{\prime}$, $\P$-a.s.
 \end{itemize}
 \end{Proposition}
\begin{Remark}
    As stated in \cite{Ma02062010}, the proofs rely exclusively on the uniqueness of the BSDE \eqref{eq:BSDE} and the comparison principle \cite{Kobylanski2000}. 
\end{Remark}
\section{Law-invariant BSDEs} \label{sec:LI}
From now on, we use the notation of $g$-expectations to designate the solution of a BSDE given a terminal time $T$ and a driver $g$.
\subsection{Static and dynamic notions of law-invariance for BSDEs}
We provide four notions of law-invariance for $g$-expectations defined through nonlinear BSDEs.
The first one (LI) appears in the continuous-time BSDE literature, see \cite{XuXuZhou2022}.
The second one (CLI) is the natural conditional and dynamic version at a fixed maturity $T$.
The third notion (CLLI) is a continuous-time analogue of the \emph{local} (or \emph{one-step}) law-invariance studied for time-consistent dynamic risk measures in discrete time; see, e.g., Cohen and coauthors for the conditional local law-invariance \cite{cohen},
and the surveys \cite{beatrice, de2023convex}.
In continuous time, ``one-step'' must be interpreted through \emph{short horizons} and thus requires varying the
deterministic terminal time $T^{\prime}$, possibly, close to $\sigma$. 

Finally, the fourth notion (MLI) is new in continuous time and is tailored to achieve law-invariance when the
\emph{maturity itself varies}: it compares $\Ec^g_{\sigma,\tau}[X]$ and $\Ec^g_{\sigma,\tau'}[X']$ for
$\sigma\le\tau\le\tau'$, thereby extending the discrete-time local viewpoint to random, time-varying maturities.

\begin{Definition}[LI] \label{def:li}
    A generator $g$ of a BSDE is law-invariant if it verifies for all $ X, X^{\prime}  \in L^{\infty}(\Fc_{T})$
    $$ X  \stackrel{d}{=} X^{\prime} \ \Rightarrow ~~ \Ec_{0,T}^g \left[ X \right] = \Ec_{0,T}^g\left[X^{\prime}\right].$$
    \end{Definition}
    \begin{Definition}[CLI] \label{def:cli}
    A generator $g$ of a BSDE is conditional law-invariant if it verifies for all $\sigma \in \Tc_{0,T}$ and $ X, X^{\prime}  \in L^{\infty}(\Fc_{T})$
    $$ Law\left( X  | \Fc_{\sigma}\right) =  Law\left( X^{\prime}  | \Fc_{\sigma}\right),~\P\text{-a.s.}  \Rightarrow ~~ \Ec_{\sigma,T}^g \left[ X \right] = \Ec_{\sigma,T}^g\left[X^{\prime}\right], ~ \P \text{-a.s.}$$
    \end{Definition}
    \begin{Definition}[CLLI] \label{def:oscli}
    A generator $g$ of a BSDE is conditional locally law-invariant if it verifies for all $T^{\prime} \in [0,T]$ and $\sigma\in \Tc_{0,T}$ such that $\sigma \leq T^{\prime}$, $\P$-a.s. and $ X, X^{\prime}  \in L^{\infty}(\Fc_{T^{\prime}})$
    $$ Law\left( X  | \Fc_{\sigma}\right) =  Law\left( X^{\prime}  | \Fc_{\sigma}\right),~\P\text{-a.s.}  \Rightarrow ~~ \Ec_{\sigma,T^{\prime}}^g \left[ X \right] = \Ec_{\sigma,T^{\prime}}^g\left[X^{\prime}\right], ~ \P \text{-a.s.}$$
    \end{Definition}
\begin{Definition}[MLI] \label{def:mcli}
    A generator $g$ of a BSDE is maturity law-invariant if it verifies for all $\sigma, \tau, \tau^{\prime}  \in \Tc_{0,T}$ such that $\sigma \leq \tau \leq \tau^{\prime}$, $\P$-a.s. and $ (X, X^{\prime})  \in L^{\infty}(\Fc_{\tau}) \times L^{\infty}(\Fc_{\tau^{\prime}}) $
    $$ Law\left( X  | \Fc_{\sigma}\right) =  Law\left( X^{\prime}  | \Fc_{\sigma}\right),~\P\text{-a.s.}  \Rightarrow ~~ \Ec_{\sigma,\tau}^g \left[ X \right] = \Ec_{\sigma,\tau^{\prime}}^g\left[X^{\prime}\right], ~ \P \text{-a.s.}$$
    \end{Definition}

\begin{Remark}\label{rmk:defs}\begin{itemize}
    \item 
    The discrete-time ``one-step'' viewpoint of \cite{cohen,beatrice} is inherently local in time. Allowing a \emph{random} terminal time in CLLI would make $\Ec^g_{\sigma,\tau}[X]$ sensitive to the coupling
between $X$ and $\tau$, which is not contained in $Law(X | \Fc_\sigma)$ alone; hence CLLI is meant as a local property at a \emph{prescribed} maturity $T^{\prime}$. We provide an example below that explicits the issue with a random terminal time.
    \item It is worth noting that 
    $$ MLI \implies CLLI \implies CLI \implies LI.$$
    \item The maturity law-invariance property implies the Constant-Preserving property. Let $\sigma, \tau \in \Tc_{0,T}$ such that $\sigma \leq \tau$ $\P$-a.s., then by MLI if $X \in L^{\infty}(\Fc_{\sigma}) \subset L^{\infty}(\Fc_{\tau})$ we have
    $$ \Ec^g_{\sigma, \tau} \left[ X \right] = \Ec^g_{\sigma, \sigma} \left[ X \right] = X, ~~ \P\text{-a.s.}$$
    \item To be more precise in the definition above. $Law\left( X  | \Fc_{\sigma}\right) =  Law\left( X^{\prime}  | \Fc_{\sigma}\right),~\P\text{-a.s.} $, means that for any $O \subset \R$ a Borel set, we have 
    $$ \P \left ( X \in O  | \Fc_{\sigma} \right) = \P \left ( X^{\prime} \in O  | \Fc_{\sigma} \right), ~\P\text{-a.s.} $$
    Note furthermore that it is equivalent to the following statement : For any $A \in \Fc_{\sigma}$, 
    $$ X \1_A  \stackrel{d}{=} X^{\prime} \1_A.$$
    \end{itemize}
\end{Remark}
\begin{Example}[\textit{Deterministic} vs. \textit{random} terminal times in CLLI]
Let $a\neq 0$ and consider the linear generator $g(t,\omega,y,z)=a\,y$.
For deterministic $0\le s<T^{\prime}\le T$, the associated $g$-expectation is explicit:
$\Ec^g_{s,T^{\prime}}[X]=e^{a(T^{\prime}-s)}\E[X|\Fc_s]$, hence it satisfies CLLI on deterministic horizons.

Let $0<t_1<t_2\le T$ and fix $A:= \{W_{t_1} \geq 0\} \in\Fc_{t_1}$. Set the finite-valued stopping time
$\tau=t_1\1_A+t_2\1_{A^c}$. For $X\in L^\infty(\Fc_\tau)$, the BSDE yields $\Ec^g_{t,\tau}[X]=e^{-at}\E[e^{a \tau}X | \Fc_t]$, which implies that $\Ec^g_{0,\tau}[X]=\E[e^{a\tau}X]$.
Choose $C:= A^c\in\Fc_{t_1}$ with $\P(C)=\P(A)$ but $\P(A\cap C) = 0$ and set $X=\1_A$, $X'=\1_C$.
Then $X\stackrel{d}{=}X'$, but
$\Ec^g_{0,\tau}[X]=e^{at_1}\P(A)\neq e^{at_2}\P(A) = \E[e^{a\tau}\1_C]=\Ec^g_{0,\tau}[X']$.
This shows that allowing a random terminal time in the one-step condition would make the value depend on the
coupling between $X$ and $\tau$, which is not captured by the law of $X$ alone but rather the joint law of $X$ and $\tau$. Note that, in the case of a random terminal time, the $g$-expectations depend on the joint law of $(X,\tau)$. 
\end{Example}

In the following Theorem, we show first that LI and CLI are actually equivalent and that all definitions are equivalent under the following sharp condition: $g(t,\omega, y,0) = 0$, for $dt \otimes d\P$-a.e. $(t,\omega) \in [0,T] \times \Omega$ and for all $y \in \R$ (see Corollary \ref{cor:mcli}).
\begin{Theorem}[Equivalence between LI, CLI, CLLI, and MLI definitions]\label{thm:const}
            Assume \textbf{(A1)}–\textbf{(A4)} hold. The following are equivalent :
 \begin{enumerate}
     \item[i)] The generator $g$ is LI;
     \item[ii)] The generator $g$ is CLI.
\end{enumerate}
In addition, if $g(t,\omega, y,0) = 0$, for $dt \otimes d\P$-a.e.
 $(t,\omega) \in [0,T] \times \Omega$ and for all $y \in \R$, then i) and ii) are both equivalent to
 \begin{enumerate}
     \item[iii)] The generator $g$ is CLLI;
     \item[iv)] The generator $g$ is MLI.
 \end{enumerate}
\end{Theorem}
\begin{proof}
Recall that, by Remark \ref{rmk:defs}, we have CLI $\Rightarrow LI$. To prove the equivalence between the LI and CLI properties, it remains to show that LI $\Rightarrow$ CLI. Fix $0\le \sigma\le T$ and $X,X^{\prime}\in L^\infty(\Fc_T)$ such that
 $$Law\left( X  | \Fc_{\sigma}\right) =  Law\left( X^{\prime}  | \Fc_{\sigma}\right),~\P\text{-a.s.}$$  
Then, for any $A\in\Fc_\sigma$, we have that $$ X \1_A  \stackrel{d}{=} X^{\prime} \1_A.  $$
Hence, by law-invariance of the $g$-expectation, we obtain
$$ \Ec^g_{0, T}[\1_A X] = \Ec^g_{0, T}[\1_A X^{\prime}] ,\qquad \forall\,A\in\Fc_\sigma.
$$
Furthermore, by Time-Consistency, and the Zero-one Law property of the $g$-expectation, we have both $X$ and $X^{\prime}$
$$
\Ec^g_{0, T}[\1_A X]
=\Ec^g_{0,\sigma}\!\Big[\Ec^g_{\sigma,T}[\1_A X]\Big]
=\Ec^g_{0,\sigma}\!\Big[\1_A\,\Ec^g_{\sigma,T}[X] + \1_{A^c}\,\Ec^g_{\sigma,T}[0]\Big],$$
and
$$
\Ec^g_{0, T}[\1_A X^{\prime}]
=\Ec^g_{0,\sigma}\!\Big[\Ec^g_{\sigma,T}[\1_A X^{\prime}]\Big]
=\Ec^g_{0,\sigma}\!\Big[\1_A\,\Ec^g_{\sigma,T}[X^{\prime}] + \1_{A^c}\,\Ec^g_{\sigma,T}[0]\Big].$$
Hence, for all $A\in\Fc_\sigma$,
$$
\Ec^g_{0,\sigma}\!\big[\1_A Y + \1_{A^c}\,\Ec^g_{\sigma,T}[0] \big]=\Ec^g_{0,\sigma}\!\big[\1_A Y^{\prime} + \1_{A^c}\,\Ec^g_{\sigma,T}[0] \big],$$
where 
$$Y:=\Ec^g_{\sigma,T}[X]\in L^\infty(\Fc_\sigma), \quad \text{ and } \quad Y^{\prime}:=\Ec^g_{\sigma,T}[X^{\prime}]\in L^\infty(\Fc_\sigma).$$

Let $D:=\{Y>Y^{\prime}\}\in\Fc_\sigma$. If $\P(D)>0$, then $\1_D Y\ge \1_D Y^{\prime}$ and $\P(\1_D Y>\1_D Y^{\prime})>0$, so by strict monotonicity of $\Ec^g_{0,\sigma}$,
$$
\Ec^g_{0,\sigma}[\1_D Y + \1_{D^c}\,\Ec^g_{\sigma,T}[0] ]>\Ec^g_{0,\sigma}[\1_D Y^{\prime} + \1_{D^c}\,\Ec^g_{\sigma,T}[0]],
$$
contradicting the equality above. Thus $\P(D)=0$. Applying the same argument to $D':=\{Y^{\prime}>Y\}$ yields $\P(D')=0$, hence $Y=Y^{\prime}$ a.s., i.e.
$$
\Ec^g_{\sigma,T}[X]=\Ec^g_{\sigma,T}[X^{\prime}]\quad \P\text{-a.s.}
$$

We now prove that CLI $\Rightarrow$ CLLI and MLI under the assumption $g(t,\omega, y,0) = 0$, for $dt \otimes d\P$-a.e.
 $(t,\omega) \in [0,T] \times \Omega$ and for all $y \in \R$. Fix $\sigma$, $\tau$, $\tau^{\prime}$ in  $\Tc_{0,T}$ such that $\sigma \leq \tau \leq \tau^{\prime}$. Let $X \in L^{\infty} \left(\Fc_{\tau}\right)$ and $X^{\prime} \in L^{\infty} \left(\Fc_{\tau^{\prime}}\right)$ satisfying
$$  Law\left( X  | \Fc_{\sigma}\right) =  Law\left( X^{\prime}  | \Fc_{\sigma}\right),~\P\text{-a.s.} $$
Then, by conditional law-invariance, we obtain $$ \Ec^g_{\sigma, T} \left [X \right] = \Ec^g_{\sigma, T} \left [X^{\prime} \right], ~ \P\text{-a.s.} $$ Moreover, by the Constant-Preserving property (since $g(t,\omega, y, 0) = 0$), we have that 
$$ \Ec^g_{\sigma, T} \left [X \right] = \Ec^g_{\sigma, \tau} \left [X \right], ~~ \text{ and, } ~~ \Ec^g_{\sigma, T} \left [X^{\prime} \right] = \Ec^g_{\sigma, \tau^{\prime}} \left [X^{\prime} \right].$$
Therefore, CLI $\Rightarrow$ MLI, which, by Remark \ref{rmk:defs}, implies CLLI. Therefore, the equivalence between LI, CLI, CLLI, and MLI follows, which concludes the proof.

\end{proof}
Let us now present the main results of our paper, which are necessary and sufficient conditions for law-invariance in three different cases $g(t,\omega, y, 0) = 0$, $g(t,\omega, y, 0)$ is \textit{deterministic} and the final \textit{non-deterministic} case. 
\subsubsection{Case $g(t, \omega, y, 0) = 0$}
In the following, we present a necessary and sufficient condition for the weakest form of law-invariance (LI) in the case where  $g(t, \omega, y, 0) = 0$. By Theorem~\ref{thm:const}, the 
four notions are equivalent in this case, so it suffices 
to characterize the weakest one (LI).
\begin{Theorem}\label{thm:main}
Assume that $g$ satisfies \textbf{(A1)}–\textbf{(A4)} and  $g(t,\omega, y,0) = 0$, for $dt \otimes d\P$-a.e.
 $(t,\omega) \in [0,T] \times \Omega$ and for all $y \in \R$. 
The BSDE \eqref{eq:BSDE}
is \emph{law-invariant} (LI) if and only if the generator $g$ is deterministic $dt \otimes d\P$-a.e and there exists a deterministic function $k:\R \to\R$ such that for all $(y,z) \in \R \times \R^d$ and $dt \otimes d\P$-a.e.
 $(t,\omega) \in [0,T) \times \Omega$
$$
g(t, \omega, y,z) =  k(y)\,|z|^{2}. $$
\end{Theorem}

\begin{proof} We start by the direct implication (the necessary condition). The proof is organized in three steps.\\

\emph{\underline{Step 1:} ($g$ is deterministic $dt \otimes d \P$-a.e.)} 
Let $C>0$ be a given constant and define the following stopping time
$$ \tau^{t}_{C} := \inf \left \{ s \in [0,T-t] , ~~ \lvert W_{t + s} - W_t \lvert > C \right \} \wedge (T-t).$$
According to Theorem \ref{thm:representation}, it suffices to show that for a given $t\in [0,T]$,  $\eps \in (0,T-t]$ and $(y,z) \in \R \times \R^d$, that \\$\Ec^g_{t,t+\eps  \wedge \tau^t_C}\left [ y + z(W_{t + \eps \wedge \tau^t_C} - W_t) \right]$ is deterministic. Suppose that \\$\Ec^g_{t,t+\eps  \wedge \tau^t_C}\left [ y + z(W_{t + \eps \wedge \tau^t_C} - W_t) \right]$ is not deterministic. Then, there exist $m \in \R$ and $A,B \in \Fc_t$ with $\P(A) = \P(B) >0$ (by the atomlessness of the probability space; see the unconditional version of Lemma \ref{lemma:condprob} in \cite{sierpinski1922fonctions}) such that $$\Ec^g_{t,t+\eps  \wedge \tau^t_C}\left [ y + z(W_{t + \eps \wedge \tau^t_C} - W_t) \right] - m >0,~ \text{ $\P$ a.s. on } A, $$ and 
$$ \Ec^g_{t,t+\eps  \wedge \tau^t_C}\left [ y + z(W_{t + \eps \wedge \tau^t_C} - W_t) \right] - m  \leq 0,~ \text{ $\P$ a.s. on } B,$$ which then yields, by the Zero-one Law property (since $g(t,\omega, 0, 0) = 0$) of the quadratic $g$-expectation 
\begin{equation} \label{inequality2}
\begin{aligned}
 \Ec^g_{t,t+\eps  \wedge \tau^t_C}\left [ \left(y + z(W_{t + \eps \wedge \tau^t_C} - W_t) \right) \1_A \right]  & >  m\1_A, ~~ \P\text{-a.s.}, \\
  \Ec^g_{t,t+\eps  \wedge \tau^t_C}\left [ \left(y + z(W_{t + \eps \wedge \tau^t_C} - W_t) \right) \1_B\right]  & \leq m\1_B, ~~ \P\text{-a.s.}.
 \end{aligned}
\end{equation}
Furthermore, by Time-Consistency and Strict-Monotonicity of the $g$-expectation, one can show that the two inequalities \eqref{inequality2} yield
\begin{equation} \label{inequality3}
\begin{aligned}
 \Ec^g_{0,t+\eps  \wedge \tau^t_C}\left [ \left(y + z(W_{t + \eps \wedge \tau^t_C} - W_t) \right) \1_A \right]  & >  \Ec^g_{0,t} \left[ m\1_A \right],\\
  \Ec^g_{0,t+\eps  \wedge \tau^t_C}\left [ \left(y + z(W_{t + \eps \wedge \tau^t_C} - W_t) \right) \1_B\right]  & \leq \Ec^g_{0,t} \left[ m\1_B \right].
 \end{aligned}
\end{equation}
Observing that $m \1_A \stackrel{d}{=} m \1_B$ and that $$\left(y + z(W_{t + \eps \wedge \tau^t_C} - W_t) \right) \1_A \stackrel{d}{=} \left(y + z(W_{t + \eps \wedge \tau^t_C} - W_t) \right) \1_B$$ since $\P(A) = \P(B)$ and $W_{t + \eps \wedge \tau^t_C} - W_t$ is independent of $\Fc_t$.
Indeed, the stopping time $\tau_C^t$ is measurable with respect to the shifted filtration $\Gc^t_u := \sigma \left(W_{t+s} - W_{t} ; 0 \leq s \leq u \right)$ with $0 \leq u \leq T-t$, which is independent of $\Fc_t$.
We have by law-invariance and Theorem \ref{thm:const} that 
$$ \Ec^g_{0,t+\eps  \wedge \tau^t_C}\left [ \left(y + z(W_{t + \eps \wedge \tau^t_C} - W_t) \right) \1_A \right] = \Ec^g_{0,t+\eps  \wedge \tau^t_C}\left [ \left(y + z(W_{t + \eps \wedge \tau^t_C} - W_t) \right) \1_B \right],$$
and
$$\Ec^g_{0,t} \left[ m\1_A \right]  = \Ec^g_{0,t} \left[ m\1_B\right],$$
which, by replacing in \eqref{inequality3} yields
two contradictory inequalities.\\ Hence, $\Ec^g_{t,t+\eps  \wedge \tau^t_C}\left [ y + z(W_{t + \eps \wedge \tau^t_C} - W_t) \right]$ is $\Fc_0$-measurable and verifies by Time-Consistency and the Constant-Preserving property (since $g(t,\omega, y, 0) = 0$) of the operator that
\begin{equation} \label{eq:tto0}
    \Ec^g_{t,t+\eps  \wedge \tau^t_C}\left [ y + z(W_{t + \eps \wedge \tau^t_C} - W_t) \right] =\Ec^g_{0,t+\eps  \wedge \tau^t_C}\left [ y + z(W_{t + \eps \wedge \tau^t_C} - W_t) \right],
\end{equation} 
Finally, we have shown that for each $(y,z) \in\R \times \R^d$, and almost every $t \in [0,T)$,
 $$ g(t, \cdot, y,  z) = \lim_{\eps \to 0^+} \frac{1}{\eps} \left( \Ec^g_{0,t+\eps  \wedge \tau^t_C}\left [ y + z(W_{t + \eps \wedge \tau^t_C} - W_t) \right]  - y \right) , ~~ \P\text{-a.s.},$$
where $\Ec^g_{0,t+\eps  \wedge \tau^t_C}\left [ y + z(W_{t + \eps \wedge \tau^t_C} - W_t) \right]$ is $\Fc_0$-measurable.
Thus, for each $(y,z) \in\R \times \R^d$, and almost every $t \in [0,T)$,
$$ g(t, \cdot , y, z) = \E\left [g(t, \cdot , y, z) \right], ~~ \P\text{-a.s.},$$
which is well-defined since $g$ is dominated by deterministic functions under assumption \textbf{(A2)}.
By continuity of $g$ in $(y,z)$ uniformly in $(t,\omega) \in [0,T]\times \Omega$, under \textbf{(A1)}, and the continuity of $(y,z) \mapsto  \E\left [ g(t, \cdot, y, z)\right] $ by domination of $g$ under \textbf{(A1)}–\textbf{(A2)}, we deduce for $dt \otimes d\P$-a.e.
 $(t,\omega) \in [0,T) \times \Omega$ and all $(y,z) \in \R \times \R^d$,
\begin{equation} \label{eq:ae}
 g(t,\omega, y, z) =  \E\left [g(t, \cdot , y, z) \right].
\end{equation}
We define $\bar{g}$ a deterministic function verifying for any $t \in [0,T]$, and any $(y,z) \in \R \times \R^d$
$$\bar{g}(t,y,z) := \E\left [g(t, \cdot , y, z) \right]. $$
It verifies by $\eqref{eq:ae}$ that $g = \bar{g}$ $dt\otimes d\P$-a.e. and in all $(y,z) \in \R \times \R^d$.
Furthermore, by the uniqueness of the BSDE \eqref{eq:BSDE} and Fubini's theorem, one can show that for any $\sigma$,
$\tau \in \Tc_{0,T}$ with $\sigma \le \tau$, $\P$-a.s., and any $X \in L^{\infty}\left( \Fc_{\tau} \right)$ :
$$ \Ec_{\sigma, \tau}^g \left[ X \right] =  \Ec_{\sigma, \tau}^{\bar{g}} \left[ X \right], ~~ \P\text{-a.s.}$$
We aim to show that there exists a deterministic function $k:\R\to \R$ such that for all $(y,z) \in \R\times \R^d$ and almost every $t \in [0,T)$ $$\bar{g}(t, y, z) = k(y) \lvert z \lvert^2.$$

\smallskip

\underline{\emph{Step 2:}} ($\bar g$ is time-invariant $a.e.$) 
Having shown that $g$ is deterministic, we now use the maturity 
law-invariance (Theorem~\ref{thm:const}) together with 
the Brownian scaling in Lemma~\ref{lemma: lawBM} to show that 
$\bar{g}$ does not depend on $t$. 

Let $(y,z) \in  \R \times \R^d$, we obtain using equation \eqref{eq:tto0}, the maturity law-invariance by Theorem \ref{thm:const}, and Lemma \ref{lemma: lawBM}, for any $0\leq s < t \leq T$ and $\eps \in (0,T-t]$
\begin{align*}
\Ec^{\bar g}_{t,t+\eps  \wedge \tau^t_C}\left [ y + z(W_{t + \eps \wedge \tau^t_C} - W_t) \right] &= \Ec_{0, t+\eps  \wedge \tau^t_C}^{\bar g}\left [ y + z(W_{t + \eps \wedge \tau^t_C} - W_t) \right] \\
 &= \Ec_{0,T}^{\bar g}\left [ y + z(W_{t + \eps \wedge \tau^t_C} - W_t) \right] \\
&=\Ec^{\bar g}_{0,T}\left [ y +  z \left(W_{s +  \eps \wedge \tau^s_{ C}} - W_s \right) \right]\\
&=\Ec^{\bar g}_{0, s + \eps \wedge \tau^s_{C}}\left [ y +  z \left(W_{ s + \eps \wedge \tau^s_{ C}} - W_s \right) \right]\\
&=\Ec^{\bar g}_{s,s + \eps \wedge \tau^s_{C}}\left [ y +  z \left(W_{ s + \eps \wedge \tau^s_{C}} - W_s \right) \right].
\end{align*}
Finally, we apply the representation Theorem \ref{thm:representation} and obtain that almost every $0 \leq s < t < T$
\begin{align*}
    \bar{g} (t,y,z) &= \lim_{\eps \to 0^+} \frac{1}{\eps} \left( \Ec^{\bar g}_{t,t+\eps  \wedge \tau^t_C}\left [ y + z(W_{t + \eps \wedge \tau^t_C} - W_t) \right]  - y \right)\\
    & = \lim_{\eps \to 0^+} \frac{1}{\eps} \left( \Ec^{\bar g}_{s,s+\eps  \wedge \tau^s_C}\left [ y + z(W_{s + \eps \wedge \tau^s_C} - W_s) \right]  - y \right) \\
    &= \bar{g} (s,y,z).
\end{align*}
Thus, for each $(y,z) \in\R \times \R^d$, almost every $s \in [0,T)$
$$\bar g(s,y,z)= \frac{1}{T} \int_{0}^T \bar g(u,y,z) du.$$
Similarly to \emph{Step 1}, we show that we can switch almost every $ s \in [0,T)$ with for any $(y,z) \in \R \times \R^d$.
We define a deterministic function $\bar{\bar g}: \R \times \R^d \to \R$ such that for any $(y,z) \in \R \times \R^d,$ 
$$ \bar{\bar g}(y,z):= \frac{1}{T} \int_{0}^T \bar g(u,y,z) du. $$
It verifies for the same reasons stated in \emph{Step 1} that for any $\sigma$,
$\tau \in \Tc_{0,T}$ with $\sigma \le \tau$, $\P$-a.s., and any $X \in L^{\infty}\left( \Fc_{\tau} \right)$ :
$$ \Ec_{\sigma, \tau}^g \left[ X \right] =  \Ec_{\sigma, \tau}^{\bar{\bar g}} \left[ X \right], ~~ \P\text{-a.s.}$$
\smallskip

\underline{\emph{Step 3:}} ($\bar{\bar g}$ is quadratic in $z$) 
It remains to identify the dependence of $\bar{\bar{g}}$ on $z$. 
The key observation is that law-invariance, combined with 
the Brownian scaling and rotational invariance 
(Lemma~\ref{lemma: lawBM}), forces the homogeneity relation 
$\bar{\bar{g}}(y,\lambda O z) = \lambda^2 
\bar{\bar{g}}(y,z)$ for all $\lambda \in (0,1)$ and 
orthogonal matrices $O$, which by 
Lemma~\ref{lemma:rotinvariance} implies the purely quadratic 
form.
  
Let $(y,z) \in  \R \times \R^d$, $\lambda \in (0,1)$ and $O$ an orthogonal matrix. For any $\eps \in (0,T)$, we obtain
using the law-invariance, Lemma \ref{lemma: lawBM}, and Theorem \ref{thm:const} 
$$\Ec^{\bar{\bar g}}_{0,\eps  \wedge \tau^0_C}\left [ y + \lambda OzW_{ \eps \wedge \tau^0_C} \right] = \Ec^{\bar{\bar g}}_{0, \lambda^2 \eps \wedge \tau^0_{\lambda C}}\left [ y +  z W_{ \lambda^2 \eps \wedge \tau^0_{\lambda C}} \right].$$
Finally, we apply the representation Theorem \ref{thm:representation} and obtain $\P$-a.s.
\begin{align*}
    \bar{\bar g}(y, \lambda Oz) &= \lim_{\eps \to 0^+} \frac{1}{\eps} \left( \Ec^{\bar{\bar g}}_{0,\eps  \wedge \tau^0_C}\left [ y + \lambda OzW_{ \eps \wedge \tau^0_C} \right]  - y \right) \\
     & = \lim_{\eps \to 0^+} \frac{1}{\eps} \left(\Ec^{\bar{\bar g}}_{0, \lambda^2 \eps \wedge \tau^0_{\lambda C}}\left [ y +  z W_{ \lambda^2 \eps \wedge \tau^0_{\lambda C}}  \right]  - y \right) \\
     & = \lambda^2 \lim_{\tilde{\eps} \to 0^+} \frac{1}{\tilde{\eps}} \left( \Ec^{\bar{\bar g}}_{0,\tilde{\eps} \wedge \tau^0_{\tilde{C}}}\left [ y +  zW_{\tilde{\eps} \wedge \tau^0_{\tilde{C}}}  \right]  - y \right) \\
    &= \lambda^2 \bar{\bar g}( y, z).
\end{align*}
In the third line, we did a change of variable on $\eps$ by considering $\widetilde \eps := \lambda^2 \eps$ and we set $\widetilde C := \lambda C$ as a new constant for the representation Theorem~\ref{thm:representation}.
From the result above, we conclude that the mapping $\bar{\bar g}(y,z) $ verifies the conditions of Lemma \ref{lemma:rotinvariance}. Thus, for all $(y,z) \in \R \times \R^d$, we have 
$$ \bar{\bar g}(y,z) = k(y) \lvert z \lvert^2, $$
where $k(y) := \bar{\bar g}(y, e_1) = \frac{1}{T}\int_0^T \E \left[g(t,\cdot,y,e_1)\right] dt$.
Finally, combining all the results above, we conclude that there exists a deterministic function $k$ such that for any $(y,z) \in \R \times \R^d$  and $dt \otimes d\P$-a.e.
 $(t,\omega) \in [0,T) \times \Omega$
$$
g(t, \omega, y,z) =  k(y)\,|z|^{2}. $$

For the converse implication, it is a direct result of Proposition 4.1 in \cite{XuXuZhou2022}.
\end{proof}

\begin{Remark}\label{rmk:assmain}
    In this proof, we only use assumptions \textbf{(A3)} and \textbf{(A4)} to guarantee the uniqueness, which in turn gives the strict-monotonicity, Zero-one Law, and Constant preserving properties in Proposition \ref{prop:Eg}. Hence, if a $g$-expectation satisfies said properties, and its generator satisfies \textbf{(A1)}-\textbf{(A2)}, then Theorem \ref{thm:main} remains valid. We note also that we need the assumptions \textbf{(A1)}-\textbf{(A2)} to use Theorem \ref{thm:representation} in the proof.
\end{Remark}
It is possible to construct generators that are LI or CLI without satisfying $g(t, \omega, y, 0) = 0$. Nevertheless, we show in the following corollary that it is not the case for the strongest form of law-invariance (MLI).

\begin{Corollary}\label{cor:mcli}
    Assume that $g$ satisfies \textbf{(A1)}–\textbf{(A4)}.
The BSDE \eqref{eq:BSDE}
is \emph{MLI} if and only if the generator $g$ is deterministic $dt \otimes d\P$-a.e and there exists a deterministic function $k:\R \to\R$ such that for all $(y,z) \in \R \times \R^d$ and $dt \otimes d\P$-a.e.
 $(t,\omega) \in [0,T) \times \Omega$
$$
g(t, \omega, y,z) =  k(y)\,|z|^{2}. $$
\end{Corollary}
\begin{proof}
    We show that the MLI property implies necessarily that $g(t, \omega, y, 0) = 0$ for $dt \otimes d\P$-a.e.
 $(t,\omega) \in [0,T) \times \Omega$ and for all $y \in \R$. As noted in Remark \ref{rmk:defs}, MLI implies the Constant-Preserving property. Thus, for all $y \in \R$ and for almost every $t \in [0,T)$
\begin{align*}
    g(t, y, 0) &= \lim_{\eps \to 0^+} \frac{1}{\eps}  \left( \Ec^{g}_{t, t+\eps  \wedge \tau^t_C}\left [y \right] - y \right) = 0, \quad \P\text{-a.s.}
\end{align*}
Hence, we obtain the desired result by applying Theorem \ref{thm:main}. Reciprocally, the sufficient condition is a by-product of Theorem \ref{thm:main}.
\end{proof}
Consequently, throughout the rest of the paper, we study the necessary conditions for law-invariance  (or equivalently conditional law-invariance), and the conditional local law-invariance only.

\subsubsection{Case $g(t, \omega, y, 0)$ deterministic}
Let $(Y,Z)$ be the unique solution of the BSDE \eqref{eq:BSDE}. The purpose of the following results is to investigate the necessary condition for law-invariance under weaker assumptions (see \textbf{(H)}). The main idea is to generalize the results of the previous regime by studying a new BSDE represented by $v(t,Y_t)$, and then to be able to recover the initial BSDE given by $Y_t$ (see Proposition \ref{prop:Koby-tildeg}). In order to construct $v\in C^{1,2}([0,T] \times \R)$, we first apply Itô's formula:
\begin{align}\label{Itoformula}
d v(t,Y_{t})&= \left[\partial_t v(t,Y_{t}) - \partial_y v(t,Y_{t})g(t,Y_{t},Z_{t})+\tfrac{1}{2}\partial_{yy}v(t,Y_{t})\lvert Z_{t}\lvert^{2}\right]dt\nonumber\\
&+\partial_y v(t,Y_{t})Z_{t} d W_{t}, 
\end{align}
where the drift is the point of interest.

One can observe the following derivation
\begin{equation}\label{decomp}
    g(t,\omega, y, z) = \underbrace{g(t,\omega, y, 0)}_{h(t, \omega, y)} + \underbrace{g(t,\omega, y, z) - g(t,\omega, y, 0)}_{\bar g (t,\omega, y,z)},
\end{equation} 
where $\bar g$ verifies the assumptions of the first case.
Plugging \eqref{decomp} into the Itô formula \eqref{Itoformula}, we decompose the new generator of the BSDE $(U_t:= v(t,Y_t), \tilde{Z}_t := \partial_y v(t, Y_t) Z_t)$ into two parts:
\begin{align*}
    \begin{cases}
        \tilde{h}(t, \omega, u) = \partial_t v(t,v^{-1}(t,u)) - \partial_y v(t,v^{-1}(t,u))h(t, \omega, v^{-1}(t,u)),  \\
        \tilde{g}(t, \omega, u ,\tilde z) = - \partial_y v(t,v^{-1}(t,u)) \bar g \left(t,\omega, v^{-1}(t,u),\frac{\tilde z}{\partial_y v(t, v^{-1}(t,u))}\right) + \frac{\partial_{yy}v(t,v^{-1}(t,u))}{2 \partial_y v(t, v^{-1}(t,u))^2}\lvert \tilde z\lvert^{2},\\
    \end{cases}
\end{align*}

where we want to enforce on $ \tilde{g}$ the assumptions of Theorem \ref{thm:main} while $\tilde{h}$ is equal to 0. This leads to the study of the following first-order linear PDE: 
\begin{equation} \label{eq:1stPDE}
\partial_t v(t,y) - \partial_y v(t,y)\,h(t,y) = 0.
\end{equation}
Furthermore, $v$ should be deterministic and invertible to transfer law-invariance back and forth.\\

We now introduce some assumptions on $h$.
\smallskip

\emph{Assumption} \textbf{(H):} $h(t,y):=g(t,\omega,y,0)$ is deterministic and satisfies

\smallskip
\noindent{\textbf{(H0)}} The function $h$ is of class 
$C^{0,2}([0,T]\times\R)$, i.e.\ $h$, $\partial_y h$, and  $\partial_{yy} h$ exist and are continuous on $[0,T] \times \R$.

\noindent{\textbf{(H1)}} There exists $H\in L^1([0,T])$ such that
$$
\big|\partial_y h(t,y)\big|\le H(t),\qquad \forall (t,y) \in [0,T] \times \R.
$$

\noindent{\textbf{(H2)}} $\partial_{yyy}h$ exists and there exists a constant $B>0$ such that 
$$
\big|\partial_{yy}h(t,y)\big|\le B,
\qquad
\big|\partial_{yyy}h(t,y)\big|\le B,\qquad \forall (t,y) \in [0,T] \times \R.
$$

The study of the PDE above is done later in the Appendix (Lemma \ref{Lemma:1stPDE}). We now introduce and study the generator $\Tilde g$ associated with the BSDE represented by $v(t, Y_t)$ and defined as follows 
$$
\Tilde g(t,u,\Tilde z)
:= -\,\partial_t v(t,y)+\partial_y v(t,y)\,g\Big(t,y,\frac{\Tilde z}{\partial_y v(t,y)}\Big)
-\frac12\,\partial_{yy}v(t,y)\,\Big|\frac{\Tilde z}{\partial_y v(t,y)}\Big|^2,
$$
with $y:=v^{-1}(t,u)$ and $\tilde{z}=\partial_y v(t,y) z$, where $v$ is the solution of PDE \eqref{eq:1stPDE}. We study its properties in Proposition \ref{prop:Koby-tildeg} and show that the two $g$-expectations are linked as follows 
$$
\Ec^{\Tilde g}_{\sigma,\tau}[X]\;=\; v\Big(\sigma,\ \Ec^{g}_{\sigma,\tau}\big[v^{-1}(\tau,X)\big]\Big), ~~ \P\text{-a.s.}
$$

Building on the previous results, we now give a necessary and sufficient condition for the law-invariance properties under the additional Assumption \textbf{(H)}.
\begin{Theorem} \label{thm:mainB}
    Assume that $g$ satisfies \textbf{(A1)}–\textbf{(A4)} and \textbf{(H)}. The following are equivalent:
    \begin{itemize}
        \item[i)] The BSDE \eqref{eq:BSDE} is LI;
        \item[ii)] The BSDE \eqref{eq:BSDE} is CLLI;
        \item[iii)] The generator $g$ is deterministic $dt\otimes d\P$-a.e. and has the following form for all $(y,z) \in \R \times \R^d$ and $dt \otimes d\P$-a.e. $(t,\omega) \in [0,T) \times \Omega$  $$ g(t, \omega, y, z) = h(t,y) + f(t,y) \lvert z \lvert^2,$$
    with $h$ in $C^{0,2}\left([0,T] \times \R \right)$ and $f$ is a $C^{1,1}\left([0,T] \times \R \right)$ solution of the following PDE
    \begin{equation}  \label{eq:PDEf}
     \partial_t f(t,y) - h(t,y) \partial_y f(t,y) - \partial_y h(t,y) f(t,y) = \tfrac12 \partial_{yy}h(t,y), \quad (t,y) \in [0,T] \times \R.
     \end{equation}
    \end{itemize} 
\end{Theorem}

\begin{proof}[Proof of Theorem~\ref{thm:mainB}] 
\smallskip
\emph{i) $\Rightarrow$ iii)}. The proof is organized in two steps:
  
\emph{\underline{Step 1:}} ($g$ is deterministic)
Set $h(t,y):=g(t,y,0)$, deterministic by assumption, and $v$ as the $C^{1,2}$ solution of PDE \eqref{eq:1stPDE}.
From the change of variable in Proposition \ref{prop:Koby-tildeg},
$$
\Tilde g(t,u,\Tilde z)= \partial_y v(t,y)\Big(g(t,y,z)-h(t,y)\Big)-\frac12\,\frac{\partial_{yy}v(t,y)}{\partial_{y}v(t,y)^2}\,| \Tilde z|^2,
\qquad y=v^{-1}(t,u),\ 
$$ and $\tilde{z}=\partial_y v(t,y) z$
which implies $\Tilde g(t,u,0)=0$. By law-invariance of $\Ec^g$, $\Ec^{\Tilde g}$ is law-invariant. Hence, Theorem~\ref{thm:main} gives a deterministic function $k:\R\to\R$ such that for all $(y,z) \in \R \times \R^d$ and $dt \otimes d\P$-a.e. $(t,\omega) \in [0,T) \times \Omega$
$$
\Tilde g(t, \omega, u,\Tilde z)=k(u)\,|\Tilde z|^2.
$$
Substituting $\Tilde z=\partial_y v(t,y) z$ and solving for $g$,

\begin{align}\label{eqq}
g(t,y,z)=h(t,y)+ \left( k(v(t,y)) \partial_yv(t,y) +\frac12\,\frac{\partial_{yy}v(t,y)}{\partial_y v(t,y)} \right)|z|^2,
\end{align}
implying that $g$ is deterministic.

\emph{\underline{Step 2:}} (PDE satisfied by $f$) 
We now construct $f$ and show that it must solve the 
PDE~\eqref{eq:PDEf}. The argument proceeds by 
introducing the auxiliary function 
\begin{align}\label{eqqq}
K(t,y) := k(v(t,y)),
\end{align}
which solves its own first-order 
PDE~\eqref{eq:KPDE} by the change of variable. We then 
differentiate this PDE and show that $K$ solves 
\eqref{eq:KPDE} if and only if $f$ solves 
\eqref{eq:PDEf}. Define
\begin{align}\label{eqqqq}
f(t,y):=k\!\big(v(t,y)\big)\,\partial_y v(t,y)
+\frac12\,\frac{\partial_{yy}v(t,y)}{\partial_y v(t,y)}.
\end{align}
By \eqref{eqq}, we obtain the claimed form
$$
g(t,y,z)=h(t,y)+f(t,y)\,|z|^2,\qquad dt\otimes d\P\text{-a.e.}
$$
Observe that $h$ is $C^{0,2}$ by assumption \textbf{(H)}. 
First, by using the definition of the function $K$ (see \eqref{eqqq}), the PDE \eqref{eq:1stPDE} and differentiating the composition of $k$ and $v$, we show that $K$ is a solution of the following First-order linear PDE
\begin{equation} \label{eq:KPDE}
 \partial_t K(t,y) - \partial_y K(t,y) h(t,y) = 0, ~~ K(0,y) = k(y).
 \end{equation}
Then, from \eqref{eqqqq}, we deduce
$$K(t,y)=\Frac{1}{\partial_y v(t,y)} \left[ f(t,y) - \Frac12 \Frac{\partial_{yy} v(t,y)}{\partial_y v(t,y)} \right].$$
Computing the partial derivatives from the above expression of $K$ yields (we omit the arguments $(t,y)$ for convenience)
\begin{align*}
    \partial_t K &= \frac{-\partial_{ty}v}{\partial_y v^2}\left[ f - \tfrac12 \frac{\partial_{yy} v}{\partial_{y} v} \right] + \frac{1}{\partial_y v} \left[ \partial_t f - \tfrac12 \frac{\partial_{tyy} v \partial_{y}v- \partial_{ty} v \partial_{yy}v}{\partial_y v^2}\right]\\
    \partial_y K   &=  \frac{-\partial_{yy}v}{\partial_y v^2}\left[ f - \tfrac12 \frac{\partial_{yy} v}{\partial_{y} v} \right] + \frac{1}{\partial_y v} \left[ \partial_y f - \tfrac12 \frac{\partial_{yyy} v \partial_y v - (\partial_{yy} v)^2 }{\partial_y v^2}\right].
\end{align*}
Hence, 
\begin{equation} \label{eq:K}
\begin{aligned}
     \partial_t K - \partial_y K \, h &= \frac{-1}{\partial_y v^2} \left[ f - \tfrac12 \frac{\partial_{yy} v}{\partial_{y} v} \right] \left( \partial_{ty}v -\partial_{yy}v h \right) + \frac{1}{\partial_y v} \left[ \partial_t f -  \partial_y f h \right] \\
     & - \frac{1}{2 \partial_y v^3} \left[ \partial_{tyy} v \partial_{y}v- \partial_{ty} v \partial_{yy}v -   \partial_{yyy} v \partial_y v h + (\partial_{yy} v)^2 h \right].
\end{aligned}
\end{equation}
In the following, we derive some equations that we will use to simplify the computations above. We proceed by differentiating the PDE \eqref{eq:1stPDE} twice w.r.t. $y$. First, since $\partial_{yt} v = \partial_{ty} v$, we obtain the following equation
\begin{equation}\label{eq:helppde0}
    \partial_{ty} v  -  \partial_y v  \, \partial_y h  - \partial_{yy} v \, h = 0,
\end{equation}
and by multiplying all the terms by $\partial_{yy} v$ yields
\begin{equation}\label{eq:helppde}
    \partial_{ty} v \, \partial_{yy} v -  \partial_y v \, \partial_{yy} v \, \partial_y h  - (\partial_{yy} v)^2 \, h = 0.
\end{equation} 
Secondly, since $\partial_{yty} v = \partial_{tyy} v$, we obtain the following equation
\begin{equation} \label{eq:helppde2}
    \partial_{tyy} v  - \partial_{y} v \, \partial_{yy} \,h - 2 \partial_{yy} v \,\partial_y h  - \partial_{yyy} v\, h= 0.
\end{equation} 
Now, we carefully rearrange the terms in \eqref{eq:K}. Using \eqref{eq:helppde0} and by multiplying by $\partial_y v^2$, we obtain
\begin{align*}
     \left(\partial_t K - \partial_y K \, h \right) \partial_y v^2  &= \left[ f - \tfrac12 \frac{\partial_{yy} v}{\partial_{y} v} \right] \left( - \partial_y v \partial_y h \right) +  \partial_y v [\partial_t f -  \partial_y f h] \\
     & - \tfrac12 \left[ \partial_{tyy} v - \partial_{yyy} v  h\right]  - \frac{1}{2 \partial_y v} \left[- \partial_{ty} v \partial_{yy}v  + (\partial_{yy} v)^2 h \right] \\ 
     &= \left[ \partial_t f - h \partial_y f - \partial_y h f \right] \partial_y v
      - \tfrac12 \left[ \partial_{tyy} v - \partial_{yy}v \partial_y h- \partial_{yyy} v  h\right]  \\ & - \frac{1}{2 \partial_y v} \left[- \partial_{ty} v \partial_{yy}v  + (\partial_{yy} v)^2 h \right].
\end{align*}
Substituting the identities \eqref{eq:helppde} and 
\eqref{eq:helppde2} into the expression above eliminates 
all terms involving mixed derivatives of $v$, leaving only 
derivatives of $f$ and $h$. Therefore, using \eqref{eq:helppde} and \eqref{eq:helppde2}, we have
\begin{align*}
     \left(\partial_t K - \partial_y K \, h \right) \partial_y v^2  &= \left[ \partial_t f - h \partial_y f - \partial_y h f \right] \partial_y v
      - \tfrac12 \left[  \partial_y v \partial_{yy} h + \partial_{yy}v \partial_y h \right]  \\ & +\frac{1}{2}  \partial_{yy} v \partial_{y} h \\
      &= \left[ \partial_t f - h \partial_y f - \partial_y h f - \tfrac12 \partial_{yy} h\right] \partial_y v.
\end{align*}
Finally, recalling that $K$ is a solution of the PDE \eqref{eq:KPDE}, we derive that $f$ must satisfy for all $(t,y) \in [0,T] \times \R$ 
\begin{equation} \label{eq:xyz}
\partial_t f(t,y) - \partial_y f(t,y)h(t,y) - f(t,y)\partial_y h(t,y) = \tfrac12 \partial_{yy}h(t,y).
\end{equation}
From the equation above and the assumption \textbf{(H0)}, we deduce that $f$ is in $C^{1,1}([0,T] \times \R)$. More importantly, the computations above show that $f$ is a solution of the PDE \eqref{eq:xyz} if and only if $K$ is a solution of the PDE \eqref{eq:KPDE}. This observation is useful for the next part of the proof.
\smallskip

\emph{iii) $\Rightarrow$ i).} If the generator $g$ is deterministic and is of the form 
for all $(y,z) \in \R \times \R^d$ and $dt \otimes d\P$-a.e. $(t,\omega) \in [0,T) \times \Omega$  $$ g(t, \omega, y, z) = h(t,y) + f(t,y) \lvert z \lvert^2.$$
Then, using the change of variable in Proposition \ref{prop:Koby-tildeg}, yields
that 
\begin{align*}
    \tilde{g}(t, u, \tilde{z}) &= \left ( \frac{f(t, v^{-1}(t,u))}{\partial_{y}v(t,v^{-1}(t,u))} - \tfrac12 \frac{\partial_{yy}v(t,v^{-1}(t,u))}{\partial_{y}v(t,v^{-1}(t,u))^2} \right) \lvert \Tilde z \lvert^2\\
    &= K\left(t, v^{-1}(t,u) \right) \lvert \tilde{z} \lvert^2.
\end{align*}
We aim to show that $t \mapsto K(t, v^{-1}(t,\cdot))$ is constant in order to apply Theorem \ref{thm:main}.
By the chain rule, we have for all $(t,u) \in [0,T] \times \R$
$$\frac{d}{d t}  K(t, v^{-1}(t,u)) = \partial_t K(t, v^{-1}(t, u)) + \partial_t v^{-1}(t, u)\partial_y K(t, v^{-1}(t, u)). $$
Moreover, differentiating $v(t, v^{-1}(t,u))$ w.r.t $t$ yields
$$ \partial_t v(t, v^{-1}(t,u)) + \partial_t v^{-1}(t,u) \partial_y v(t, v^{-1}(t,u)) = 0. $$
Using the PDE \eqref{eq:1stPDE}, we obtain
$$\frac{d}{d t}  K(t, v^{-1}(t,u)) = \partial_t K(t, v^{-1}(t, u)) - h(t, v^{-1}(t,u)) \partial_y K(t, v^{-1}(t, u)). $$
From the computations in \emph{i)} $\Rightarrow$ \emph{iii)}, we showed that 
$$
     \left(\partial_t K - \partial_y K \, h \right) \partial_y v =  \partial_t f - h \partial_y f - \partial_y h f - \tfrac12 \partial_{yy} h.
$$
By the PDE condition in ii) on $f$, the RHS is 0. Hence, it implies that for all $(t,u) \in [0,T] \times \R$
$$ \frac{d}{d t}  K(t, v^{-1}(t,u)) = 0.$$
Consequently, by Theorem~\ref{thm:main}, the BSDE associated with $\Tilde g$ is
LI. Since the deterministic $C^{1,2}$ change of
variable induced by $v$
preserves laws, the BSDE associated with $g$ is also LI. The equivalence between i) and ii) is a consequence of Theorem \ref{thm:const} and Proposition \ref{prop:Koby-tildeg} since for all $T^{\prime} \in [0,T]$ and $\sigma \in \Tc_{0,T}$ we have
$$
\Ec^{\Tilde g}_{\sigma,T^{\prime}}[X]\;=\; v\Big(\sigma,\ \Ec^{g}_{\sigma,T^{\prime}}\big[v^{-1}(T^{\prime},X)\big]\Big), ~~ \P\text{-a.s.}
$$
The
proof is complete.

\end{proof}
\begin{Remark}\label{rmk:assumptions}
The regularity imposed by Assumption~\textbf{(H)} could likely be relaxed. For instance, the boundedness of $\partial_{yyy}h$ in \textbf{(H2)} is used only to control the third-order variational ODE in Lemma~\ref{Lemma:1stPDE}, and weaker integrability conditions may suffice. However, our primary goal is to identify the \emph{structural} form of law-invariant generators, not to pursue minimal regularity; the assumptions in \textbf{(H)} 
are chosen to ensure that the change of variable $v \in C^{1,2}$ is a diffeomorphism, allowing us to transfer law-invariance between the original and transformed BSDEs.

Similarly, for generators of the form $g(t,y,z) = h(t,y) + f(t,y)|z|^2$, the one-sided 
bound in Assumption~\textbf{(A4)} forces $\partial_y f \leq 0$, restricting $f$ to be 
non-increasing in~$y$. We believe that our structural characterization should remain valid beyond this regime, but establishing this under weaker assumptions on $\partial_y g$ would require new well-posedness and comparison results for quadratic BSDEs that go beyond the scope of the present work.
\end{Remark}
\subsubsection{Case $g(t, \omega, y, 0)$ non-deterministic}

In this part, we investigate law-invariance when the generator does not satisfy \emph{Assumption} \textbf{(H)}. Before delving into the characterization of non-deterministic law-invariant generators, we first give an example proving that there exists at least one class of non-deterministic generators satisfying assumptions \textbf{(A1)}–\textbf{(A4)} such that the BSDE \eqref{eq:BSDE} is law-invariant. Indeed, consider any arbitrary non-deterministic essentially bounded process $h$ such that for $t \in [0,T]$, and $(y,z) \in \R \times \R^d$,
$$ g(t, \omega, y, z):= h_t(\omega).$$
Then, for all $\xi \in L^{\infty} (\Fc_{\tau})$, we have
$$ \Ec_{0,T}^g \left [\xi \right] = \E\left [\xi \right] +  \E \left[\int_{0}^{T} h_s  ds \right].$$
Thus, the BSDE \eqref{eq:BSDE} is law-invariant. 

\smallskip
As in the previous case, one may attempt to remove the random drift $g(t,\omega, y, 0)$ by a change of variable through a random field $v$ solving a first order linear PDE with random coefficients, cf.\ \eqref{eq:1stPDE}. 
This yields an equivalence of the form
$$ X\mapsto \Ec^g_{\sigma,\tau}[X]\ \text{ is LI}\ \Longleftrightarrow\ X\mapsto \Ec^{\Tilde g}_{\sigma,\tau}[v(\tau,\cdot,X)]\ \text{ is LI}, $$
where $\Tilde g(t,\omega,y,0)=0$. 
However, in contrast with the deterministic case (where $v$ is a deterministic $C^1$-diffeomorphism in $y$), the randomness of $v$ prevents one from deducing law-invariance of the \emph{untransformed} operator $\Ec^{g}$  from that of the transformed one $\Ec^{\Tilde g}$. We illustrate this by constructing, via the Itô–Wentzell formula\footnote{In this specific example, the multidimensional version of Itô's formula suffices. However, this mechanism can be generalized to more general regular random fields.} an explicit non-deterministic law-invariant generator that is not of the form given by Theorem \ref{thm:mainB}.

\begin{Example}[A non-deterministic LI generator with a linear term]\label{ex:random_LI0_linear_term}
Let $g(t,\omega, y, z)=\frac12|z|^2$ and consider the solution $(Y,Z)$ of the BSDE associated with the generator $g$. Fix $r\in C^1([0,T])$ with $r(0)=r(T)=0$
and $\psi\in C_b^2(\R)$. Define for all $(t,\omega, y) \in [0,T] \times \Omega \times \R$
$$ v(t,\omega,y):=y-r(t)\psi\big(W_t^1(\omega)\big), $$
where $W^1$ denotes the first coordinate of $W$. We now construct a law-invariant BSDE by applying Itô–Wentzell's formula to $v(t,Y_t)$
$$ dv(t,Y_t) = -\big(\underbrace{r^{\prime}(t) \psi\big(W_t^1\big)  + \frac12 r(t)\psi^{\prime \prime} \big(W_t^1 \big)}_{a_t} + \frac12|Z_t|^2 \big)dt + \big(Z_t - \underbrace{r(t)\psi^{\prime} \big(W_t^1 \big) e_1}_{b_t} \big) dW_t.$$
Hence, $(\tilde{Y}_t:= v(t,Y_t), \tilde{Z}_t:= Z_t - b_t)$ is the solution of the BSDE associated with the following generator
$$
\tilde g(t,\omega, y, z) =  a_t(\omega)+\frac12|b_t(\omega)|^2 + b_t(\omega) z + \frac12|z|^2.
$$
In particular, $\Tilde g$ is non-deterministic and not of the form $h(t,y)+f(t,y)|z|^2$ due to the linear term
$ b_t z$, while it is
law-invariant by construction (since $r(0) = r(T) = 0$):

$$ \Tilde Y_0 = v \left(0, \Ec^g_{0,T} \left[v^{-1}(T,X) \right] \right) = \Ec^g_{0,T} \left[X \right], ~~ \forall X \in L^{\infty}(\Fc_T).$$
\end{Example}
Thus, it is clear for us that we should extract the necessary conditions using other means than the invariance properties of the Brownian increments. In the following, we use the differentiability representation in Theorem \ref{thm:diff} to do so. 

\smallskip

Let us now provide an alternative necessary condition on the generator for law-invariance and conditional local law-invariance. We denote by $\Ec$ the Dol\'eans-Dade exponential expressed as follows
$$
\mathcal E\Big(\int_t^{T^{\prime}}\theta_r\,dW_r\Big)_s
=\exp\Big(\int_t^s \theta_r\,dW_r-\frac12\int_t^s|\theta_r|^2\,dr\Big),
\qquad s\in[t,T^{\prime}],
$$
where $T^{\prime} \in [0,T]$, $\theta\in\Hc^2_{[0,T^{\prime}]}$, and $t \in [0,T^{\prime}]$.
\begin{Theorem} \label{thm:general}
    Assume \textbf{(A1)}-\textbf{(A4$^*$)} hold. If the BSDE \eqref{eq:BSDE} is LI, then the generator $g$ must satisfy the following identity:
    for all $(t, y) \in [0,T] \times \R,$ $\P$-a.s.
\begin{equation}\label{eq:cons1}
    \Ec \left( \int_{t}^T \partial_z g(s, Y^y_s, Z^y_s) dW_s \right)_T = \Frac{\exp\left( -\int_{t}^T \partial_y g(s, Y^y_s, Z^y_s) ds \right)}{\E \left[\exp\left( -\int_{t}^T \partial_y g(s, Y^y_s, Z^y_s) ds \right) \Big | \Fc_{t} \right]} ,
\end{equation}
    where $(Y^y, Z^y)$ is the solution of BSDE$(T, y, g)$.

    Furthermore, if the BSDE \eqref{eq:BSDE} is CLLI, then the generator $g$ must satisfy the following identity: for all $(T^{\prime}, y) \in [0,T] \times \R,$ 
    \begin{equation}\label{eq:cons2}
    \Ec \left( \int_{t}^
    {T^{\prime}}\partial_z g(s, Y^{y,T^{\prime}}_s, Z^{y,T^{\prime}}_s) dW_s \right)_{T^{\prime}} = \Frac{\exp\left( -\int_{t}^{T^{\prime}} \partial_y g(s, Y^{y,T^{\prime}}_s, Z^{y,T^{\prime}}_s) ds \right)}{\E \left[\exp\left( -\int_{t}^{T^{\prime}} \partial_y g(s, Y^{y,T^{\prime}}_s, Z^{y,T^{\prime}}_s) ds \right) \Big | \Fc_{t} \right]} ,
    \end{equation}
    where $(Y^{y,T^{\prime}}, Z^{y,T^{\prime}})$ is the solution of BSDE$(T^{\prime}, y, g)$.
\end{Theorem}
\begin{proof} The proof is divided into three steps.

    \emph{\underline{Step 1:}}
    Using the differentiability results presented in Theorem \ref{thm:diff}, we have that for all $X \in L^{\infty}(\Fc_T)$, $(t,y) \in [0,T] \times \R$ 
    \begin{equation} \label{eq:limitEg}
         \lim_{\eps \to 0^+} \frac{1}{\eps} \left( \Ec^g_{t, T} \left [y + \eps X\right] - \Ec^g_{t, T} \left [ y \right] \right) = U^X_t, \quad \text{in } L^1(\P),
    \end{equation}  
    where $(U^X, V^X)$ is the unique solution of the following linear BSDE
\begin{align*}
    \begin{cases}
         -d U_s^X = \left(\partial_y g(s, Y_s^y, Z_s^y) U_s^X + \partial_z g(s, Y_s^y, Z_s^y) V_s^X   \right) ds -V_s^X dW_s, ~~ s \in [t, T), \\
        U_{T}^X = X,
    \end{cases}
\end{align*}
   and $(Y^y, Z^y)$ is the solution of BSDE$(T, y, g)$. Furthermore, we have
   \begin{equation} \label{eq: equalityUcond}
        U_t^X = \E\left [\Gamma^{t,y}_T X \Big | \Fc_t\right],
   \end{equation}
   where $\Gamma^{t,y}$ is the solution of the linear SDE $$d\Gamma^{t,y}_s= \partial_y g(s, Y^y_s, Z_s^y) \Gamma^{t,y}_s ds + \partial_z g(s, Y^y_s, Z_s^y)  \Gamma^{t,y}_s dW_s, ~~ \Gamma^{t,y}_t = 1. $$
   Therefore, from \eqref{eq:limitEg} and \eqref{eq: equalityUcond}, we obtain the following relationship
   \begin{equation} \label{eq:equalityLimitCond}
       \lim_{\eps \to 0^+} \frac{1}{\eps} \left( \Ec^g_{t, T} \left [y + \eps X\right] - \Ec^g_{t, T} \left [ y \right] \right) = \E\left [\Gamma^{t,y}_T X \Big | \Fc_t\right].
   \end{equation}
   Herein, we first study the conditional law-invariant case and deduce the result on law-invariance afterwards. Fix $(t,y)\in[0,T]\times\R$. Since $g$ is conditionally law-invariant (by Theorem \ref{thm:const}, LI $\implies$ CLI), we deduce that the linear operator
$$ X \mapsto \E\left[\Gamma_T^{t,y} X \Big|\Fc_t\right] $$
is conditionally law-invariant in the sense that for all $X,X^{\prime}\in L^\infty(\Fc_T)$,
\begin{equation} \label{eq:linop}
    Law(X|\Fc_t)=Law(X^{\prime}|\Fc_t),~\P\text{-a.s.}\ \Longrightarrow\ 
\E\left[\Gamma_T^{t,y} X \Big|\Fc_t\right]=\E\left[\Gamma_T^{t,y} X^{\prime} \Big|\Fc_t\right],~\P\text{-a.s.}
\end{equation}
Indeed, one can see that for all $X,X^{\prime}\in L^\infty(\Fc_T)$ satisfying the left hand side equality in \eqref{eq:linop},
we have for all $y \in \R$ and $\eps >0$ $$Law(y + \eps X|\Fc_t)=Law(y + \eps X^{\prime}|\Fc_t),~\P\text{-a.s.},$$ 
hence
$$\Ec^g_{t, T} \left [y + \eps X\right] = \Ec^g_{t, T} \left [y + \eps X^{\prime} \right],~\P\text{-a.s.} $$
Together with \eqref{eq:equalityLimitCond}, this yields \eqref{eq:linop}.

\emph{\underline{Step 2:}}
We show here that $\Gamma^{t,y}_T$ is $\Fc_t$-measurable.
We argue by contradiction using the conditional 
atomlessness of $\Fc_T$ relative to $\Fc_t$ 
(Lemma~\ref{lemma:condprob}). By Lemma \ref{lemma:condprob}, there exists $D \in \Fc_t$ such that $\P(D)>0$ and two events $A \in \Fc_{T}$ and $B \in \Fc_{T}$ such that $\P(A | \Fc_t) = \P(B | \Fc_t) > 0$, $\P\text{-a.s.}$ on $D$. Furthermore, there exists an $\Fc_t$-measurable random variable $m_t$ such that
\begin{align*}
    \Gamma_T^{t,y} \1_A &\leq m_t \1_A,\qquad \P\text{-a.s.}\\
    \Gamma_T^{t,y} \1_B &> m_t \1_B,\qquad \P\text{-a.s.} 
\end{align*}
Set $X:=\1_{A\cap D}$ and $X^{\prime}:=\1_{B\cap D}$. Since $\P(A|\Fc_t)=\P(B|\Fc_t)$ on $D$,
we have $\P(A\cap D|\Fc_t)=\P(B\cap D|\Fc_t)$, $\P\text{-a.s.}$ Therefore, since $X$ and $X^{\prime}$ are two random variables following the same Bernoulli distribution with the same parameter conditionally to $\Fc_t$,
$$
Law(X|\Fc_t)=Law(X^{\prime}|\Fc_t),\qquad \P\text{-a.s.}
$$
By \emph{Step 1},
\begin{equation} \label{eq:equalityCond}
\E\!\left[\Gamma_T^{t,y} X|\Fc_t\right]=\E\!\left[\Gamma_T^{t,y} X^{\prime}|\Fc_t\right],\qquad \P\text{-a.s.}
\end{equation}
On the other hand, from $\Gamma_T^{t,y}\1_A\le m_t\1_A$ and $\Gamma_T^{t,y}\1_B>m_t\1_B$, we get
$$
\E\!\left[\Gamma_T^{t,y} X|\Fc_t\right]\le m_t\,\E[X|\Fc_t],\qquad
\E\!\left[\Gamma_T^{t,y} X^{\prime}|\Fc_t\right]> m_t\,\E[X^{\prime}|\Fc_t].
$$
Since $\E[X|\Fc_t]=\P(A\cap D|\Fc_t)=\P(B\cap D|\Fc_t)=\E[X^{\prime}|\Fc_t]$,
we obtain on $D$ the strict inequality
$$
\E\!\left[\Gamma_T^{t,y} X|\Fc_t\right]<\E\!\left[\Gamma_T^{t,y} X^{\prime}|\Fc_t\right],
$$
a contradiction with \eqref{eq:equalityCond}. Hence $\Gamma_T^{t,y}$ is $\Fc_t$-measurable. 

\emph{\underline{Step 3:}}
We now show the constraint \eqref{eq:cons1}. Under the assumptions \textbf{(A1)}-\textbf{(A4$^*$)} (cf. Lemma \ref{lem:coeff_satisfy_L}), $\Ec \left( \int_{t}^T \partial_z g(s, Y^y_s, Z^y_s) dW_s \right)_{\cdot}$ is a u.i. martingale. Hence it satisfies for all $t \in [0,T]$ 
\begin{equation} \label{eq:equalityMart}
\E \left[\Ec \left( \int_{t}^T \partial_z g(s, Y^y_s, Z^y_s) dW_s \right)_T \Big | \Fc_t \right] = 1, ~~ \P\text{-a.s.}    
\end{equation}
Furthermore, since $\Gamma_T^{t,y}$ is $\Fc_t$-measurable 
we have, using
$$
\Ec\left(\int_t^T \partial_z g(s,Y_s^y,Z_s^y)\,dW_s\right)_T
=\Gamma_T^{t,y}\exp\left(-\int_t^T \partial_y g(s,Y_s^y,Z_s^y)\,ds\right),
$$
and \eqref{eq:equalityMart} 
$$
1=\Gamma_T^{t,y}\,\E\left[\exp\left(-\int_t^T \partial_y g(s,Y_s^y,Z_s^y)\,ds\right)\Big|\Fc_t\right].
$$
Hence
$$
\Gamma_T^{t,y}
=\Frac{1}{\E\left[\exp\left(-\int_t^T \partial_y g(s,Y_s^y,Z_s^y)\,ds\right)\Big|\Fc_t\right]}.
$$
Thus, we conclude for all $(t,y) \in [0,T] \times \R$
$$ \Ec \left( \int_{t}^T \partial_z g(s, Y^y_s, Z^y_s) dW_s \right)_T = \Frac{\exp\left( -\int_{t}^T \partial_y g(s, Y^y_s, Z^y_s) ds \right)}{\E \left[\exp\left( -\int_{t}^T \partial_y g(s, Y^y_s, Z^y_s) ds \right) \Big | \Fc_{t} \right]}, ~~ \P\text{-a.s.}$$
The proof for the CLLI case (i.e. statement \eqref{eq:cons2}) follows by the same arguments as above by taking a different maturity $T^{\prime} \in [0,T]$ from the beginning.

\end{proof}
\begin{Remark}
    Note that, in contrast to the proof of Theorem~\ref{thm:main}, we do not use the invariance properties of the Brownian motion in the proof above.
\end{Remark}
We present here two Theorems as special cases of Theorem \ref{thm:general} to give an explicit use of the two conditions \eqref{eq:cons1} and \eqref{eq:cons2}. To derive necessary and sufficient conditions in a special subcase, we consider the following extra assumption:
\begin{itemize}
     \item[ \textbf{(A5)}] $g$ is independent of $y$ and for $dt\otimes d\P$-a.e.\ $(t,\omega) \in [0,T] \times \Omega$, and the equation $\partial_z g(t,\omega,z)=0$ has the unique solution $z=0$.
\end{itemize}

\begin{Theorem}[LI]\label{cor:litransinv}
Assume \textbf{(A1)}--\textbf{(A4$^*$)} and \textbf{(A5)} hold. Then, the following are equivalent:
  \begin{itemize}
        \item[i)] The BSDE \eqref{eq:BSDE} is LI;
        \item[ii)] The generator $g$ has the following form for all $z \in \R^d$ and $dt \otimes d\P$-a.e. $(t,\omega) \in [0,T) \times \Omega$  $$ g(t, \omega, z) = g(t, \omega, 0) + \beta \lvert z \lvert^2,$$
        where $\beta \in \R$, and $\int_0^T g(t,\cdot,0)dt$ is deterministic $\P$-a.s.
\end{itemize}
\end{Theorem}
\begin{proof}
    We first assume that the generator $g$ of BSDE \eqref{eq:BSDE} is law-invariant. By Theorem \ref{thm:general} and assumption \textbf{(A5)}, we deduce from \eqref{eq:cons1} with $(Y^0,Z^0)$ the solution of BSDE$(T, 0, g)$, that 
    
    $$\Ec \left( \int_{0}^T \partial_z g(s, Z^0_s) dW_s \right)_T = 1,$$
    which in turn yields
    $$\partial_z g(t, Z^0_t) = 0, ~~ dt\otimes d \P \text{-a.e.}$$
    By assumption \textbf{(A5)}, this is only achievable if $Z^0 \equiv 0$, $\P$-a.s.
    Thus
    $$ \Ec^{g}_{0,T}[0] = \int_0^T g(t, 0)dt,~~ \P\text{-a.s.}$$
    Hence $\int_0^T g(t, 0)dt$ is $\P$-a.s.\ deterministic.
    By a classic change of variable, we obtain for any $X \in L^{\infty}(\Fc_t)$ 
    $$ \Ec^g_{0,T}\left[ X \right] = \Ec^{\Tilde g}_{0,T}\left[ X + \int_0^T g(t,0)dt \right], $$
    where $\tilde{g}(t,\omega, z) = g(t, \omega, z) - g(t,\omega, 0)$.
    Since $\int_0^T g(t, 0)dt$ is $\P$-a.s.\ deterministic and that $\Ec^g$ is law-invariant, we obtain that $\Ec^{\tilde{g}}$ is law-invariant. Finally, by Theorem \ref{thm:main}, we conclude that there exists $\beta \in \R$ such that for all $z \in \R^d$ 
    $$ \tilde{g}(t, \omega, z) = \beta \lvert z \lvert^2, ~~ dt\otimes d\P \text{-a.e.}~(t,\omega) \in[0,T) \times \Omega.$$
    This concludes the first implication. The proof for the converse implication is a direct consequence of the same change of variables above.
\end{proof}

\begin{Theorem}[CLLI]\label{cor:CLLItransinv}
Assume \textbf{(A1)}--\textbf{(A4$^*$)} and \textbf{(A5)} hold. Then, the following are equivalent:
  \begin{itemize}
        \item[i)] The BSDE \eqref{eq:BSDE} is CLLI;
        \item[ii)] The generator $g$ has the following form for all $z \in \R^d$ and $dt \otimes d\P$-a.e. $(t,\omega) \in [0,T) \times \Omega$  $$ g(t, \omega, z) = g(t, \omega, 0) + \beta \lvert z \lvert^2,$$
        where $\beta \in \R$, and $g$ is deterministic $dt \otimes d\P$-a.e.
\end{itemize}
\end{Theorem}
\begin{proof}
We apply the same proof of Theorem \ref{cor:litransinv} and obtain that for all $T^{\prime} \in [0,T]$ that $\int_0^{T^{\prime}} g(t, 0)dt $ is deterministic $\P$-a.s.
\end{proof}
We give now an example of a non-deterministic generator $g$ such that $\int_0^T g(t,0)dt$ is deterministic. 
\begin{Remark}\label{rem:LI_not_CLLI}
Let $\tau$ be a non-deterministic stopping time in $\Tc_{0,T-\varepsilon}$ and set
$$ h_t(\omega)=\mathbf 1_{[\tau(\omega),\tau(\omega)+\varepsilon]}(t). $$
Then $\int_0^T h_s\,ds=\varepsilon$ is deterministic, so Theorem \ref{cor:litransinv} yields law-invariance for
$g(t,\omega,y,z)=h_t(\omega)+\frac12|z|^2$, while conditional local law-invariance fails since $h$ is not deterministic
$dt\otimes d\P$-a.e. (Theorem \ref{cor:CLLItransinv}).
\end{Remark}

\section{Law-invariant Dynamic risk measures} \label{sec:risk}
In this Section, we present the implication of the main theorems above on the theory of (fully) dynamic risk measures.
In particular, using our results on law-invariant BSDEs from the previous Section, we show that we can construct a large class of law-invariant risk measures which goes beyond the results from the existing literature (e.g. Kupper-Schachermayer (2009)\cite{KupperSchachermayer2009}).

Now, we recall the definition of a dynamic convex risk measure in a general setting. We cite \cite{BionNadal2008} among others.
\begin{Definition}
A  dynamic risk measure   $(\rho_{\sigma, \tau})_{0 \leq \sigma \leq \tau}$ is a 
family of maps 
defined on $L^{\infty}({\cal F}_{\tau})$ with values into $L^{\infty}({\cal F}_{\sigma})$ such that each $\rho_{\sigma,\tau}$ is a convex conditional risk measure, i.e. $\rho_{\sigma,\tau}$ satisfies  the following  properties, for any $\sigma$,
$\tau \in \Tc_{0,T}$ with $\sigma \le \tau$, $\P$-a.s. : 
\begin{itemize}
\item[\(1.\)] \it{Convexity}: ~For any $X, X^{\prime} \in L^{\infty}(\Fc_{\tau})$ and $\lambda \in [0,1]$,

$$\rho_{\sigma,\tau}(\lambda X +(1-\lambda)X^{\prime}) \leq \lambda \rho_{\sigma,\tau}(X)+(1-\lambda) \rho_{\sigma,\tau}(X^{\prime}), ~~ \P\text{-a.s.};$$ 

\item[\(2.\)] \it{Monotonicity}: ~For any $X, X^{\prime} \in L^{\infty}(\Fc_{\tau})$ with $X \geq X^{\prime}$,
$\P$-a.s., we have $\rho_{\sigma,\tau}(X) \geq \rho_{\sigma,\tau}(X^{\prime})$,
$\P$-a.s.  
\end{itemize}
Moreover, a dynamic convex risk measure can have the following additional properties :
\begin{itemize}
\item[\(3.\)] \it{Strict-Monotonicity}: ~$\rho$ satisfies (2) and, if $\rho_{\sigma,\tau}(X) = \rho_{\sigma,\tau}(X^{\prime})$,
$\P$-a.s., then $X=X^{\prime}$, $\P$-a.s.;

\item[\(4.\)] \it{Cash-additivity}: $$
\;\;\rho_{\sigma,\tau}(X+X^{\prime})=\rho_{\sigma,\tau}(X)+X^{\prime}, \quad \P\text{-a.s.} \quad
\forall X \in L^{\infty}(\Fc_{\tau}),\quad X^{\prime} \in L^{\infty}(\Fc_{\sigma});$$

\item[\(5.\)] \it{Normalized}:~ If $$\rho_{\sigma,\tau}(0)=0, ~~ \P\text{-a.s.};$$
\item[\(6.\)] \it{Fatou property}:~ It is continuous from below (resp  above) 
if for every increasing (resp decreasing) sequence $X_n$ of elements of  $L^{\infty}({\cal F}_{\tau})$ such that $X=\lim\;X_n$, the decreasing (resp increasing) 
sequence $\rho_{\sigma,\tau}(X_n) $ has the limit $\rho_{\sigma,\tau}(X)$;
\item[\(7.\)] \it{Time-Consistency}:~ For any $\iota$, $\sigma$,
$\tau \in \Tc_{0,T}$ with $\iota \le \sigma \le \tau$, $\P$-a.s. :
$$
\rho_{\iota,\sigma}\big(\rho_{\sigma,\tau}(X)\big)=\rho_{\iota,\tau}(X),
\quad ~~ \P\text{-a.s.} \quad \forall X \in L^{\infty}(\Fc_{\tau});
$$
\end{itemize}
Finally, we say that a dynamic risk measure is cash non-additive to emphasize that it does not satisfy cash-additivity.
\label{definition1}
\end{Definition} 
\begin{Remark}
We naturally extend the law-invariance Definitions~\ref{def:li}, \ref{def:cli}, \ref{def:oscli}, and \ref{def:mcli} from Section~\ref{sec:LI} to dynamic risk measures.
\end{Remark}
\subsection{Cash-additive risk measures} 
In this subsection, we are interested in studying the specific case of cash-additive (or translation invariant) and normalized dynamic convex risk measures. Furthermore, we assume that the Fatou property is satisfied. Then, the robust representation holds: for every $0\le \sigma\le \tau\le T$ and $X\in L^\infty(\Fc_\tau)$,
$$
\rho_{\sigma,\tau}(X)
=
\text{esssup}_{\Q\in\tilde{\cal M}_{\sigma,\tau}}
\Big(\E^\Q[X | \Fc_\sigma]-\alpha^m_{\sigma,\tau}(\Q)\Big),
\qquad \P\text{-a.s.},
$$
with $\tilde{\cal M}_{\sigma,\tau}=\{\Q\ll\P:\ \Q_{|\Fc_\sigma}=\P\}$ and $\alpha^m$ the minimal penalty.
This is the conditional version of the dual representation introduced by \cite{FollmerSchied2002article} in the static setting and \cite{DetlefsenScandolo2005} in the dynamic setting.

 We now provide the following theorem which fully characterizes law-invariant and time-consistent dynamic convex risk measures.
\begin{Theorem} \label{thm:risk}
Let $\rho$ be a dynamic convex risk measure which is strictly monotone, cash-additive, normalized, and satisfies the Fatou property. Then the following are equivalent:
\begin{itemize}
\item[i)] $\rho$ is law-invariant, time-consistent, its robust representation is attained in some $\Q^\ast\ll\P$ for each payoff $X\in L^\infty(\Fc_T)$, and
$$
\alpha^m_{0,T}(\P)=0;
$$
\item[ii)] There exists $\beta \geq 0$ such that $\rho$ is generated by the following BSDE :
\begin{align*}
\begin{cases}
-dY_t = \beta \lvert Z_t \lvert^2 \,dt - Z_t dW_t,\quad t\in[0,T),\\[1mm]
Y_T = X \in L^{\infty}(\Fc_T).
\end{cases}
\end{align*}
\end{itemize}
Moreover, LI, CLI, CLLI, and MLI are equivalent for $\rho$.
\end{Theorem}
\begin{proof}
\noindent$\textit{i)} \Longrightarrow \textit{ii).}$ The proof is organized in two steps. In Step~1 we show that $\rho$ admits a BSDE representation whose generator $g$ satisfies $g(t,\omega,y,0) = 0$ and is independent of~$y$. In Step~2 we use law-invariance and Theorem~\ref{thm:main} to extract the form of~$g$.
 
\medskip
\noindent\emph{\underline{Step~1}: BSDE representation.}
Fix $X \in L^{\infty}(\Fc_T)$. Since $\rho$ is a normalized dynamic convex risk measure satisfying the Fatou property, it is in particular continuous from above. Hence, it admits the robust representation
\begin{equation}\label{eq:robust-rep}
\rho_{\sigma,\tau}(X)
=
\text{esssup}_{\Q\in\tilde{\mathcal{M}}_{\sigma,\tau}}
\Big(\E^\Q[X\mid\Fc_\sigma]-\alpha^m_{\sigma,\tau}(\Q)\Big),
\qquad \P\text{-a.s.},
\end{equation}
with $\alpha^m$ the minimal penalty (see Detlefsen--Scandolo~\cite{DetlefsenScandolo2005}, F\"ollmer--Schied~\cite{FollmerSchied2002article}, and Bion-Nadal~\cite{BionNadal2009}).
 
We combine the time-consistency property with the condition $\alpha^m_{0,T}(\P)=0$ to derive three intermediate results:
 
\smallskip
\noindent\textbf{(a)} By Theorem~3.2 in Delbaen--Peng--Rosazza~Gianin~\cite{DelbaenPengRosazza2010}, if the minimal penalty additionally satisfies the regularity property (termed Assumption~(g) therein), then for all $\Q\sim\P$,
\begin{equation}\label{eq:penalty-decomp}
\alpha^m_{\sigma,\tau}(\Q)
=
\E^\Q\Big[\int_\sigma^\tau r(s,\omega,q_s)\,ds \Big| \Fc_\sigma\Big],
\end{equation}
for some proper convex lower semicontinuous function $r(s,\omega,\cdot)$ and the Girsanov drift $q$ of $\Q$ with respect to $\P$. In our setting, this regularity property is automatically satisfied: as observed in Bion-Nadal~\cite{BionNadal2009}, any risk measure satisfying monotonicity and cash-additivity necessarily satisfies the Zero-one law property, which is equivalent to the regularity property required in~\cite{DelbaenPengRosazza2010}. Hence \eqref{eq:penalty-decomp} holds.
 
\smallskip
\noindent\textbf{(b)} By Theorem~3 in Bion-Nadal~\cite{BionNadal2009}, the time-consistency and normalization of $\rho$ imply the existence of a c\`adl\`ag supermartingale process $Y^X = (Y^X_t)_{0\le t\le T}$ such that for all $\sigma \in \Tc_{0,T}$,
\begin{equation}\label{eq:supermartingale}
\rho_{\sigma, T}(X) = Y_{\sigma}^X, \qquad\P\text{-a.s.}
\end{equation}
Applying the Doob--Meyer decomposition theorem to the supermartingale $Y^X$, and the martingale representation theorem to the local martingale part, we obtain $Z \in \Hc^2$ and a unique non-increasing predictable process $A$ with $A_0 = 0$ such that for all $t \in [0,T]$,
\begin{equation}\label{eq:doob-meyer}
Y_{t}^X = X + A_t - A_T - \int_t^T Z_s\, dW_s.
\end{equation}
 
\smallskip
\noindent\textbf{(c)} We now show how the supremum in the robust representation~\eqref{eq:robust-rep}, combined with the Doob--Meyer decomposition~\eqref{eq:doob-meyer} and the penalty structure~\eqref{eq:penalty-decomp}, forces the BSDE driver to be the convex conjugate of the function~$r$.
 
Let $\Q\in\tilde{\mathcal{M}}_{t,T}$ with $\Q\sim\P$, and denote by $q = (q_s)_{t\le s\le T} \in \Hc^2$ the Girsanov drift of $\Q$ with respect to~$\P$, so that
$$
W_s^\Q := W_s - \int_t^s q_u\,du, \qquad s\in[t,T],
$$
is a $\Q$-Brownian motion. Starting from the decomposition~\eqref{eq:doob-meyer}, we derive
\begin{align}
X
&= Y_t^X - A_t + A_T + \int_t^T Z_s\,dW_s \notag\\
&= Y_t^X - A_t + A_T + \int_t^T Z_s\,dW_s^\Q +\int_t^T q_s Z_s\,ds,\label{eq:girsanov-conditional}
\end{align}
where the stochastic integral $\int_t^{\cdot} Z_s\,dW_s^\Q$ vanishes under $\E^\Q[\,\cdot\,|\Fc_t]$ since it is a $\Q$-local martingale (and a true $\Q$-martingale by the BMO property of $\int Z dW$ and the reverse H\"older inequality from Theorem~\ref{thm:bmo_toolkit}).
 
Substituting~\eqref{eq:girsanov-conditional} and the penalty decomposition~\eqref{eq:penalty-decomp} into the robust representation~\eqref{eq:robust-rep}, we obtain

\begin{equation}\label{eq:A-sup}
 S_t:=A_t + \int_0^t \big(q_s Z_s - r(s,q_s)\big)\,ds \geq \,\E^\Q\Big[A_T + \int_0^T \big(q_s Z_s - r(s,q_s)\big)\,ds\Big|\Fc_t\Big].
\end{equation}
The inequality~\eqref{eq:A-sup} yields that $S$ is a $\Q$-supermartingale with finite variation. Hence, $-dA_t - \big(q_t Z_t - r(t,q_t)\big)dt$ defines a non-negative measure. Thus, for all $q \in \Hc^2$
$$ - dA_t \geq \big(q_t Z_t - r(t,q_t)\big)dt, $$
in the sense of random measures.
In particular, since $r(s,\omega,\cdot)$ is proper, convex, and lower semicontinuous by~\textbf{(a)}, the pointwise maximisation in $q_s$ at each $(s,\omega)$ yields the convex conjugate (or Fenchel--Legendre transform) of $r$ with respect to the drift variable:
\begin{equation}\label{eq:fenchel}
r^*(s,\omega,z) := \sup_{q\in\R^d}\big\{q z - r(s,\omega,q)\big\}.
\end{equation}
Thus, one can construct a sequence of elements $q^n \in \Hc^2$ using the subdifferential of $r$ as was done in \cite{DelbaenHuBao2011} and obtain by passing to the limit  
$$
- dA_t \geq r^*(t,Z_t)dt.
$$
Observe that the inequality in~\eqref{eq:A-sup} is attained at some $\Q^\star\ll\P$ with drift $q^\star$ by assumption. At the optimiser, equality holds in~\eqref{eq:A-sup}, so that the martingality yields
$$
-dA_t = r^*(t,Z_t) dt,  ~~ \P\text{-a.s.}
$$

Plugging this into the Doob--Meyer decomposition~\eqref{eq:doob-meyer}, we conclude that for every $X\in L^\infty(\Fc_T)$, the pair $(Y^X,Z^X) := (Y^X,Z)$ satisfies the BSDE
\begin{equation}\label{eq:bsde-rep}
\rho_{t,T}(X)=Y_t^X,
\qquad
-dY_t^X=r^*(t,Z_t^X)\,dt-Z_t^X\,dW_t,
\qquad
Y_T^X=X.
\end{equation}
We refer to Bion-Nadal~\cite{BionNadal2009}, Rosazza~Gianin~\cite{RosazzaGianin2006}, and Delbaen--Hu--Bao~\cite{DelbaenHuBao2011} for the full technical details of this correspondence, including the uniqueness of the solution in $\Sc^\infty \times \Hc^2$.
 
We set
$$
g(t,\omega, y, z):=r^*(t,\omega,z).
$$
It remains to verify that $g$ satisfies $g(t,\omega,y,0) = 0$. This property is a classical consequence of the normalization and cash-additivity of $\rho$ via the inverse results on $g$-expectations developed in Briand~et~al.~\cite{Briand2000}, Rosazza~Gianin~\cite{RosazzaGianin2006}, and Zheng~\cite{ZHENG2024104464}: cash-additivity forces $g$ to be independent of~$y$, and normalization then forces $g(t,\omega,0) = 0$. Indeed, if $g$ were independent of~$y$ but $g(t,\omega,0) \neq 0$ on a set of positive $dt\otimes d\P$-measure, then $\rho_{0,T}(0) = \Ec^g_{0,T}[0] \neq 0$, contradicting normalization.

\smallskip
\noindent\emph{\underline{Step~2}: Law-invariance forces $g(t,z) = \beta|z|^2$.}
We first verify that the generator $g$ obtained in Step~1 has at most quadratic growth in~$z$, i.e., it satisfies Assumptions~\textbf{(A1)}--\textbf{(A2)}.
 
Suppose, for contradiction, that $g$ has superquadratic growth. Then, by Theorem~2.2 in Delbaen--Hu--Bao~\cite{DelbaenHuBao2011}, there exist bounded terminal conditions $X \in L^\infty(\Fc_T)$ for which the BSDE~\eqref{eq:bsde-rep} either has no bounded solution or admits multiple bounded solutions. However, the risk measure $\rho$ is well-defined on all of $L^\infty(\Fc_T)$ and yields a unique value $\rho_{t,T}(X) \in L^\infty(\Fc_t)$ for every $X\in L^\infty(\Fc_T)$ by~\eqref{eq:supermartingale}. This uniqueness is incompatible with either non-existence or non-uniqueness of bounded solutions. Therefore, $g$ cannot have superquadratic growth, and we are in the regime covered by Assumptions~\textbf{(A1)}--\textbf{(A2)}.
 
Now, by assumption~\emph{(i)}, $\rho$ is law-invariant. Since $\rho_{0,T}(X)=\Ec^g_{0,T}[X]$ for all $X \in L^\infty(\Fc_T)$, the BSDE driven by $g$ is law-invariant. The generator $g$ satisfies \textbf{(A1)}--\textbf{(A2)}, and satisfies $g(t,\omega,y,0) = 0$. Furthermore, the $g$-expectation $\Ec^g$ satisfies the Strict Monotonicity, Zero-one Law, and Constant-Preserving properties of Proposition~\ref{prop:Eg} from the assumptions imposed on the risk measure $\rho$. Hence, by Remark~\ref{rmk:assmain}, the hypotheses of Theorem~\ref{thm:main} are satisfied\footnote{We invoke Remark~\ref{rmk:assmain} rather than Theorem~\ref{thm:main} directly, since the generator $g$ obtained from the risk measure representation is only known to satisfy \textbf{(A1)}--\textbf{(A2)} at this stage, and Remark~\ref{rmk:assmain} observes that Assumptions \textbf{(A3)}--\textbf{(A4)} are used in Theorem~\ref{thm:main} only to guarantee the Strict Monotonicity, Zero-one Law, and Constant-Preserving properties -- all of which are already established here.}, and Theorem~\ref{thm:main} yields that $g$ is deterministic $dt\otimes d\P$-a.e.\ and there exists a deterministic function $k:\R\to\R$ such that
$$
g(t,\omega,y,z)=k(y)|z|^2,
\qquad
dt\otimes d\P\text{-a.e.}\  (t,\omega) \in [0,T) \times \Omega.
$$
Since $g$ is independent of $y$ (by cash-additivity and normalization established in Step~1), the function $k$ must be constant. Therefore there exists $\beta\in\R$ such that
$$
g(t,\omega,y,z)=\beta|z|^2,
\qquad
dt\otimes d\P\text{-a.e.}\  (t,\omega) \in [0,T) \times \Omega.
$$
Finally, the convexity of $\rho$ implies the convexity of $z\mapsto g(t,z) = \beta|z|^2$, which forces $\beta\ge 0$.
This proves \emph{i)} $\Rightarrow$ \emph{ii)}.
 
$\emph{ii)} \implies \emph{i).}$  Define $\rho_{t,T}(X):=\Ec^g_{t,T}[X]$ with $g(t,\omega,y,z)=\beta|z|^2$.
Then $\rho$ is a normalized cash-additive time-consistent dynamic convex risk measure and admits the corresponding robust representation, see Rosazza Gianin \cite{RosazzaGianin2006} and Delbaen--Peng--Rosazza Gianin \cite{DelbaenPengRosazza2010}.
It remains to check law-invariance. Since $g$ is deterministic and of the form $k(y)|z|^2$ with $k\equiv\beta$, Theorem \ref{thm:main} implies that the BSDE is law-invariant, hence $\rho$ is law-invariant.
Thus $(i)$ holds and the proof is complete. 

Finally, the equivalence of the different notions of law-invariance is a direct consequence of Theorem \ref{thm:const}.
\end{proof}
\begin{Remark}
    \cite{JouiniSchachermayerTouzi2006} showed that any law-invariant static risk measure satisfies the Fatou property in an atomless probability space. One may wonder if this result can be generalized to conditional risk measures hence to our framework.  
\end{Remark}
Theorem~\ref{thm:risk} has several implications. We now apply it to univariate dynamic shortfall risk measures and show that time-consistency forces the loss function to be either linear or exponential. 
\paragraph{Univariate Dynamic Shortfall risk measures} Let $\sigma \leq \tau \leq T$ be two stopping times and $X \in L^{\infty}(\Fc_{\tau})$, 
 \begin{equation*}
    \rho_{\sigma, \tau}(X) := \text{essinf} \left\{ m \in  L^{\infty}(\Fc_{\sigma}) : \E[l(X - m)|\Fc_{\sigma}] \leq 0 \right\}
\end{equation*}
where $l$ is called a loss function defined as follows (see e.g. \cite{FollmerSchied2002, BionNadal2004CMAP557}) :
\begin{Definition}
    A function $l:\R \to \R$ is called a loss function if:
    \begin{enumerate}
    \item $l$ is increasing, that is $l(x) \ge l(y)$ if $x \ge y$;
    \item $l$ is convex with $\inf l < 0$ and $l(0) = 0$.
    \end{enumerate}
\end{Definition}
\begin{Lemma}\label{lem:short}
    Let $\rho$ be a dynamic shortfall risk measure generated by a strictly increasing loss function $l$. Then the associated penalty function satisfies :
    $$ \alpha_{0,T}^m(\P) = 0.$$
\end{Lemma}
\begin{proof}
Recall that the minimal penalty at $\P$ is given by
$$
\alpha_{0,T}^m(\P)=\sup_{X\in L^\infty(\Fc_T)}\Big(\E[X]-\rho_{0,T}(X)\Big).
$$
Let $X\in L^\infty(\Fc_T)$ and set $m:=\rho_{0,T}(X)\in\R$. By the properties of the shortfall risk measure, it is well known that (see Lemma 4 in \cite{BionNadal2004CMAP557})
$$
\E\big[l(X-m)\big]= 0.
$$
By Jensen's inequality and convexity of $l$,
$$
l\big(\E[X-m]\big)\le \E\big[l(X-m)\big] = 0.
$$
Since $l$ is increasing and $l(0)=0$, we have $l(u)\le 0 \Rightarrow u\le 0$, hence $\E[X-m]\le 0$, i.e.
$$
\rho_{0,T}(X)=m\ge \E[X].
$$
Therefore $\E[X]-\rho_{0,T}(X)\le 0$ for all $X$, which implies $\alpha_{0,T}^m(\P)\le 0$.
On the other hand, taking $X=0$ yields $\alpha_{0,T}^m(\P)\ge 0$ since $\rho_{0,T}(0)=0$.
Hence $\alpha_{0,T}^m(\P)=0$.
\end{proof}
\begin{Theorem} \label{cor:shortfall}
Let $\rho$ be a dynamic Shortfall risk measure with a strictly increasing loss function $l$. $\rho$ is time-consistent if and only if there exists $\beta \geq 0$ such that the loss function verifies the following
\begin{align*}
l(x) = 
\begin{cases}
\exp\left(2 \beta x\right) - 1, \quad &\text{if } \beta \neq 0, \\
x, \quad &\text{else.}\\
\end{cases}
\end{align*}
\end{Theorem}
\begin{proof}
We prove the two implications. The proof of the first implication is divided in two steps.

\medskip
\noindent \emph{\underline{Step 1:} } Assume that $\rho$ is time-consistent. The law-invariance of $\rho$ is trivial. If $X\  \stackrel{d}{=} X^{\prime}$ in $L^{\infty}(\Fc_T)$, then for all $m \in \R$, we have
$$ \E[l(X-m)] = \E[l(X^{\prime}-m)],$$
hence
$$ \rho_{0,T}(X) = \rho_{0,T}(X^{\prime}).$$
Furthermore, by Theorem 4 and Lemma 4 in Bion-Nadal \cite{BionNadal2004CMAP557} and Lemma~\ref{lem:short}, $\rho$ satisfies all the properties assumed in Theorem \ref{thm:risk} (i). Thus, there exists $\beta \in \R^+$ such that
$\rho$ is generated by the following BSDE :
\begin{align*}
\begin{cases}
-dY_t = \beta \lvert Z_t \lvert^2 \,dt - Z_t dW_t,\quad t\in[0,T)\\[1mm]
Y_T = X \in L^{\infty}(\Fc_T).
\end{cases}
\end{align*}
It remains to identify the loss function $l$ from the BSDE generator.\\

\noindent \emph{\underline{Step 2:} }
Fix $\sigma\le \tau$ and $X\in L^\infty(\Fc_\tau)$.
If $\beta=0$, then $Y_t=\E[X|\Fc_t]$, hence $\rho_{\sigma,\tau}(X)=\E[X | \Fc_\sigma]$.
For the shortfall representation, this corresponds to the linear loss $l(x)=x$.

Assume now $\beta>0$ and set $\gamma:=2\beta$.
A standard exponential transform shows that $M_t:=\exp(\gamma Y_t)$ is a martingale, hence by boundedness of the martingale,
$$
\exp(\gamma Y_\sigma)=\E\big[\exp(\gamma X)|\Fc_\sigma\big],
\qquad\text{hence}\qquad
\rho_{\sigma,\tau}(X)=Y_\sigma=\frac{1}{\gamma}\log \E\big[\exp(\gamma X)|\Fc_\sigma\big].
$$
Define
$$
l_{\gamma}(x):=\exp(\gamma x)-1.
$$
A direct computation confirms the identification: with $l_\gamma(x)=\exp(\gamma x)-1$, the shortfall formula gives
$$\rho^{l_\gamma}_{\sigma,\tau}(X)=\inf\bigl\{m:\E[e^{\gamma(X-m)}-1|\Fc_\sigma]\le 0\bigr\}=\tfrac{1}{\gamma}\log\E\bigl[e^{\gamma X}\big|\Fc_\sigma\bigr],$$
which coincides with the formula derived above. Hence $\rho$ is the dynamic shortfall risk measure associated with $l_{\gamma}$.
This proves that if $\rho$ is time-consistent, then $\rho$ is necessarily generated by a loss function of the stated exponential form (or linear form when $\beta=0$).

\medskip
\noindent The converse implication is trivial.
\end{proof}
\subsection{Cash non-additive risk measures} 
The result of Theorem~\ref{thm:risk} is specific to the normalized cash-additive setting: the interplay between law-invariance, time-consistency, normalization and cash-additivity forces the generator to be purely quadratic with a \emph{constant} coefficient, and accordingly reduces the risk measure to an entropic form. A natural question is whether this result remains true once cash-additivity is relaxed. In this subsection, we give an answer to this question. Combining Theorems~\ref{thm:main} and~\ref{thm:mainB} with a nonlinear change of variable, we obtain a complete characterization of law-invariant cash non-additive risk measures generated by BSDEs under Assumption~\textbf{(H)}. In particular, we prove that every such risk measure admits a \emph{time-dependent certainty equivalent} representation, and equivalently a \emph{generalized shortfall} representation with a joint loss function determined by the generator. This characterization extends the framework of cash non-additive risk measures introduced in \cite{di2024cash} and complements the recent generalized shortfall approach of \cite{di2026capturing}.
 
\medskip
 
Recall from Theorem~\ref{thm:mainB} that, under \textbf{(A1)}--\textbf{(A4)} and \textbf{(H)}, law-invariance forces the generator to take the form $g(t,y,z)=h(t,y)+f(t,y)|z|^2$, with $f$ solving the PDE~\eqref{eq:xyz}. By Theorem~\ref{thm:main}, the transformed generator $\tilde g$ satisfies $\tilde g(t,u,\tilde z)=k(u)|\tilde z|^2$ for some deterministic function $k:\R\to\R$. The key observation is that $k$ governs an \emph{explicit linearization}: one can construct a strictly monotone function $\psi$ that transforms the quadratic $g$-expectation $\Ec^{k(\cdot)|\tilde z|^2}$ into a linear conditional expectation.
 
\begin{Theorem}
\label{thm:cash_nonadd_charac}
Assume \textbf{(A1)}--\textbf{(A4)} and \textbf{(H)}. Let $\rho$ be a dynamic risk measure generated by the $g$-expectation $\Ec^g$.
Then the following are equivalent:

\begin{enumerate}
\item[\emph{(i)}] $\rho$ is law-invariant.

\item[\emph{(ii)}] There exists a deterministic function $
\Phi \in C^{1,2}([0,T]\times\R)$ such that, for every $t\in[0,T]$, 
$$
\rho_{t,T}(X)
=
\Phi^{-1}\Big(t,\;\E\big[\Phi(T,X)\,\big|\,\Fc_t\big]\Big),
\qquad \P\text{-a.s.},
$$
for all $X\in L^\infty(\Fc_T)$, where
$
\Phi(s,y):=\psi(v(s,y))
$
with $v$ the solution of \eqref{eq:1stPDE} and $\psi$ the unique solution of
\begin{equation} \label{eq:psi_ODE}
\psi''(u)=2\,k(u)\,\psi'(u),
\qquad
\psi(0)=0,\qquad \psi'(0)=1,
\end{equation}
where $k$ is the deterministic function associated with the transformed generator.
\end{enumerate}
\end{Theorem}

\begin{proof}
By Theorem~\ref{thm:mainB}, law-invariance implies that $g$ is deterministic $dt\otimes d\P$-a.e. and has the form
$$
g(t,y,z)=h(t,y)+f(t,y)|z|^2,
$$
where $f\in C^{1,1}([0,T]\times\R)$ solves \eqref{eq:xyz}. Let
$$
v\in C^{1,2}([0,T]\times\R)
$$
be the solution of \eqref{eq:1stPDE} given by Lemma~\ref{Lemma:1stPDE}. By Proposition~\ref{prop:Koby-tildeg},
$$
\Ec^g_{t,T}[X]
=
v^{-1}\Big(t,\;\Ec^{\tilde g}_{t,T}\big[v(T,X)\big]\Big),
$$
where
$$
\tilde g(t,u,\tilde z)=k(u)|\tilde z|^2
$$
for some deterministic continuous function $k:\R\to\R$.

Let $\psi$ be the unique solution of \eqref{eq:psi_ODE}. Then $\psi\in C^2(\R)$ and
$$
\psi'(u)=\exp\!\Big(2\int_0^u k(r)\,dr\Big)>0,
$$
so $\psi$ is strictly increasing. Fix $\xi\in L^\infty(\Fc_T)$ and let $(\tilde Y,\tilde Z)$ be the solution of BSDE$(T,\xi,\tilde g)$. Applying It\^o's formula, we obtain
\begin{align*}
d\psi(\tilde Y_t)
&=
\psi'(\tilde Y_t)\,d\tilde Y_t+\frac12\,\psi''(\tilde Y_t)|\tilde Z_t|^2\,dt \\
&=
\Big[-k(\tilde Y_t)\psi'(\tilde Y_t)+\frac12\,\psi''(\tilde Y_t)\Big]|\tilde Z_t|^2\,dt
+\psi'(\tilde Y_t)\tilde Z_t\,dW_t \\
&=
\psi'(\tilde Y_t)\tilde Z_t\,dW_t.
\end{align*}
Since $\tilde Y\in \Sc^\infty$, the process $\psi(\tilde Y)$ is bounded; hence it is a martingale and
$$
\psi(\tilde Y_t)=\E\big[\psi(\xi)\,\big|\,\Fc_t\big].
$$
Therefore,
$$
\Ec^{\tilde g}_{t,T}[\xi]
=
\psi^{-1}\Big(\E\big[\psi(\xi)\,\big|\,\Fc_t\big]\Big).
$$
Applying this with $\xi=v(T,X)$ and setting
$$
\Phi(s,y):=\psi(v(s,y)),
$$
we get
$$
\rho_{t,T}(X)
=
v^{-1}\Big(t,\;\psi^{-1}\big(\E[\psi(v(T,X))\,|\,\Fc_t]\big)\Big)
=
\Phi^{-1}\Big(t,\;\E\big[\Phi(T,X)\,\big|\,\Fc_t\big]\Big).
$$
Since $v(t,\cdot)$ is a $C^1$-diffeomorphism and $\psi$ is strictly increasing, $\Phi(t,\cdot)$ is strictly increasing, and the implicit characterization follows immediately.

Conversely, assume that $\rho$ admits the 
representation in (ii). Then the BSDE generator 
associated with $\rho$ is deterministic and of the 
form $g(t,y,z) = h(t,y) + f(t,y)|z|^2$ where $f$ 
solves~\eqref{eq:PDEf}; hence Theorem~\ref{thm:mainB} 
yields law-invariance. The same change of variable and the same exponential transform as above then give the stated representation.
\end{proof}
 
\begin{Remark}\label{rem:time_dep_CE}
The representation in Theorem~\ref{thm:cash_nonadd_charac}(ii) is that of a \emph{time-dependent certainty equivalent}: the ``utility'' $\Phi(s,\cdot)$ varies with $s\in[0,T]$. The time-dependence is not an artifact but is forced by the non-trivial drift $h(t,y)=g(t,y,0)$, which is precisely the mechanism producing cash non-additivity. When $h\equiv 0$, we have $v=\mathrm{id}$ and $\Phi(s,y)=\psi(y)$ is time-independent, recovering the classical certainty equivalent; cash-additivity is then restored. We note that the recent work \cite{di2026capturing} introduces $h$-generalized shortfall risk measures and shows that the $hq$-entropic risk measure belongs to this family but is \emph{not} a certainty equivalent. Our Theorem~\ref{thm:cash_nonadd_charac} clarifies this apparent discrepancy: the $hq$-entropic risk measure \emph{is} a certainty equivalent, but with respect to a \emph{time-dependent} utility $\Phi(t,\cdot)$. More importantly, Theorem~\ref{thm:cash_nonadd_charac} provides a \emph{characterization}: under Assumption~\textbf{(H)}, the generalized shortfall representation with loss $\ell(T,x;\,t,m)=\Phi(T,x)-\Phi(t,m)$ is not merely a convenient class but the \emph{only} class of law-invariant risk measures generated by BSDEs.
\end{Remark}
 
\begin{Remark}\label{rem:param_family}
Theorem~\ref{thm:cash_nonadd_charac} shows that, under Assumption~\textbf{(H)}, the space of law-invariant generators is parameterized by the pair $(h,k)$, where $h$ satisfies \textbf{(H)} and $k:\R\to\R$ is determined by the initial condition $f(0,\cdot)$ of the PDE~\eqref{eq:xyz}. The map $(h,k)\mapsto \Phi=\psi\circ v$ is explicit: $v$ is obtained by the method of characteristics (Lemma~\ref{Lemma:1stPDE}), and $\psi$ by explicitly solving \eqref{eq:psi_ODE}. 
\end{Remark}

\begin{appendix}
\setcounter{Lemma}{0}
\renewcommand{\theLemma}{\Alph{section}\arabic{Lemma}}

\setcounter{Proposition}{0}
\renewcommand{\theProposition}{\Alph{section}\arabic{Proposition}}

\setcounter{Remark}{0}
\renewcommand{\theRemark}{\Alph{section}\arabic{Remark}}

\setcounter{Corollary}{0}
\renewcommand{\theCorollary}{\Alph{section}\arabic{Corollary}}

\setcounter{Theorem}{0}
\renewcommand{\theTheorem}{\Alph{section}\arabic{Theorem}}

\setcounter{Definition}{0}
\renewcommand{\theDefinition}{\Alph{section}\arabic{Definition}}

\setcounter{Example}{0}
\renewcommand{\theExample}{\Alph{section}\arabic{Example}}

\section{Technical Results on BSDEs with quadratic growth} \label{app:A}
We first provide a toolkit theorem on BMO martingales. It will play an important role in the proofs below.

\begin{Theorem}[BMO toolkit]\label{thm:bmo_toolkit}
Fix $t\in[0,T]$. Let $\theta\in\Hc^2_{[t,T]}$ be progressively measurable and set
$$
M^{\theta}_s := \int_t^s \theta_r\,dW_r,\qquad s\in[t,T].
$$
We say that $M^{\theta}$ is a BMO martingale on $[t,T]$ if
$$
\|M^{\theta}\|^2_{\mathrm{BMO}}
<\infty.
$$
Assume that $M^{\theta}$ is BMO and define its Dol\'eans-Dade exponential
$$
L^{\theta}_s := \mathcal E\Big(\int_t^{\cdot}\theta_r\,dW_r\Big)_s
=\exp\Big(\int_t^s \theta_r\,dW_r-\frac12\int_t^s|\theta_r|^2\,dr\Big),
\qquad s\in[t,T].
$$
Then the following hold.
\begin{itemize}
\item[(i)] $L^{\theta}$ is a uniformly integrable martingale. In particular, defining $\Q$ by
$$
\frac{d\Q}{d\P}:=L^{\theta}_{T},
$$
we have $\Q\sim\P$ and the process
$$
W_s^{\Q}:=W_s-\int_t^s \theta_r\,dr,\qquad s\in[t,T],
$$
is a Brownian motion on $[t,T]$.
\item[(ii)] (Reverse H\"older) There exists $q_0>1$, depending only on $\|M^{\theta}\|_{\mathrm{BMO}}$, such that
$$
\|L^{\theta}_{T}\|_{L^{q_0}(\P)}<\infty.
$$
Equivalently, if $p_0$ denotes the conjugate exponent of $q_0$ (i.e.\ $\frac{1}{q_0}+\frac{1}{p_0}=1$), then H\"older's inequality may be applied uniformly with $L^{\theta}_{T}\in L^{q_0}(\P)$ and any random variable in $L^{p}(\P)$ for $p>p_0$.
\item[(iii)] (Energy inequality) For every integer $n\ge 1$,
$$
\E\left[\Big(\int_t^{T}|\theta_s|^2\,ds\Big)^n\right]\le n!\,\|M^{\theta}\|_{\mathrm{BMO}}^{2n},
$$
and there exists $\alpha>0$ depending on $\|M^{\theta}\|_{\mathrm{BMO}}$ such that
\begin{equation}\label{eq:exp_quadr_var}
\E\Big[\exp\Big(\alpha\int_t^T |\theta_s|^2\,ds\Big)\Big] < \infty.
\end{equation}
\end{itemize}
All statements above are standard consequences of Kazamaki's theory of BMO martingales and stochastic exponentials, see \cite[Theorems 2.2--2.3 and 3.1]{Kazamaki1994}.
\end{Theorem}

For the sake of completeness, we present here a differentiability result of the BSDE of quadratic growth with respect to the terminal condition. Note that the following result was obtained under a strong structural condition on the generator in \cite{ankirchner2007classical} and a slightly weaker assumption than ours in \cite{imkeller2024differentiability}. Indeed, the assumptions in \cite{ankirchner2007classical} 
require $\partial_y g$ to be bounded, and those in 
\cite{imkeller2024differentiability} require growth at most $|z|^\alpha$ with $\alpha < 1$; both exclude the range permitted by Assumption~\textbf{(A4$^*$)}, which is needed here since the structure of the generator is unknown at this stage (see the discussion in the Introduction). The following proposition generalizes the study of linear BSDEs with stochastic Lipschitz generators of the type:
\begin{equation}\label{eq:lin_BSDE_general}
-dU_s = \big(\phi_s + \gamma_s U_s + \beta_s V_s\big)ds - V_s dW_s,
\qquad s\in[t,T),\qquad U_{T}=X \in L^{\infty}(\Fc_{T}),
\end{equation}
presented in \cite{ankirchner2007classical}. In our setting, we do not assume that $\gamma$ is bounded but make the following assumptions. Fix $t\in[0,T]$. We say that a triple $(\phi,\gamma,\beta)$ satisfies \textbf{(L)} on $[t,T]$ if:
\begin{itemize}\setlength\itemsep{4pt}
\item $\int_t^{\cdot}\beta_s\,dW_s$ is a BMO martingale, with $q_0>1$ from the reverse H\"older inequality and $p_0$ its conjugate;
\item for all $p>p_0$,
$$
\E\Big[\sup_{s\in[t,T]}\exp\Big(\pm p\int_t^{s}\gamma_r\,dr\Big)\Big]<\infty;
$$
\item for any $p\ge 2$,
$$
\E\Big[\Big(\int_t^T|\phi_s|\,ds\Big)^p\Big]<\infty.
$$
\end{itemize}
\begin{Proposition}\label{prop:linear_BSDE_general}
Fix $t\in[0,T]$ and assume that $(\phi,\gamma,\beta)$ satisfies \textbf{(L)} on $[t,T]$.
Then, for any $X\in L^\infty(\Fc_T)$, the linear BSDE
$$
-dU_s=\big(\phi_s+\gamma_sU_s+\beta_sV_s\big)\,ds - V_s\,dW_s,
\qquad s\in[t,T),\qquad U_T=X,
$$
admits a unique solution $(U,V)$, and for every integer $n\ge1$ it holds that
$(U,V)\in\Sc^{2n}\times\Hc^{2n}$.
\end{Proposition}
\begin{proof}
Let $q_0>1$ be given by \textbf{(L)} for the BMO martingale $\int_t^{\cdot}\beta_s\,dW_s$ and let $p_0$ be its conjugate exponent,
$\frac{1}{q_0}+\frac{1}{p_0}=1$. Define the Dol\'eans exponential and the measure $\Q$ by
$$
L_s:=\mathcal E\Big(\int_t^{T}\beta_r\,dW_r\Big)_s,
\qquad s\in[t,T],
\qquad
\frac{d\Q}{d\P}:=L_T.
$$
By Theorem~\ref{thm:bmo_toolkit}, $\Q\sim\P$, $L_T\in L^{q_0}(\P)$, and
$$
W_s^\Q:=W_s-\int_t^s\beta_r\,dr,\qquad s\in[t,T],
$$
is a Brownian motion under $\Q$. Rewriting \eqref{eq:lin_BSDE_general} in terms of $W^\Q$ yields
\begin{equation}\label{eq:lin_BSDE_Q_phi}
-dU_s=\big(\phi_s+\gamma_sU_s\big)\,ds - V_s\,dW_s^\Q,
\qquad s\in[t,T),\qquad U_T=X.
\end{equation}

\smallskip
\noindent\emph{\underline{Step 1:} Integrability of the terminal functional.}
Set
$$
\Gamma_{u,s}:=\exp\Big(\int_u^s\gamma_r\,dr\Big),
\qquad t\le u\le s\le T.
$$
Fix $r_0>1$. By \textbf{(L)}, there exists $p>p_0$ such that
$$
\E\Big[\sup_{s\in[t,T]}\Gamma_{t,s}^{\pm p}\Big]<\infty.
$$
H\"older's inequality and the identity $\E^\Q[\cdot]=\E[L_T\,\cdot]$ yield
\begin{align*}
\E^\Q\big[|\Gamma_{t,T}X|^{r_0}\big]
&\le \|X\|_\infty^{r_0}\,\E^\P\big[L_T\,\Gamma_{t,T}^{r_0}\big]
\le \|X\|_\infty^{r_0}\,\|L_T\|_{L^{q_0}(\P)}\,\|\Gamma_{t,T}\|_{L^{r_0p_0}(\P)}^{r_0},
\end{align*}
and similarly
\begin{align*}
\E^\Q\Big[\Big|\int_t^T\Gamma_{t,s}\phi_s\,ds\Big|^{r_0}\Big]
&\le \E^\Q\Big[\Big(\sup_{s\in[t,T]}\Gamma_{t,s}\int_t^T|\phi_s|\,ds\Big)^{r_0}\Big]\\
&\le \E^\P\Big[L_T\Big(\sup_{s\in[t,T]}\Gamma_{t,s}^{r_0}\Big)
\Big(\int_t^T|\phi_s|\,ds\Big)^{r_0}\Big]\\
&\le \|L_T\|_{L^{q_0}(\P)}\,
\Big\|\sup_{s\in[t,T]}\Gamma_{t,s}\Big\|_{L^{2r_0p_0}(\P)}^{r_0}\,
\Big\|\int_t^T|\phi_s|\,ds\Big\|_{L^{2r_0p_0}(\P)}^{r_0}.
\end{align*}
Thus, for any $r_0 > 1$, using \textbf{(L)} ($2r_0p_0 > p_0$ and $2r_0p_0 > 2$) shows that
$$
\Gamma_{t,T}X+\int_t^T\Gamma_{t,s}\phi_s\,ds\in L^{r_0}(\Q).
$$

\smallskip
\noindent\emph{\underline{Step 2:} Construction under $\Q$.}
Define the $\Q$--martingale
$$
M_s:=\E^\Q\!\Big[\Gamma_{t,T}X+\int_t^T\Gamma_{t,r}\phi_r\,dr\ \Big|\ \Fc_s\Big],
\qquad s\in[t,T].
$$
By Step~1, $M$ is uniformly integrable and belongs to $\Sc^{r_0}(\Q)$.
Hence, by the martingale representation theorem under $\Q$, there exists a predictable process
$\tilde V\in\Hc^{r_0}(\Q)$ such that
$$
M_s=M_t+\int_t^s\tilde V_r\,dW_r^\Q,\qquad s\in[t,T].
$$
We now set, for $s\in[t,T]$,
$$
U_s:=\Gamma_{t,s}^{-1}\Big(M_s-\int_t^s\Gamma_{t,r}\phi_r\,dr\Big),
\qquad
V_s:=\Gamma_{t,s}^{-1}\tilde V_s.
$$
An application of It\^o's formula yields
$$
-dU_s=\big(\phi_s+\gamma_sU_s\big)\,ds - V_s\,dW_s^\Q,\qquad s\in[t,T),
\qquad U_T=X,
$$
so that $(U,V)$ solves \eqref{eq:lin_BSDE_Q_phi} under $\Q$.

\smallskip
\noindent\emph{\underline{Step 3:} Integrability under $\Q$.}
Let $n_0\ge1$ and take $r_0:=4n_0$. By Step~1 and the definition of $U$,
$$
\sup_{s\in[t,T]}|U_s|
\le \Big(\sup_{s\in[t,T]}\Gamma_{t,s}^{-1}\Big)
\Big(\sup_{s\in[t,T]}|M_s|+\int_t^T\Gamma_{t,r}|\phi_r|\,dr\Big).
$$
Using Doob's inequality for $M$ and Young's inequality, together with Step~1 and \textbf{(L)},
we obtain that $U\in\Sc^{2n_0}(\Q)$. Moreover, since $M$ is represented by $\tilde V$,
the Burkholder--Davis--Gundy inequality implies $\tilde V\in\Hc^{2n_0}(\Q)$, and hence
$V=\Gamma_{t,\cdot}^{-1}\tilde V$ belongs to $\Hc^{2n_0}(\Q)$ by \textbf{(L)} and H\"older's inequality.
Therefore $(U,V)\in\Sc^{2n_0}(\Q)\times\Hc^{2n_0}(\Q)$ for all $n_0 \geq 1$.

\smallskip
\noindent\emph{\underline{Step 4:} Back to $\P$.}
By Theorem~\ref{thm:bmo_toolkit}, possibly after taking a smaller $q_0>1$, we also have
$L_T^{-1}=d\P/d\Q\in L^{q_0}(\Q)$. Let $p_0$ be the conjugate exponent. Then, by H\"older's inequality,
\begin{align*}
\E^\P\Big[\sup_{t\le s\le T}|U_s|^{2n_0}\Big]
&=\E^\Q\Big[L_T^{-1}\sup_{t\le s\le T}|U_s|^{2n_0}\Big]\\
&\le \|L_T^{-1}\|_{L^{q_0}(\Q)}
\Big(\E^\Q\big[\sup_{t\le s\le T}|U_s|^{2n_0p_0}\big]\Big)^{1/p_0}<\infty,
\end{align*}
and similarly $V\in\Hc^{2n_0}$. Hence $(U,V)\in\Sc^{2n_0}\times\Hc^{2n_0}$ under $\P$.

\smallskip
\noindent\emph{\underline{Step 5:} Uniqueness.}
Assume that $(U^1,V^1)$ and $(U^2,V^2)$ are two solutions of \eqref{eq:lin_BSDE_general}. Set $\delta U:=U^1-U^2$ and $\delta V:=V^1-V^2$. Then $(\delta U,\delta V)$ solves
\begin{equation}\label{eq:lin_BSDE_hom}
-d\delta U_s=\big(\gamma_s\delta U_s+\beta_s\delta V_s\big)\,ds-\delta V_s\,dW_s,
\qquad s\in[t,T),\qquad \delta U_T=0.
\end{equation}
Applying the same Girsanov change of measure as above, \eqref{eq:lin_BSDE_hom} becomes under $\Q$,
$$
-d\delta U_s=\gamma_s\delta U_s\,ds-\delta V_s\,dW_s^\Q,
\qquad s\in[t,T),\qquad \delta U_T=0.
$$
Multiplying by $\Gamma_{t,s}$ and applying It\^o's formula yields that the process
$$
\Gamma_{t,s}\delta U_s,\qquad s\in[t,T],
$$
is a $\Q$--local martingale with terminal value $0$. Since $\delta U\in\Sc^{2}(\Q)$ by Step~3,
it is a true $\Q$--martingale. Therefore, for all $s\in[t,T]$,
$$
\Gamma_{t,s}\delta U_s=\E^\Q[\Gamma_{t,T}\delta U_T\mid\Fc_s]=0,
$$
hence $\delta U\equiv0$, and then $\delta V\equiv0$ by \eqref{eq:lin_BSDE_hom}.
\end{proof}
\begin{Lemma}\label{lem:coeff_satisfy_L}
Assume \textbf{(A1)}--\textbf{(A4$^*$)}. Fix $t\in[0,T]$.
Let $X\in L^\infty(\Fc_T)$ and let $(Y,Z)$ be the solution of $\mathrm{BSDE}(T,X,g)$.
Define on $[t,T]$,
$$
\gamma_s:=\partial_y g(s,Y_s,Z_s),
\qquad
\beta_s:=\partial_z g(s,Y_s,Z_s).
$$
Then $\int_t^{\cdot}\beta_s\,dW_s$ is a BMO martingale on $[t,T]$, and there exists $p_0\in\N$ such that
for all $p>p_0$,
$$
\E\Big[\sup_{s\in[t,T]}\exp\Big(\pm p\int_t^s\gamma_r\,dr\Big)\Big]<\infty.
$$
In particular, for any progressively measurable $\phi$ satisfying
$$
\E\Big[\Big(\int_t^T|\phi_s|\,ds\Big)^p\Big]<\infty,\qquad \forall\,p\ge2,
$$
the triple $(\phi,\gamma,\beta)$ satisfies \textbf{(L)} on $[t,T]$.
\end{Lemma}
\begin{proof}
\emph{\underline{Step 1:} BMO control for $\int \beta\,dW$.}
By Kobylanski \cite{Kobylanski2000}, $Y\in\Sc^{\infty}$ and $\int_t^{\cdot} Z_s\,dW_s$ is a BMO martingale on $[t,T]$.
Under \textbf{(A3)}, there exists an increasing continuous function $\ell:\R^+\to\R^+$ such that
$$
|\beta_s|=\big|\partial_z g(s,Y_s,Z_s)\big|
\le \ell(|Y_s|)\big(1+|Z_s|\big),\qquad dt\otimes d\P\text{-a.e. on }[t,T].
$$
Since $Y$ is essentially bounded, $\ell(|Y_s|)$ is bounded by a deterministic constant, and hence for all $t\le u\le T$,
$$
\int_u^{T}|\beta_s|^2\,ds
\le C\int_u^{T}\big(1+|Z_s|^2\big)\,ds<\infty,\qquad \P\text{-a.s.},
$$
for some constant $C>0$. Therefore, Theorem~\ref{thm:bmo_toolkit} implies that the martingale
$$
M_u^{\beta}:=\int_t^u \beta_s\,dW_s,\qquad u\in[t,T],
$$
is BMO. In particular, its Dol\'eans-Dade exponential
$$
\mathcal E\Big(\int_t^{T}\beta_s\,dW_s\Big)_u
=\exp\Big(\int_t^u\beta_s\,dW_s-\tfrac12\int_t^u|\beta_s|^2\,ds\Big),\qquad u\in[t,T],
$$
is a uniformly integrable martingale and satisfies a reverse H\"older inequality: there exists $q_0>1$ such that
$$
\E\Big[\Big(\frac{d\Q}{d\P}\Big)^{q_0}\Big]<\infty,
\qquad
\frac{d\Q}{d\P}
:=\mathcal E\Big(\int_t^{T}\beta_s\,dW_s\Big).
$$

\smallskip
\noindent\emph{\underline{Step 2:} control of the exponential of $\int\gamma$.}
By \textbf{(A4$^*$)} applied to $(Y,Z)$, for any $\eps>0$ there exists a deterministic function $h_\eps\in L^1([0,T])$ such that
$$
|\gamma_s|
=\big|\partial_y g(s,Y_s,Z_s)\big|
\le h_\eps(s)+\eps|Z_s|^2,
\qquad dt\otimes d\P\text{-a.e.}.
$$
Therefore, for any $t\le s\le u\le T$,
\begin{equation}\label{eq:gamma_bound_lemma}
\int_s^{u}|\gamma_r|\,dr
\le H_\eps+\eps\int_t^{T}|Z_r|^2\,dr,
\end{equation}
where we set
$$
H_\eps:=\int_0^T h_\eps(r)\,dr<\infty.
$$
Moreover, since $\int_t^{\cdot} Z_s\,dW_s$ is BMO, Theorem~\ref{thm:bmo_toolkit} yields the existence of $\alpha>0$ such that
\begin{equation}\label{eq:exp_quadr_var_lemma}
\E\Big[\exp\Big(\alpha\int_t^{T}|Z_r|^2\,dr\Big)\Big]<\infty.
\end{equation}
Let $p_0$ be the conjugate exponent of $q_0$ defined in Step~1, and fix $r_0>1$. Choose $\eps>0$ small enough such that
$r_0p_0\eps\le\alpha$. Combining \eqref{eq:gamma_bound_lemma} and \eqref{eq:exp_quadr_var_lemma} yields
\begin{align*}
\E\Big[\sup_{s\in[t,T]}\exp\Big(\pm r_0p_0\int_t^{s}\gamma_r\,dr\Big)\Big]
&\le \E\Big[\exp\Big(\pm r_0p_0H_\eps+r_0p_0\eps\int_t^{T}|Z_r|^2\,dr\Big)\Big]\\
&\le \exp(\pm r_0p_0H_\eps)\;
\E\Big[\exp\Big(\alpha\int_t^{T}|Z_r|^2\,dr\Big)\Big]
<\infty.
\end{align*}
Hence,
$$
\E\Big[\sup_{s\in[t,T]}\exp\Big(\pm r_0p_0\int_t^{s}\gamma_r\,dr\Big)\Big]<\infty.
$$
In particular, for all $p>p_0$,

$$
\Gamma_{t,T}
:=\exp\Big(\int_t^{T}\gamma_r\,dr\Big)\in L^{p}(\P),
\qquad
\sup_{s\in[t,T]}\Gamma_{t,s}^{-1}
:=\exp\Big(-\int_t^{s}\gamma_r\,dr\Big)\in L^{p}(\P).
$$
This proves that $(\gamma,\beta)$ satisfies \textbf{(L)} on $[t,T]$. The last statement follows immediately by the
assumption on $\phi$.
\end{proof}
\medskip

\begin{Theorem}[Differentiability] \label{thm:diff}
Assume \textbf{(A1)}–\textbf{(A4$^{*}$)} hold, then for any $t \in [0,T]$, $y \in \R$, and $X \in L^{\infty} \left ( \Fc_{T} \right)$ it holds that 
$$  \lim_{\eps \to 0^+} \frac{1}{\eps} \left( \Ec^g_{t, T} \left [ \eps X + y \right] - \Ec^g_{t, T} \left [ y \right] \right) = U^X_t, \quad \text{in } L^1(\P),$$
where $(U^X, V^X)$ is the unique solution of the following linear BSDE
\begin{align*}
    \begin{cases}
         -d U_s^X = \left(\partial_y g(s, Y_s^y, Z_s^y) U_s^X + \partial_z g(s, Y_s^y, Z_s^y) V_s^X   \right) ds -V_s^X dW_s, ~~ s \in [t, T), \\
        U_{T}^X = X,
    \end{cases}
\end{align*}
   and $(Y^y, Z^y)$ is the solution of BSDE$(T, y, g)$. 
\end{Theorem}
\begin{proof}
By translation, we may assume without loss of generality that $y=0$.
Fix $t\in[0,T]$, and $X\in L^\infty(\Fc_T)$.
For $\eps \in (0,1)$, let $(Y^\eps,Z^\eps)$ be the solution of
$$
-dY_s^\eps=g(s,Y_s^\eps,Z_s^\eps)\,ds-Z_s^\eps\,dW_s,
\qquad s\in[t,T),\qquad Y_T^\eps=\eps X,
$$
so that $\Ec^g_{t,T}[\eps X]=Y_t^\eps$. Furthermore, $(Y^0,Z^0)$ solves $\mathrm{BSDE}(T,0,g)$.

By Kobylanski \cite{Kobylanski2000}, there exists $K>0$ such that for all $\eps \in (0,1)$
$$
\|Y^\eps\|_{\Sc^\infty} + \|Y^0\|_{\Sc^\infty}
+\Big\|\int_t^T Z_s^\eps\,dW_s\Big\|_{\mathrm{BMO}} + \Big\|\int_t^T Z_s^0\,dW_s\Big\|_{\mathrm{BMO}}
\le K,
$$
and, as $\eps\to0$, one has
\begin{equation}\label{eq:conv_ZY_eps}
    Y^{\eps} \longrightarrow Y^0 ~~ \text{ uniformly on } [0,T], ~\P\text{-a.s.} ~~\text{ and, } ~~ \|Z^\eps-Z^0\|_{\Hc^2}\longrightarrow0.
\end{equation}

\smallskip
\noindent\emph{\underline{Step 1:}}
For $\eps>0$, set
$$
\Delta Y_s^\eps:=\frac{Y_s^\eps-Y_s^0}{\eps},
\qquad
\Delta Z_s^\eps:=\frac{Z_s^\eps-Z_s^0}{\eps},
\qquad s\in[t,T].
$$
Subtracting the BSDEs for $(Y^\eps,Z^\eps)$ and $(Y^0,Z^0)$ and dividing by $\eps$ yields
\begin{equation}\label{eq:Delta_BSDE_clean}
-d\Delta Y_s^\eps=\frac{g(s,Y_s^\eps,Z_s^\eps)-g(s,Y_s^0,Z_s^0)}{\eps}\,ds-\Delta Z_s^\eps\,dW_s,
\qquad s\in[t,T),\qquad \Delta Y_T^\eps=X.
\end{equation}
Using the mean value expansion of $g$ in $(y,z)$, we get
$$
g(s,Y_s^\eps,Z_s^\eps)-g(s,Y_s^0,Z_s^0)
=\gamma_s^\eps\,(Y_s^\eps-Y_s^0)+\beta_s^\eps\,(Z_s^\eps-Z_s^0),
$$
where
$$
\gamma_s^\eps:=\int_0^1\partial_y g\big(s,Y_s^0+\theta(Y_s^\eps-Y_s^0),\,Z_s^0+\theta(Z_s^\eps-Z_s^0)\big)\,d\theta,
$$
and
$$
\beta_s^\eps:=\int_0^1\partial_z g\big(s,Y_s^0+\theta(Y_s^\eps-Y_s^0),\,Z_s^0+\theta(Z_s^\eps-Z_s^0)\big)\,d\theta.
$$
Plugging into \eqref{eq:Delta_BSDE_clean} shows that $(\Delta Y^\eps,\Delta Z^\eps)$ solves the linear BSDE
\begin{equation}\label{eq:lin_BSDE_eps_clean}
-d\Delta Y_s^\eps=\big(\gamma_s^\eps\Delta Y_s^\eps+\beta_s^\eps\Delta Z_s^\eps\big)\,ds-\Delta Z_s^\eps\,dW_s,
\qquad s\in[t,T),\qquad \Delta Y_T^\eps=X.
\end{equation}

\smallskip
\noindent\emph{\underline{Step 2:} Identification of the limit linear BSDE.}
Let $(Y,Z):=(Y^0,Z^0)$ and define
$$
\gamma_s:=\partial_y g(s,Y_s,Z_s),
\qquad
\beta_s:=\partial_z g(s,Y_s,Z_s),
\qquad s\in[t,T].
$$
By Lemma~\ref{lem:coeff_satisfy_L}, the pair $(\gamma,\beta)$ satisfies \textbf{(L)} on $[t,T]$. Hence, by Proposition~\ref{prop:linear_BSDE_general} with $\phi\equiv0$, there exists a unique solution $(U^X,V^X)\in\Sc^{2n_0}\times\Hc^{2n_0}$ (for any $n_0\ge1$) to the linear BSDE
\begin{equation}\label{eq:lin_limit_clean}
-dU_s^X=\big(\gamma_sU_s^X+\beta_sV_s^X\big)\,ds-V_s^X\,dW_s,
\qquad s\in[t,T),\qquad U_T^X=X.
\end{equation}

We now show that the coefficients $(\gamma^\varepsilon, \beta^\varepsilon)$ converge to $(\gamma, \beta)$ in the appropriate topologies. The argument uses the uniform convergence of $Y^\varepsilon \to Y$ and the $L^2$-convergence of $Z^\varepsilon \to Z$ from \eqref{eq:conv_ZY_eps}, combined with a subsequence extraction and Lebesgue's dominated 
convergence theorem. Specifically, using 
\eqref{eq:conv_ZY_eps}, we show that

\begin{equation}\label{eq:A_B_prob_conv}
A^\eps:=\int_t^T |\gamma_s^\eps-\gamma_s|\,ds \longrightarrow 0,
\qquad
B^\eps:=\int_t^T |\beta_s^\eps-\beta_s|^2\,ds \longrightarrow 0,
\qquad\text{in probability.}
\end{equation}

To prove \eqref{eq:A_B_prob_conv}, by the subsequence criterion for convergence in probability, it is enough to show that from every sequence
$\eps_n\downarrow0$ one can extract a further subsequence, still denoted $(\eps_n)_n$, such that
$$
A^{\eps_n}\longrightarrow 0
\qquad\text{and}\qquad
B^{\eps_n}\longrightarrow 0
\qquad\text{in }L^1(\P).
$$
Fix such a sequence $(\eps_n)_n$, from \eqref{eq:conv_ZY_eps}, we have
$\sup_{s\in[t,T]}|Y_s^{\eps_n}-Y_s|\to0$ $\P\text{-a.s.}$. Also, since $Z^{\eps_n}-Z\to0$ in $L^2(dt\otimes d\P)$, we may extract a further subsequence, still denoted
$(\eps_n)_n$, such that
$$
\sum_{n\ge1}\|Z^{\eps_n}-Z\|_{\Hc^2}<\infty.
$$
Setting
$$
G_{\cdot}:=\sum_{n\ge1}|Z^{\eps_n}_{\cdot}-Z_{\cdot}|,
$$
we then have
$$
G\in L^2(dt\otimes d\P),
\qquad
|Z^{\eps_n}-Z|\le G
\quad dt\otimes d\P\text{-a.e. for all }n,
$$
and, in particular,
$$
Z^{\eps_n}\longrightarrow Z
\qquad dt\otimes d\P\text{-a.e. on }[t,T]\times\Omega.
$$
Since $(Y^\eps)_{\eps\in(0,1)}$ is uniformly bounded in $\Sc^\infty$, there exists $K>0$ such that
$$
|Y_s^\eps|\vee |Y_s|\le K,
\qquad dt\otimes d\P\text{-a.e. on }[t,T]\times\Omega,\ \forall\,\eps\in(0,1).
$$

We now estimate the coefficients. By \textbf{(A3)}, since $|Y_s^{\eps_n}|\vee |Y_s|\le K$, there exists a constant $C>0$ such that
$$
|\beta_s^{\eps_n}|+|\beta_s|
\le C\big(1+|Z_s^{\eps_n}|+|Z_s|\big),
\qquad dt\otimes d\P\text{-a.e.}
$$
Hence
$$
|\beta_s^{\eps_n}-\beta_s|^2
\le 3C\big(1+|Z_s^{\eps_n}|^2+|Z_s|^2\big).
$$
Using
$$
|Z_s^{\eps_n}|\le |Z_s|+G_s,
$$
we derive that
$$
|\beta_s^{\eps_n}-\beta_s|^2
\le 3C\big(1+|Z_s|^2+G_s^2\big),
\qquad dt\otimes d\P\text{-a.e.}
$$
Using the same domination idea above applied to $\partial_z g$, one can show using Lebesgue's dominated convergence
$$
\beta_s^{\eps_n}\longrightarrow \beta_s,
\qquad dt\otimes d\P\text{-a.e. on }[t,T]\times\Omega.
$$
Since $Z,G\in L^2(dt\otimes d\P)$, the right-hand side belongs to $L^1(dt\otimes d\P)$. Therefore, by Lebesgue's dominated convergence,
$$
\E\Big[\int_t^T |\beta_s^{\eps_n}-\beta_s|^2\,ds\Big]\longrightarrow 0.
$$
That is,
\begin{equation}\label{eq:B_subseq_L1}
B^{\eps_n}\longrightarrow 0
\qquad\text{in }L^1(\P).
\end{equation}

Similarly, by \textbf{(A4$^*$)}, for $\delta=1$ there exists $h_1\in L^1([0,T])$ such that
$$
|\partial_y g(s,\omega,y,z)|\le h_1(s)+|z|^2,
\qquad dt\otimes d\P\text{-a.e.},\ \forall (y,z)\in\R\times\R^d.
$$
Therefore,
\begin{align*}
|\gamma_s^{\eps_n}|
&=
\left|
\int_0^1
\partial_y g\big(s,
Y_s+\theta(Y_s^{\eps_n}-Y_s),
Z_s+\theta(Z_s^{\eps_n}-Z_s)\big)\,d\theta
\right| \\
&\le
\int_0^1 \Big(h_1(s)+\big|(1-\theta)Z_s+\theta Z_s^{\eps_n}\big|^2\Big)\,d\theta \\
&\le
h_1(s)+|Z_s|^2+|Z_s^{\eps_n}|^2,
\end{align*}
and also
$$
|\gamma_s|\le h_1(s)+|Z_s|^2.
$$
Hence
$$
|\gamma_s^{\eps_n}-\gamma_s|
\le 2h_1(s)+2|Z_s|^2+|Z_s^{\eps_n}|^2.
$$
Using once again
$$
|Z_s^{\eps_n}|\le |Z_s|+G_s,
$$
we obtain
$$
|\gamma_s^{\eps_n}-\gamma_s|
\le C\big(h_1(s)+|Z_s|^2+G_s^2\big),
\qquad dt\otimes d\P\text{-a.e.}
$$
The right-hand side belongs to $L^1(dt\otimes d\P)$, since $h_1\in L^1([0,T])$ and $Z,G\in L^2(dt\otimes d\P)$.
Using again the same domination idea above applied to $\partial_y g$, one can show using Lebesgue's dominated convergence 
$$
\gamma_s^{\eps_n}\longrightarrow \gamma_s,
\qquad dt\otimes d\P\text{-a.e. on }[t,T]\times\Omega.
$$
Applying again Lebesgue's dominated convergence yields
$$
\E\Big[\int_t^T |\gamma_s^{\eps_n}-\gamma_s|\,ds\Big]\longrightarrow 0.
$$
That is,
\begin{equation}\label{eq:A_subseq_L1}
A^{\eps_n}\longrightarrow 0
\qquad\text{in }L^1(\P).
\end{equation}
Since the above argument applies to every sequence $\eps_n\downarrow0$, the subsequence criterion yields
\eqref{eq:A_B_prob_conv}.

On the other hand, by the uniform BMO bound on $\int_t^\cdot Z_s^\eps\,dW_s$ and the energy inequalities in
Theorem~\ref{thm:bmo_toolkit}, for every integer $n_0\ge1$,
$$
\sup_{\eps\in(0,1)}
\E\Big[\Big(\int_t^T |Z_s^\eps|^2\,ds\Big)^{2n_0}\Big]<\infty.
$$
Using \textbf{(A3)} and \textbf{(A4$^*$)} as above, we deduce that, for every integer $n_0\ge1$,
$$
\sup_{\eps\in(0,1)}\E\big[(A^\eps)^{2n_0}\big]<\infty,
\qquad
\sup_{\eps\in(0,1)}\E\big[(B^\eps)^{2n_0}\big]<\infty.
$$
Hence the families $\big((A^\eps)^{n_0}\big)_{\eps\in(0,1)}$ and $\big((B^\eps)^{n_0}\big)_{\eps\in(0,1)}$ are uniformly integrable. Combining this with
\eqref{eq:A_B_prob_conv} and Vitali's theorem yields
\begin{equation}\label{eq:beta_gamma_Lp_clean}
\E\Big[\Big(\int_t^T|\gamma_s^\eps-\gamma_s|\,ds\Big)^{n_0}\Big]\longrightarrow0,
\qquad
\E\Big[\Big(\int_t^T|\beta_s^\eps-\beta_s|^2\,ds\Big)^{n_0}\Big]\longrightarrow0,
\qquad \eps\to0.
\end{equation}

\smallskip
\noindent\emph{\underline{Step 3:} Convergence.} We conclude by showing that 
$\Delta Y^\varepsilon_t \to U^X_t$ in $L^1(\P)$. The 
error $\delta U^\varepsilon := U^X - \Delta 
Y^\varepsilon$ solves a linear BSDE with source term 
$I^\varepsilon$ that vanishes as 
$\varepsilon \to 0$; the rate of convergence is 
controlled by the representation formula from 
Proposition~\ref{prop:linear_BSDE_general} and Vitali's theorem.
For $\eps\in(0,1)$, set
$$
\delta U_s^\eps:=U_s^X-\Delta Y_s^\eps,
\qquad
\delta V_s^\eps:=V_s^X-\Delta Z_s^\eps,
\qquad s\in[t,T].
$$
Subtracting \eqref{eq:lin_BSDE_eps_clean} from \eqref{eq:lin_limit_clean} gives
\begin{align*}
\delta U_s^\eps
&=\int_s^T\big(\gamma_r^\eps\delta U_r^\eps+\beta_r^\eps\delta V_r^\eps\big)\,dr
+\int_s^T\Big((\gamma_r-\gamma_r^\eps)U_r^X+(\beta_r-\beta_r^\eps)V_r^X\Big)\,dr
-\int_s^T \delta V_r^\eps\,dW_r.
\end{align*}
Define
$$
I_r^\eps:=(\gamma_r-\gamma_r^\eps)U_r^X+(\beta_r-\beta_r^\eps)V_r^X,
\qquad r\in[t,T].
$$
Then, for all $p\geq2$, we have by Cauchy-Schwarz
\begin{align} \label{ineq}
    \E\left[ \left ( \int_t^T \lvert I^{\eps}_s \lvert ds \right)^p \right] \leq &\E \Big[\sup_{t\le s\le T}|U^X_s|^{2p}\Big]^{\frac{1}{2}} \E\left[\Big(\int_t^T|\gamma_s^\eps-\gamma_s|\,ds\Big)^{2p}\right]^{\frac{1}{2}} + \\
    &\E^\P\left[ \left( \int_t^T |V_s^X|^{2} ds \right)^{p}\right]^{\frac{1}{2}} \E\left[\Big(\int_t^T|\beta_s^\eps-\beta_s|^2\,ds\Big)^{p}\right]^{\frac{1}{2}} < \infty.
\end{align}   
By Proposition~\ref{prop:linear_BSDE_general} applied to the coefficients
$(\phi,\gamma,\beta)=(I^\eps,\gamma^\eps,\beta^\eps)$ on $[t,T]$ (note that $\int\beta^\eps\,dW$ is BMO uniformly in $\eps$ by the uniform BMO bound on $\int Z^\eps\,dW$), we obtain the representation
$$
\delta U_t^\eps=\E^{\Q^\eps}\Big[\int_t^T \Gamma_{t,s}^\eps I_s^\eps\,ds\ \Big|\ \Fc_t\Big],
$$
where $\frac{d\Q^\eps}{d\P}=\mathcal E(\int_t^T \beta_s^\eps\,dW_s)$ and
$\Gamma_{t,s}^\eps=\exp(\int_t^s\gamma_r^\eps\,dr)$.

Using again Proposition~\ref{prop:linear_BSDE_general} (in particular the reverse H\"older control for $\Q^\eps$ and the exponential integrability of $\Gamma^\eps$ with constants independent of $\eps$), we can choose exponents
$q_0>1$, $p_0$ conjugate (independent of $\eps$), and $r_0,n_0>1$ with $\frac{1}{r_0}+\frac{1}{n_0}=\frac{1}{p_0}$, such that
\begin{equation}\label{eq:deltaU_est_clean}
\E\big[|\delta U_t^\eps|\big]
\le C\,\Big\|\int_t^T|I_s^\eps|\,ds\Big\|_{L^{n_0}(\P)},
\end{equation}
for a constant $C>0$ independent of $\eps\in(0,1)$.
Finally, using \eqref{ineq}, $(U^X,V^X)\in\Sc^{2n_0}\times\Hc^{2n_0}$, and \eqref{eq:beta_gamma_Lp_clean}, we deduce that
$$\big\|\int_t^T|I_s^\eps|\,ds\big\|_{L^{n_0}(\P)}\to0,$$ therefore by \eqref{eq:deltaU_est_clean},
$$
\E\big[|U_t^X-\Delta Y_t^\eps|\big]=\E\big[|\delta U_t^\eps|\big]\longrightarrow0,
\qquad \eps\to0.
$$
Recalling that $\Delta Y_t^\eps=\frac{1}{\eps}\Ec^g_{t,T}[\eps X]$, we obtain
$$
\lim_{\eps\to0^+}\frac{1}{\eps}\Big(\Ec^g_{t,T}[\eps X]-\Ec^g_{t,T}[0]\Big)=U_t^X,
\qquad \text{in }L^1(\P).
$$
This concludes the proof.
\end{proof}
We now provide a representation result for quadratic generators developed in \cite{zheng2018representation} and a proposition on the properties of quadratic $g$-expectations with a deterministic generator adapted from \cite{Ma02062010}.

 Let $C>0$ be a given constant and define the following stopping time
$$ \tau^{t}_{C} := \inf \left \{ s \in [0,T-t] , ~~ \lvert W_{t + s} - W_t \lvert > C \right \} \wedge (T-t).$$
We now recall the following result which will be used in the proof of Theorem \ref{thm:main}.
\begin{Theorem}[\cite{zheng2018representation}] \label{thm:representation}
    Assume \textbf{(A1)}–\textbf{(A2)}, then for each $(y,z) \in \R \times \R^d$ and almost every $t \in [0,T)$  
    $$ g(t, \cdot, y,  z) = \lim_{\eps \to 0^+} \frac{1}{\eps} \left( \Ec^g_{t,t+\eps  \wedge {\tau}^t_C}\left [ y + z(W_{t + \eps \wedge \tau^t_C} - W_t) \right]  - y \right), ~~ \P\text{-a.s.}$$
 where $C>0$ is an arbitrary positive constant.
\end{Theorem}

\section{Conditional atomlessness} \label{app:B}
We present here a technical result used in the proof of Theorem \ref{thm:general}. For the sake of completeness, we first give the definition of \emph{conditional atomlessness}.
\begin{Definition}[Conditionally atomless {\cite[Definition~1]{Delbaen2020}}]\label{def:condatomless}
Let $(\Omega,\Fc,\P)$ be a probability space with a sub-$\sigma$-algebra $\Gc \subset \Fc$.
We say that $\Fc$ is \emph{atomless conditionally to} $\Gc$ if for every $A\in\Fc$
there exists a set $B\subset A$, $B\in\Fc$, such that
$$
0<\E[\1_B\mid \Gc]<\E[\1_A\mid \Gc]
\quad\text{on the set }\{\E[\1_A\mid \Gc]>0\}.
$$
\end{Definition}
\begin{Remark}
An equivalent formulation is: for every $A\in\Fc$ with $\P(A)>0$ there exists $B\subset A$
such that
$$
\P\big(0<\E[\1_B\mid\Gc]<\E[\1_A\mid\Gc]\big)>0.
$$
See \cite[Theorem~1]{Delbaen2020}. This is the conditional version of the result in \cite{sierpinski1922fonctions}. Furthermore, conditional atomlessness to $\Fc_0$, in our Wiener filtered space, is equivalent to the classical definition of atomlessness. 
\end{Remark}
We now introduce a result on the conditional atomlessness when the probability space is filtered with Brownian filtration. 

\begin{Lemma}\label{lemma:condprob}
Fix $0\le t<T$. Then $\Fc_T$ is atomless conditionally to $\Fc_t$.
In particular, if $Z\in L^{1}(\Fc_T)$ is not $\Fc_t$-measurable and if we set
$m_t:=\E[Z\mid\Fc_t]$, then there exist $A,B\in\Fc_T$ such that
$$
\P(A\mid\Fc_t)=\P(B\mid\Fc_t)>0,\qquad \P\text{-a.s. on a set of positive probability,}
$$
and
$$
Z\,\1_A \le m_t\,\1_A,\qquad Z\,\1_B > m_t\,\1_B.
$$
\end{Lemma}

\begin{proof}
\emph{\underline{Step 1} (Conditional atomlessness).}
Define
$$
U:=\Phi\!\left(\frac{W_T-W_t}{\sqrt{T-t}}\right).
$$
Then $U$ is $\Fc_T$-measurable, uniformly distributed on $[0,1]$, and independent of $\Fc_t$.
Hence $\sigma(U)\subset \Fc_T$ is atomless and independent of $\Fc_t$.
By \cite[Theorem~2]{Delbaen2020}, $\Fc_T$ is atomless conditionally to $\Fc_t$.

\noindent \emph{\underline{Step 2} (Separation with equal conditional mass).}
Let $m_t:=\E[Z\mid\Fc_t]$ and set
$$
E^-:=\{Z<m_t\},\qquad E^+:=\{Z>m_t\}.
$$
Since $Z$ is not $\Fc_t$-measurable, we have $\P(E^-|\Fc_t)>0$ and $\P(E^+|\Fc_t)>0$ on a set
of positive probability denoted by $D$. Define the $\Fc_t$-measurable random variable
$$
p_t:=\frac12\big(\P(E^-\mid\Fc_t)\wedge \P(E^+\mid\Fc_t)\big),
\qquad
D:=\{\P(E^-|\Fc_t)>0\}\cap\{\P(E^+|\Fc_t)>0\}.
$$
Set
$$
h^-:=\1_D\,\frac{p_t}{\P(E^-\mid\Fc_t)},\qquad
h^+:=\1_D\,\frac{p_t}{\P(E^+\mid\Fc_t)}.
$$
By \cite[Lemma~2]{Delbaen2020}, applied with $C=E^-$ (resp. $C=E^+$) and
$h=h^-$ (resp. $h=h^+$), there exist $A\subset E^-$ and $B\subset E^+$ in $\Fc_T$ such that
$$
\E[\1_A\mid\Fc_t]=h^-\,\E[\1_{E^-}\mid\Fc_t]=p_t,\qquad
\E[\1_B\mid\Fc_t]=h^+\,\E[\1_{E^+}\mid\Fc_t]=p_t.
$$
Thus $\P(A|\Fc_t)=\P(B|\Fc_t)=p_t>0$ on $D$. Moreover, since $A\subset\{Z<m_t\}$ and
$B\subset\{Z>m_t\}$, we have $Z\1_A\le m_t\1_A$ and $Z\1_B>m_t\1_B$.
\end{proof}
\section{Invariance properties} \label{app:C}
The following lemmas are key ingredients for the proof of Theorem \ref{thm:main}.

\begin{Lemma} \label{lemma: lawBM}
Let $t \in [0,T]$, $s \in [0,t)$, $\varepsilon \in (0, T-t]$, $\lambda \in (0,1)$ and $O$ be an orthogonal matrix in $\R^{d\times d}$. Then
$$ \lambda  O\left( W_{t +  \eps \wedge \tau^t_C} - W_t \right) ~ \stackrel{d}{=} ~ W_{s + \lambda^2 \eps \wedge \tau^s_{\lambda C}} - W_s .$$
\end{Lemma}
\begin{proof}

We organise the proof into two steps.\\

\underline{\textit{Step 1:}} We first show that 
$$\lambda  O\left( W_{t +  \eps \wedge \tau^t_C} - W_t \right) \stackrel{d}{=} \lambda  O\left( W_{  \eps \wedge \tau^0_C}\right). $$
Observe that, by definition, 
$$\tau_C^t=\tau\left((W_{s+t}-W_t)_{0 \leq s \leq T-t}\right),$$
for a measurable function $\tau$ defined on the canonical space $C^0([0,T], \R^d)$ endowed with the Borel $\sigma$-algebra and that
$$O\left( W_{t +  \eps \wedge \tau^t_C} - W_t \right)=\Psi_{\eps} \left((W_{s+t}-W_t)_{0 \leq s \leq T-t}\right),$$
with $\Psi_{\eps}$ a measurable function that is also defined on the canonical space $C^0([0,T], \R^d)$ endowed with the Borel $\sigma$-algebra.
Similarly, we have
$$O\left( W_{\eps \wedge \tau^0_C}  \right)=\Psi_{\eps}\left((W_s)_{0 \leq s \leq T-t}\right).$$
Denote $W_s^t:=W_{s+t}-W_t$, for $0 \leq s \leq T-t$. By properties of the Brownian motion, we have
$W^t \stackrel{d}{=} W^0$,
from which the conclusion follows.

\underline{\textit{Step 2:}}  Let us now show that $$ \lambda O\left(W_{\varepsilon \wedge \tau_C^0}\right) \stackrel{d}{=} W_{\lambda^2 \varepsilon \wedge \tau_{\lambda C}^0}.$$ 
Using a Brownian scaling and rotation-invariance, we obtain that for all $s \in [0,T]$ 
$$  O W_s \stackrel{d}{=} \Tilde W_s := \lambda W_{\frac{s}{\lambda^2 }},$$
where $\Tilde W$ is a Brownian motion.
Furthermore, observe that 
\begin{align*}
    \lambda^2 \tau^0_C &= \inf \left \{ \lambda^2 s \in [0,T] , ~~ \lvert W_{s}\lvert >  C \right \} \wedge (T-t) \\
    &= \inf \left \{u \in [0,T] , ~~ \lvert \Tilde W_{u}\lvert >  \lambda C \right \} \wedge (T-t) \\
    &= \tau^{\lambda} \left( \left( \Tilde W_{s} \right)_{0 \leq s \leq T} \right),
\end{align*}
for some measurable function $\tau^{\lambda}$ defined on the canonical space $C^0([0,T], \R^d)$ endowed with the Borel $\sigma$-algebra.
Finally, we obtain that 
$$\lambda O\left(W_{\varepsilon \wedge \tau_C^0}\right) =  \Tilde W_{\lambda^2 \varepsilon \wedge \lambda^2 \tau_C^0} = \Psi_{\eps}^{\lambda} \left(\left( \Tilde W_{s} \right)_{0 \leq s \leq T} \right),$$
for some measurable function $\Psi_{\eps}^{\lambda}$ defined on the canonical space $C^0([0,T], \R^d)$ endowed with the Borel $\sigma$-algebra.
On the other hand, it immediately follows that
$$ \tau^{0}_{\lambda C} = \tau^{\lambda}\left( \left( W_{s} \right)_{0 \leq s \leq T} \right).$$
Thus
$$W_{\lambda^2 \varepsilon \wedge \tau_{\lambda C}^0} = \Psi_{\eps}^{\lambda} \left(\left( W_{s} \right)_{0 \leq s \leq T} \right).$$
This concludes the proof.
\end{proof}

\begin{Lemma} \label{lemma:rotinvariance}
Let $f$ be a real-valued map on $\R \times \R^d$ verifying, for any arbitrary $\lambda \in [0,1]$ and $O$ orthogonal matrix in $\R^{d \times d}$,
$$ f(y, \lambda Oz) = \lambda^2 f(y,z).$$
Then, it verifies 
$$ f(y,z) = f(y,e_1) \lvert z \lvert^2,$$
where $(e_1, \cdots, e_d)$ forms an orthonormal basis of $\R^d$.
\end{Lemma}
\begin{proof}We first note that taking $\lambda=0$ gives $f(y,0)=0$.

First, we assume $d = 1$. Then, for all $\lambda \in [0,1]$, by taking $z=1$ we have
$$ f(y,\lambda) = f(y,1)\,\lambda^2.$$
Let $r\in\R\setminus\{0\}$ and set $O=\frac{r}{|r|}\in\{\pm1\}$.

\textbf{Case } $|r|\ge 1$. Let $\lambda=\frac1{|r|}\in(0,1]$. Then $1=\lambda Or$, so
$$
f(y,1)=\lambda^2 f(y,r)\quad\Rightarrow\quad
f(y,r)=\lambda^{-2}f(y,1)=|r|^2 f(y,1)=r^2 f(y,1).
$$

\textbf{Case } $0<|r|\le 1$. Let $\lambda=|r|\in(0,1]$. Then $r=\lambda O$, so
$$
f(y,r)=f(y,\lambda O)=\lambda^2 f(y,1)=|r|^2 f(y,1)=r^2 f(y,1).
$$
Combined with $f(y,0)=0$, this proves the result for all $r\in\R$.

Let us now prove the result for an arbitrary dimension $d>1$. For a given $z \in \R^d$, the case $z=0$ follows from the first line. Assume $z\neq 0$ and set $v := \frac{z}{\lvert z \rvert}$ to be a unit vector. Choose an orthogonal matrix $O$ verifying $Ov=e_1$ (for instance, by a Householder reflection). Define
$$ g(y,r):=f(y,r e_1), \qquad r\in\R.$$
Then, for any $\lambda\in[0,1]$,
$$ g\bigl(y,\lambda \lvert z \rvert\bigr)=f\bigl(y,\lambda \lvert z \rvert e_1\bigr)=f(y,\lambda Oz)=\lambda^2 f(y,z). $$
By the one-dimensional case applied to $g$ (in the variable $r$), we have $g(y,r)=g(y,1)\,r^2$ for all $r\in\R$. Therefore,
$$ f(y,z)=g\bigl(y,\lvert z\rvert\bigr)=g(y,1)\,\lvert z\rvert^2=f(y,e_1)\,\lvert z\rvert^2. $$
    
\end{proof}
\section{Study of the change of variable} \label{app:D}
We present here the results for the change of variable used in the case where $g(t, \omega, y,0)$ is deterministic.

\begin{Lemma}\label{Lemma:1stPDE}
Assume that $g$ satisfies \textbf{(A1)}–\textbf{(A4)} and \emph{Assumption} \textbf{(H)}. 
Then, the following first-order linear PDE, which holds for all $(t,y) \in [0,T] \times \R$,
\begin{equation} 
\partial_t v(t,y) - \partial_y v(t,y)\,h(t,y) = 0,\qquad v(0,y)=y,
\end{equation}
admits a unique solution $v$ in $C^{1,2}([0,T] \times \R)$ which satisfies for all $(t,y) \in [0,T]\times \R$ $$ \partial_{ty} v(t,y) = \partial_{yt} v(t,y) ~~ \text{ and } ~~ \partial_{tyy} v(t,y) = \partial_{yty} v(t,y).$$
Moreover, for every $t\in[0,T]$, the map $y\mapsto v(t,y)$ is a $C^1$-diffeomorphism of $\R$. In addition, there exist deterministic constants $m_1,M_1,M_2,M_3>0$ such that, for all $(t,u) \in [0,T] \times \R$, 
$$
0<m_1\ \le\ \partial_y v\bigl(t,v^{-1}(t,u)\bigr)\ \le\ M_1,\qquad
\bigl|\partial_{yy}v\bigl(t,v^{-1}(t,u)\bigr)\bigr|\ \le\ M_2,
$$
and
$$
\bigl|\partial_{yyy}v\bigl(t,v^{-1}(t,u)\bigr)\bigr|\ \le\ M_3.
$$
\end{Lemma}
\begin{proof}
\emph{\underline{Step 1:} Characteristics method.}
Consider the ODE
$$
\dot{y}(t)=-h\bigl(t,y(t)\bigr),\qquad y(0)=y_0\in\R.
$$
By \textbf{(H0)} and \textbf{(H1)}, $h$ is a continuous function and is locally Lipschitz in $y$.
Hence there exists a unique local solution. Moreover, using the linear growth bound
$\lvert h(t,y)\rvert\le \kappa(1+\lvert y\rvert)$ (from \textbf{(A2)} with $z=0$), this solution is global and $C^1$ on $[0,T]$,
see \cite[Theorem~2.5, Theorem~2.17]{Teschl2012}.
Denote it by $\phi(t;y_0)$ and set $\Phi_t(y_0):=\phi(t;y_0)$.

\smallskip
\noindent \emph{\underline{Step 2:} Monotonicity and bounds for $\partial_y v$.}
 Set $\xi(t;y_0):=\partial_{y_0}\Phi_t(y_0)$, which is well-defined and continuous in $(t,y_0)$
under \textbf{(H0)} (see \cite[Theorem~7.2]{coddington1956theory}). Then, by the differentiability of $h$ w.r.t. $y$,
$$
\dot\xi(t;y_0)=-\partial_y h\bigl(t,\Phi_t(y_0)\bigr)\,\xi(t;y_0),\qquad \xi(0;y_0)=1.
$$
By \textbf{(H1)} and Gr\"onwall,
$$
e^{-\int_0^t H(r)\,dr}\ \le\ \xi(t;y_0)\ \le\ e^{\int_0^t H(r)\,dr},\qquad t\in[0,T],\ y_0\in\R.
$$
In particular, for every $t\in[0,T]$, the map $\Phi_t$ is strictly increasing and bi-Lipschitz:
$$
e^{-\int_0^t H(r)\,dr}\,|y_0-y_1|\ \le\ |\Phi_t(y_0)-\Phi_t(y_1)|\ \le\ e^{\int_0^t H(r)\,dr}\,|y_0-y_1|.
$$
Hence $|\Phi_t(y_0)|\to\infty$ as $|y_0|\to\infty$, so $\Phi_t:\R\to\R$ is onto. Together with $\xi(t;y_0)>0$,
the inverse function theorem yields that $\Phi_t$ is a $C^1$-diffeomorphism. Define 
\begin{equation}\label{eq:vinv-strong}
v(t,y)=\Phi_t^{-1}(y),\qquad (t,y)\in[0,T]\times\R.
\end{equation}
Then, by construction,
\begin{equation}\label{eq:prePDE}
\Phi_t \left ( v(t,y) \right)= y,\qquad (t,y)\in[0,T]\times\R.
\end{equation}
Therefore, $v$ is the solution of the PDE \eqref{eq:1stPDE} since by differentiating $t \in [0,T]$ equation \eqref{eq:prePDE}, we have
\begin{equation*}
    0 = \partial_t \Phi_t(v(t,y)) + \partial_y \Phi_t(v(t,y)) \partial_t v(t,y) = -h(t,y) + \partial_y \Phi_t(v(t,y)) \partial_t v(t,y).
\end{equation*}
Rearranging the terms yields
\begin{equation*}
    \partial_t v(t,y)  = \Frac{1}{\partial_y \Phi_t(v(t,y))} h(t,y)  = \partial_y v(t,y) h(t,y).
\end{equation*}
Moreover, from \eqref{eq:vinv-strong} and the inverse function theorem,
$$
\partial_y v\bigl(t,y\bigr)=\frac{1}{\xi\bigl(t,v(t,y)\bigr)}.
$$
Therefore, for all $(t,u) \in [0,T] \times \R$,
$$
m_1:=e^{-\int_0^T H(r) dr}\ \le\ \partial_y v\bigl(t,v^{-1}(t,u)\bigr)\ \le\ e^{\int_0^T H(r) dr}=:M_1.
$$

\smallskip
\emph{\underline{Step 3:} Bounds for $\partial_{yy}v$ and $\partial_{yyy}v$.}
Let $\eta(t;y_0):=\partial_{y_0y_0}^2\Phi_t(y_0)$ and $\theta(t;y_0):=\partial_{y_0y_0y_0}^3\Phi_t(y_0)$. Differentiating the variational equation yields
$$
\dot\eta=-\partial_{yy}^2 h\bigl(t,\Phi_t(y_0)\bigr)\,\xi^2-\partial_y h\bigl(t,\Phi_t(y_0)\bigr)\,\eta,\qquad \eta(0;y_0)=0,
$$
and an analogous linear ODE for $\theta$ with coefficients $\partial_y h,\partial_{yy}^2 h,\partial_{yyy}^3 h$ and $(\xi,\eta)$. Using \textbf{(H1)} and \textbf{(H2)} together with the bounds for $\xi$, variation of constants and Gr\"onwall give
$$
|\eta(t;y_0)|\ \le\ C_2,\qquad |\theta(t;y_0)|\ \le\ C_3,\qquad t\in[0,T],\ y_0\in\R,
$$
for deterministic constants $C_2,C_3>0$ depending only on $T,H,B$.

By the inverse-map identities,
$$
\partial_{yy}v\bigl(t,y\bigr)=-\frac{\eta\bigl(t,v(t,y)\bigr)}{\xi\bigl(t,v(t,y)\bigr)^3},\qquad
\partial_{yyy}v\bigl(t,y\bigr)=\frac{3\,\eta(t,v)^2-\xi(t,v)\,\theta(t,v)}{\xi(t,v)^5},
$$
and the bounds above on $(\xi,\eta,\theta)$, we obtain global constants $M_2,M_3>0$ such that, for all $(t,u) \in [0,T] \times \R$,
$$
\bigl|\partial_{yy}v\bigl(t,v^{-1}(t,u)\bigr)\bigr|\ \le\ M_2,\qquad
\bigl|\partial_{yyy}v\bigl(t,v^{-1}(t,u)\bigr)\bigr|\ \le\ M_3.
$$   

\smallskip
\emph{\underline{Step 4:} Regularity of $v$.} The continuity of $(t,y_0)\mapsto \Phi_t(y_0)$ and of $(t,y_0)\mapsto \xi(t;y_0)$ (hence of $v$ and $\partial_y v$)
under \textbf{(H0)}--\textbf{(H1)} follows from the continuous dependence theory for ODEs,
see again \cite{coddington1956theory, Teschl2012}.
Then \eqref{eq:1stPDE} gives $\partial_t v=\partial_y v\,h$, and \textbf{(H0)} yields continuity of $\partial_t v$.
The bounds and formulas in Step~3 provide existence and continuity of $\partial_{yy}v$ and $\partial_{yyy}v$.
Finally, differentiating \eqref{eq:1stPDE} in $y$ shows that $u_1:=\partial_y v$ solves
$\partial_t u_1-\partial_y u_1\,h-\partial_y v\,\partial_y h=0$, and similarly $u_2:=\partial_{yy}v$
solves the corresponding linear first-order equation, which yields
$\partial_{ty}v=\partial_{yt}v$ and $\partial_{tyy}v=\partial_{yty}v$.
\end{proof}
\medskip
\noindent In the following proposition, we study the properties of the new generator $\Tilde g$.
\begin{Proposition}\label{prop:Koby-tildeg}
Assume that $g$ satisfies \textbf{(A1)}–\textbf{(A4)} and \textbf{(H)}, then the generator $\Tilde g$ satisfies \textbf{(A1)}–\textbf{(A3)}. Moreover, the associated quadratic $g$-expectation $\Ec^{\Tilde g}$ is well-defined and verifies the  properties in Proposition \ref{prop:Eg}. Finally, for bounded $X$ and $0\le \sigma \le \tau \le T$, the associated $g$-expectation is linked by
$$
\Ec^{\Tilde g}_{\sigma,\tau}[X]\;=\; v\Big(\sigma,\ \Ec^{g}_{\sigma,\tau}\big[v^{-1}(\tau,X)\big]\Big), ~~ \P\text{-a.s.}
$$

\end{Proposition}

\begin{proof} The proof is organized in two steps:

\emph{\underline{Step 1:}}
First, we begin by proving that the generator $\tilde{g}$ verifies the assumptions \textbf{(A1)}–\textbf{(A3)}.\\By Lemma \ref{Lemma:1stPDE} there exist deterministic constants $m_1,M_1,M_2,M_3>0$ such that for all $(t,y)\in[0,T]\times\R$
$$
0<m_1\le \partial_y v(t,y)\le M_1,\qquad
|\partial_{yy}v(t,y)|\le M_2,\qquad
|\partial_{yyy}v(t,y)|\le M_3.
$$
Hence $v(t,\cdot)$ and $v^{-1}(t,\cdot)$ are globally Lipschitz with
$$
|v(t,y)-v(t,y')|\le M_1|y-y'|,\qquad
|v^{-1}(t,u)-v^{-1}(t,u')|\le \frac{1}{m_1}\,|u-u'|.
$$
Let $\kappa>0$ be the linear-growth constant from \textbf{(A2)} at $z=0$, so that $|h(t,y)|\le \kappa(1+|y|)$.
Solving the characteristic ODE backward from time $t$ yields the uniform bound
$$
\sup_{t\in[0,T]}\big|v^{-1}(t,0)\big|\ \le\ e^{\kappa T}-1=:K_0.
$$
Therefore, for all $(t,u) \in [0,T] \times \R$,
\begin{equation}\label{eq:y-vs-u}
\big|v^{-1}(t,u)\big|
\le \big|v^{-1}(t,0)\big|+\frac{1}{m_1}|u|
\le K_0+\frac{1}{m_1}|u|.
\end{equation}

 Now, we prove here that assumption \textbf{(A1)} is verified.
Measurability follows from the measurability of $g$ and the continuity with respect to time of the deterministic function $v$; continuity in $(u,\Tilde z)$ is inherited from the continuity of $(y,z)\mapsto g(t,y,z)$ and the smooth change of variables $(u,\Tilde z)\mapsto (y,z)$.

Recall the specific form of the generator $g$ in the standing assumption :
$$
 g(t,\omega,y,z)=g_1(t,\omega,y,z)y+g_2(t,\omega,y,z), \quad ~~ \forall (t,\omega,y,z) \in
 [0,T] \times \Omega  \times \R \times \R^d,
$$
and we can obtain the following identity by using the fact that $v$ is the solution of PDE \eqref{eq:1stPDE}
\begin{align*}
\Tilde g(t,u,\Tilde z) & =  \partial_y v(t,y)\,[g(t,y,z)-g(t,y,0)]-\frac12\,\partial_{yy}v(t,y)\,|z|^2\quad \\
&= \partial_y v(t,y)\,[g_1(t,y,z)-g_1(t,y,0)]y + \partial_y v(t,y)\,[g_2(t,y,z)-g_2(t,y,0)]\\
&-\frac12\,\partial_{yy}v(t,y)\,|z|^2,
\end{align*}
with $z:=\frac{\Tilde z}{\partial_y v(t,y)}$.

\noindent Set 
\begin{align*}
    \Tilde g_1(t,u, \Tilde z) &:= 
    \begin{cases}
        \partial_y v(t,y)\,[g_1(t,y,z)-g_1(t,y,0)]\tfrac{v^{-1}(t,u) - v^{-1}(t,0)}{u}, \quad \text{ if } u \neq 0,\\
        0, \quad \text{else.}
    \end{cases} \\
    \Tilde g_2(t,u, \Tilde z) &:= \partial_y v(t,y)\,[g_1(t,y,z)-g_1(t,y,0)]v^{-1}(t,0) +\\ &\partial_y v(t,y)\,[g_2(t,y,z)-g_2(t,y,0)] -\frac12\,\partial_{yy}v(t,y)\,|z|^2.
\end{align*}
Then, $$\Tilde g(t,u,\Tilde z) =  \Tilde g_1(t,u, \Tilde z) u + \Tilde g_2(t,u, \Tilde z),$$
and using the estimates from \textbf{(A2)} and the global bounds on $v$, we have
\begin{align*}
    |\Tilde g_1(t,u,\Tilde z)|
&\le 2\frac{M_1}{m_1} \kappa ,\\
|\Tilde g_2(t,u,\Tilde z)| &\le 2 M_1 \kappa  \left(K_0 + 1\right) + \left ( \frac{ M_1 \ell(\lvert y \lvert )}{m_1^2}+\frac{M_2}{2m_1^2}\right) \lvert \Tilde z|^2.
\end{align*}
Define the increasing continuous function
$$
\tilde\ell(r):=  \left( \frac{M_1}{m_1^2} \ell\Big(K_0+\frac{1}{m_1}\,r\Big) + \tfrac{M_2}{2 m_1^2} \right),\qquad r\ge0,
$$
and observe that
$$\ell(\lvert y \lvert) = \ell(\lvert v^{-1}(t,u) \lvert) \leq \ell\Big(K_0+\frac{1}{m_1}\,\lvert u\lvert \Big). $$
Then 
$$
|\Tilde g_2(t,u,\Tilde z)|
\le 2 M_1 \kappa  \left(K_0 + 1\right) + \Tilde \ell \left( \vert u \lvert \right)\,|\Tilde z|^2.
$$
For the assumption \textbf{(A3)}, we take $z=\Tilde z/\partial_y v(t,y)$ and differentiate $\Tilde g$ with respect to $\Tilde z$ 
$$
\partial_{\Tilde z} \Tilde g(t,u,\Tilde z)
=\partial_z g(t,y,z) - \frac{\partial_{yy}v(t,y)}{\big(\partial_y v(t,y)\big)^2} \Tilde z.
$$
Therefore,
\begin{align*}
   \big|\partial_{\Tilde z} \Tilde g(t,u,\Tilde z)\big|
&\le \ell(|y|)\big(1+|z|\big)+\frac{M_2}{m_1^2}\,|\Tilde z| \\
&\le \tfrac{1}{M_1}\tilde\ell(|u|)\Big(1+\frac{1}{m_1}|\Tilde z|\Big)+\frac{M_2}{m_1^2}|\Tilde z| \\
&\le \Tilde \ell^{\star}(|u|)\,\big(1+|\Tilde z|\big), 
\end{align*}
for an increasing continuous function $\Tilde \ell^{\star}$ related to $\Tilde \ell$ such that $\Tilde \ell \le \Tilde \ell^{\star}$.

\emph{\underline{Step 2:} Transfer of well-posedness (no need for \textbf{(A4)} on $\Tilde g$).}
Let $(Y,Z)$ be the unique solution of the original BSDE with data $(g,\xi)$. Define
$$
\Tilde Y_t:=v(t,Y_t),\qquad \Tilde Z_t:=\partial_y v(t,Y_t)\,Z_t.
$$
By It\^o’s formula and $\partial_t v - \partial_y v\,h = 0$, the pair $(\Tilde Y,\Tilde Z)$ solves the BSDE with data $(\Tilde g,\Tilde\xi:=v(T,\xi))$. Conversely, any solution $(\widehat Y,\widehat Z)$ of the transformed BSDE yields a solution $$(v^{-1}(t,\widehat Y_t),\,\widehat Z_t/\partial_y v(t,v^{-1}(t,\widehat Y_t)))$$ of the original equation; uniqueness for the original problem then implies $(\widehat Y,\widehat Z)=(\Tilde Y,\Tilde Z)$. Hence uniqueness is inherited by transfer, and checking \textbf{(A4)} (and \textbf{(A3)}) for $\Tilde g$ is not required. 
Consequently, for bounded $X$ and $0\le \sigma\le \tau \le T$ $\P$-a.s., the associated $g$--expectations are linked by
$$
\Ec^{\Tilde g}_{\sigma,\tau}[X]\;=\; v\Big(\sigma,\ \Ec^{g}_{\sigma,\tau}\big[v^{-1}(\tau,X)\big]\Big), ~~ \P\text{-a.s.}
$$
It follows immediately that the properties of the quadratic $g$--expectation are preserved under the transform. In particular, the representation limit in Theorem~\ref{thm:representation} and the properties in Proposition~\ref{prop:Eg} remain valid for $\Tilde g$.
\end{proof}

\end{appendix}
\section*{Funding}

This work was supported by the Agence Nationale de la Recherche through the project
DREAMeS (ANR-21-CE46-0002) and by the France 2030 programme through the MIRTE
project (ANR-23-EXMA-0011).

 The research of the first, third, and fourth authors was also supported by the Chair “Risques Émergents en Assurance” (RE2A), the Chair “Impact de la Transition Climatique en Assurance” (ITCA), and GRESC “Gestion des risques ESG pour le cr\'edit” , under the aegis of the Fondation du Risque, in partnership with the Risk and Insurance Institute of Le Mans, MMA-Covéa, Groupama, and Natixis, respectively.

The funding bodies had no role in the design of the study, the analysis, the writing
of the manuscript, or the decision to submit the article.
\bibliographystyle{abbrv}
\bibliography{ref}

\end{document}